\documentclass[review]{elsarticle}
 
\usepackage{lineno}
\modulolinenumbers[5]
\usepackage{graphicx} 
\usepackage[font=small,labelfont=bf]{caption}
\usepackage{subcaption}

\usepackage[top=2.5cm,bottom=2.5cm,left=2.0cm,right=2.0cm]{geometry}

\usepackage{graphicx}
\usepackage{caption}
\usepackage{multicol}
\usepackage{amssymb,amsfonts,amsthm,amsmath}    
\usepackage{color}                        
\usepackage{newlfont}
\usepackage{algorithm}
\usepackage{hyperref}
\usepackage{enumerate}
\usepackage{nomencl}

\makeglossary

\theoremstyle{plain}
\newtheorem{thm}{Theorem}[section]

\newtheorem{lemma}[thm]{Lemma}
\newtheorem{corollary}[thm]{Corollary}

\newtheorem{definition}[thm]{Definition}
\newtheorem{remark}[thm]{Remark}
\newtheorem{example}[thm]{Example}

\numberwithin{equation}{section}
\numberwithin{algorithm}{section}

\def\bv{\mathbf{v}}
\def\bw{\textbf{w}}
\def\bx{\textbf{x}}
\def\bn{\textbf{n}}

\journal{}
\begin{document}

\begin{frontmatter}

\title{Optimal Control of  Convection-Cooling  and Numerical Implementation}

\author{Cuiyu He\fnref{label1}}
\ead{cuiyu.he@uga.edu}
\author{Weiwei Hu\corref{cor1}\fnref{label1}}
\ead{Weiwei.Hu@uga.edu}
\author{Lin Mu\fnref{label1}}
\ead{linmu@uga.edu}
\fntext[label1]{Department of Mathematics, University of Georgia, Athens, GA 30602}
\cortext[cor1]{Corresponding author}

\begin{abstract}
This paper is concerned with   the problem of enhancing
convection-cooling  via  active control of the incompressible 
 velocity field, described by a stationary  diffusion-convection  model.  This essentially leads to a bilinear optimal control problem. 
A  rigorous proof of the existence of an optimal control is presented and  the first order optimality  conditions are derived  for solving  the control  using a variational inequality.  Moreover,  the second order sufficient conditions  are established  to characterize  the  local minimizer.  Finally,  numerical experiments are conducted  utilizing  finite elements methods together with  nonlinear iterative schemes, to  demonstrate   and validate the effectiveness  of our control design.
\end{abstract}

\begin{keyword}
convection-cooling, bilinear control, optimality conditions, variational inequality, numerical experiments
\end{keyword}

\end{frontmatter}




\section{Introduction}

Convection-cooling is the mechanism where heat is transferred  from the hot object into the surrounding air or liquid. 
There are several factors {determining} the effectiveness of cooling, including 
 temperature difference between the surrounding and the hot object, 
  viscosity of the fluid (air or liquid),
  and   ability of the fluid to move in response to the density difference, etc.
There are two types of convectional cooling, namely the natural convection cooling and the forced air convection cooling (cf.~\cite{bergman2011introduction,  bejan2013convection, kreith2012principles}).
In the natural cooling, the air surrounding the object transfers the heat away from the object and does not use any fans or blowers. 
In contrast, forced air convection cooling  is used in designs where the enclosures or environment do not offer an effective natural cooling performance 
and areas where natural cooling is not effective.
The forced air convection cooling is the most effective cooling method 
in many industrial applications. It can be designed to provide the required cooling performance while increasing  the efficiency of the related components.

The current work utilizes  an optimal control approach for the forced air  convection-cooling. To be more precise, consider   a stationary  diffusion-convection  model for a cooling application  in an open bounded and connected domain $\Omega\subset \mathbb{R}^d, d=2,3$,  with a Lipschitz   boundary  $\Gamma$. The  velocity field is assumed to be divergence-free. The system of equations reads
  \begin{align}
	-\kappa \Delta T+\bv \cdot \nabla T&=f \quad \text{in}\quad \Omega\label{sta_T}\\
	\nabla \cdot \bv&=0, \label{sta_v}
\end{align}
with Dirichlet boundary condition for temperature and no-slip boundary condition for velocity
 \begin{align}
 T|_{\Gamma}&=0,\quad   \bv|_{\Gamma}=0, \label{sta_BC}
\end{align}	
	where $T$ is the  temperature,  $\kappa>0$  is  the thermal diffusivity,  $\bv$ is the velocity, and $f\in L^\infty(\Omega)$ is the external heat source distribution. 
The Dirichlet  boundary condition is  corresponding to a given fixed surface temperature, for example, when the surface is in contact with a melting solid or a boiling liquid. Although Neumann type of boundary conditions  are often used  in the diffusion-convection problems for describing heat flux at the boundaries,  
the Dirichlet boundary condition is also  commonly employed   in the study of natural convection and heat transfer in enclosures, which may be 
 simultaneously heated from below and cooled from above (cf.~\cite{calcagni2005natural, corcione2003effects, dalal2006natural, sezai2000natural, temam2012infinite}).
 Linear controls, either internal (distributed) or boundary controls,  of   the temperature and the corresponding numerical schemes   have been well studied for diffusion-convection equations (cf.~\cite{burns2014numerical,burns2015full,  chen2019class, chen2019hdg,chen2018hdg,   gong2018new, hu2018superconvergent, quarteroni2005optimal, zhang2017optimal}).   The objective of this work is aimed at enhancing  convection-cooling via active control of the 
flow velocity. For example,  in high power applications, a cooling fan is used  to blow and direct air towards the electronic components with or without heat sinks.  Most power supply units have built-in fans that provide the required forced-air convectional cooling.  Mathematically, our control design gives rise  to a bilinear optimal control problem.

 Optimal control   for enhancing  heat transfer  and fluid  mixing or optic flow control  via flow advection,  governed by nonstationary diffusion-convection, has   been discussed in  (cf.~\cite{barbu2016optimal, hu2017enhancement, hu2018optimal,  liu2008mixing}).   However,  to solve the resulting nonlinear optimality system, one has to solve the state equations forward in time, coupled with  the adjoint system backward in time 
together with a nonlinear optimality condition.  This leads to extremely high computational costs and  intractable problems. Some preliminary numerical results were obtained in \cite{liu2008mixing} with  simplified conditions. As a first step to tackle such a complex system, our current work will focus  on the stationary  case and  present  a rigorous theoretical and numerical study of the optimal control design.

Now denote the spatial average of temperature by 
\[\langle T\rangle=\frac{1}{|\Omega|} \int_{\Omega} T\, dx.\]
The objective is  to minimize the variance of the temperature with optimal control cost, that is,
 \begin{align*}
 J({\bv})&=\frac{1}{2}\|T-\langle T\rangle\|^2_{L^2}
 +\frac{\gamma}{2}\|\bv\|^2_{U_{\text{ad}}}, \qquad (P)
 \end{align*}
subject to \eqref{sta_T}--\eqref{sta_BC}, where $\gamma>0$ is the   control weight parameter and $U_{\text{ad}}$  stands for the set of admissible control.  The choice of the  set of admissible control  is usually dependent on  the physical properties and  the need to establish the existence of an optimal control. Due to the advection term $\bv\cdot \nabla T$, the control map  $\bv\mapsto T$ is bilinear and hence  problem $(P)$ is non-convex.  Establishing the  existence of an optimal velocity field  will involve  a  compactness argument associated with the control map. Moreover,
 in order to reduce the effects of rotation on the flow and the shear stress at the boundary in the cooling process,  we consider to minimize the magnitude of the strain tensor  (cf.~\cite{foias2001navier}), which is equivalent to minimize   $\|\nabla \bv\|_{L^2}$. To this end, we set
  $$
  U_{\text{ad}}=\{\bv\in H^1_0(\Omega)\colon \nabla \cdot  \bv=0\}
  $$ equipped with $H^1$-norm
\[\|\bv\|_{U_{\text{ad}}}=\|\bv\|_{H^1}.\]

 The remainder of this paper is organized as follows. Section~\ref{existence} focuses on the existence of an optimal solution to problem $(P)$. {Section}~\ref{opt_cond} presents the first and second order optimality  conditions for solving and charactering  the optimal solution by using a variational inequality (cf.~\cite{ lions1971optimal}). Moreover, it can be shown that there exists a strict local minimizer  if the control weight $\gamma$ is large enough. Section~\ref{Sect:Implementation}  discusses  the numerical implementation of our control design, where the finite element formulation and nonlinear iterative solvers are used to construct our numerical schemes. In particular,  the relation regarding the solutions of the optimality system associated with different values in $\kappa$ and $\gamma$ is established. This result provides a practical guidance for choosing  these parameters in our  numerical implementation.   In Section~\ref{Sect:Num}, several numerical experiments  are conducted  to demonstrate the effectiveness of our control design for convection-cooling. {Lastly, this paper concludes with potential problems for future work in Section \ref{Sect:Con}.}

In the sequel, the symbol $C$ denotes a generic positive constant, which is allowed to depend on the domain as well as on indicated {parameters without ambiguous.}

\section{ Existence of an Optimal Solution }\label{existence}

As a starting point to analyze  problem $(P)$, 
we first recall   some basic  properties of the state equations \eqref{sta_T}--\eqref{sta_BC}. The following lemmas  will be often used  in this paper. 

\begin{lemma}\label{lem0}
Let $\mathbf{w}\in ( H^1(\Omega))^d, d=2,3$, and $\phi,\psi \in H^1(\Omega)$.
Then  we have 
\begin{align}
\left|\int_{\Omega} \mathbf{w} \cdot \nabla \phi \psi\,dx\right|\leq \|\mathbf{w}\|_{L^4}\|\nabla \phi\|_{L^2}\|\psi\|_{L^4}
\leq C\|\nabla \mathbf{w}\|_{L^2}\|\nabla \phi\|_{L^2}\|\nabla \psi\|_{L^2}.  \label{1EST_tri}
\end{align}
Moreover, if $\nabla \cdot \mathbf{w}=0$ and $\mathbf{w}|_{\Gamma}=0$, then 
\begin{align}
\int_{\Omega} \mathbf{w} \cdot \nabla \phi \psi\,dx=-\int_{\Omega} \phi  \mathbf{w} \cdot \nabla\psi\,dx.
\label{2EST_tri}
\end{align}
\end{lemma}
\begin{proof} Inequalities in  \eqref{1EST_tri}  are  direct results of  H\"{o}lder's inequality  and Sobolev embedding theorem (cf.~\cite{Te1997}). To see \eqref{2EST_tri}, applying Stokes formula together with $\nabla \cdot \mathbf{w}=0$ and $\mathbf{w}|_{\Gamma}=0$ follows
\begin{align*}
\int_{\Omega} \mathbf{w} \cdot \nabla \phi \psi\,dx
&=\int_{\Omega} \mathbf{w} \cdot \nabla( \phi \psi)\,dx-\int_{\Omega}\phi  \mathbf{w} \cdot \nabla  \psi \,dx\\
&=\int_{\Gamma} \mathbf{w}\cdot n ( \phi \psi)\,dx-\int_{\Omega} \nabla\cdot \mathbf{w} \phi \psi\,dx
-\int_{\Omega}\phi  \mathbf{w} \cdot \nabla  \psi \,dx\\
&=-\int_{\Omega}\phi  \mathbf{w} \cdot \nabla  \psi \,dx.
\end{align*}

\end{proof}
\begin{lemma} \label{lem1}
  Let  $f\in L^\infty(\Omega)$. For $\bv\in L^2(\Omega)$ with $\nabla \cdot \bv=0$ and  $\bv|_{\Gamma}=0$,     there exists a unique weak solution  
to  equation  \eqref{sta_T} with Dirichlet boundary condition $T|_{\Gamma}=0$, which satisfies  
$T\in H^1_0(\Omega)\cap L^\infty(\Omega)$. Moreover, 
   \begin{align}
\|T\|_{L^2}+\|\nabla T\|_{L^2}\leq  \frac{C}{\kappa}  \|f\|_{L^2} \label{EST_TH1}
\end{align}
and 
   \begin{align}
\|T\|_{L^\infty}\leq C\|f\|_{L^{\infty}}, \label{EST_Tinfty}
\end{align}
where  $C>0$ depends on $\Omega$ but not on $f$.
   \end{lemma}
\begin{proof}
The existence of a unique solution follows the standard approaches  for the elliptic equations 
(cf.~\cite{evans2010partial}). To see \eqref{EST_TH1},  taking  the {inner product} of \eqref{sta_T} with $T$ and integrating  by parts using \eqref{sta_BC}, we have
\begin{align}
\kappa \|\nabla T\|^2_{L^2} &=-\int_{\Omega} (\bv\cdot \nabla T)T\, dx
+\int_{\Omega}fT\,dx\nonumber\\
&\leq-\frac{1}{2}\int_{\Omega}\bv\cdot  \nabla (T^2)\,dx+ \|f\|_{L^2}\|T\|_{L^2}\nonumber\\
&= -\frac{1}{2} (\int_{\Gamma}\bv\cdot  n \, T^2\,dx - \int_{\Omega}\nabla \cdot \bv \,T^2\,dx)
+\|f\|_{L^2}\|T\|_{L^2} \nonumber\\
&=  \|f\|_{L^2}\|T\|_{L^2}\leq C \|f\|_{L^2}\|\nabla T\|_{L^2}, \label{1EST_TL2}
\end{align}
which follows 
$$
\|\nabla T\|_{L^2}\leq  \frac{C}{\kappa}  \|f\|_{L^2}.
$$
Note that in \eqref{1EST_TL2}  we have used Poncar\'{e} inequality  
$\|T\|_{L^2}\leq  C\| \nabla T\|_{L^2}$,
where   $C>0$ is a constant dependent on domain $\Omega$ but not $f$.   

Analogously, taking the inner product of  \eqref{sta_T} with $T^{N-1}$ for a positive even
integer $N$ and then letting $N\to \infty$ we get \eqref{EST_Tinfty}. In fact,  
a finer  estimate of $f$ in \eqref{EST_Tinfty} can be achieved  by using the  Stampacchia  theory.  
The reader is referred to \cite{stampacchia1965probleme} for details.
This completes the proof.
\end{proof}

To show the existence of an  optimal  control to problem $(P)$,  we first  introduce the weak solution to  \eqref{sta_T}--\eqref{sta_BC}. 

 \begin{definition}\label{def1}
Let $f\in L^{\infty}(\Omega)$ and $\bv\in U_{\text{ad}}$. $T\in H^1_0(\Omega)$  is said to be  a weak solution to system \eqref{sta_T}--\eqref{sta_BC}, if $T$ satisfies 
  \begin{align}
\kappa(  \nabla T,  \nabla \psi)- (T\bv , \nabla \psi)&=(f, \psi),  \quad  \forall\psi\in H^{1}_0(\Omega)\label{1weak_T}.
 \end{align}
\end{definition}

\begin{thm}\label{thm1} 
For $f\in L^\infty(\Omega)$, there exists an optimal velocity
 $\bv\in U_{\text{ad}}$   to problem $(P)$.
\end{thm}
\begin{proof}
Since $J$ is bounded from below, we  may choose a minimizing sequence 
 $\{\bv_{m}\} \subset U_{\text{ad}_1}$ such that 
 \begin{align}
 \lim_{m\to \infty} J(\bv_{m})=\inf_{\bv\in U_{\text{ad}}}J(\bv). \label{inf_J}
 \end{align}
This also indicates that $\{\bv_{m}\}$ is uniformly bounded in  $U_{\text{ad}}$, and hence  there exists a weakly convergent subsequence, still denoted by  $\{\bv_{m}\}$, such that 
    \begin{align}
&\bv_{m}\to \bv^{*} \quad \text{weakly in}\quad  H^{1}(\Omega), \ \ \text{as}\ \  m\to \infty,
 \label{cong_vH1}\\
&\bv_{m}\to \bv^{*} \quad \text{strongly  in}\quad L^2(\Omega), \ \ \text{as}\ \  m\to \infty. \label{cong_vL2}
 \end{align}
Let $\{T_m\}$ be the solutions corresponding to  $\{\bv_{m}\}$. Then  $\{T_m\}$ is uniformly bounded in $H^1(\Omega)\cap L^\infty(\Omega)$ 
according to  \eqref{EST_TH1} and \eqref{EST_Tinfty} . Thus there exists   a subsequence,  still denoted by $\{T_{m}\}$, satisfying 
  \begin{align}
 &T_{m}\to  T^* \quad \text{weakly  in}\  H^1(\Omega),  \ \text{as}\ \  m\to \infty, \label{EST_wcong_T}\\
& T_{m}\to  T^* \quad \text{weakly*  in}\  L^{\infty}(\Omega),  \ \text{as}\ \  m\to \infty. \label{EST_wstarcong_T}
 \end{align}
Next we  show  that   $T^*$  is the solution  corresponding to $\bv^{opt}$ by Definition \ref{def1}.  Recall  that $\bv_{m}$ and $T_{m}$ satisfy  
  \begin{align}
\kappa(  \nabla T_m,  \nabla \psi)- (T_m\bv_m , \nabla \psi)=(f, \psi),  \quad  \forall\psi\in H^{1}_0(\Omega)\label{weak_Tm},
 \end{align} 
With the help of  \eqref{EST_wcong_T}, it is easy to pass to the limit in the  first term on the left hand of \eqref{weak_Tm}. Next  we show that  applying \eqref{cong_vH1}--\eqref{cong_vL2} and \eqref{EST_wstarcong_T} makes passing to the limit in the nonlinear term $\bv_m  T_m\to \bv^{*}   T^*$ possible. 

In fact,  for the second term  on the left hand of \eqref{weak_Tm}, we have for $\psi\in H^1_0(\Omega)$,
 \begin{align}
& \left| \int_{\Omega} T_{m}\bv_{m}\cdot\nabla  \psi\, dx - \int_{\Omega}T^{*} \bv^{*}\cdot\nabla  \psi\, dx\right|\nonumber\\
 &\quad\leq \left| \int_{\Omega} T_{m}\bv_{m}\cdot\nabla  \psi
 -T_{m}\bv^{*}\cdot\nabla \psi\, dx\right| \nonumber\\
 &\qquad+\left| \int_{\Omega} T_{m}\bv^{*}\cdot\nabla \psi- T^{*} \bv^{*}\cdot\nabla \psi\, dx\right|\nonumber\\
 &\quad=I_1+I_2, \label{EST_I12}
 \end{align}
where 
  \begin{align*}
I_1  
&\leq  \|T_{m}\|_{L^{\infty}} \|\bv_{m}-\bv^{*}\|_{L^{2}} \|\nabla \psi\|_{L^2} 
\to 0\quad \text{as}\quad m\to\infty,
 \end{align*}
due to  \eqref{cong_vL2} and the uniform boundedness of $\|T_{m}\|_{L^{\infty}}$. Moreover,  $I_2\to 0$ due to   
 \eqref{EST_wstarcong_T} and $\bv^{*}\nabla \psi\in L^1(\Omega)$.
Clearly, $T^*\in H^1_0(\Omega)$  is the solution  corresponding to $\bv^{*}$ based on Definition \ref{def1}.

 Lastly, using the weakly lower semicontinuity property of norms yields 
\[\|\bv^{*}\|_{U_{\text{ad}}} \leq  \displaystyle \varliminf_{m\to \infty} \|\bv_m\|_{U_{\text{ad}}} \quad \text{and}\quad
\|T^*-\langle T^*\rangle \|_{L^2} \leq \varliminf_{m\to \infty} \|T_m-\langle T_m\rangle\|_{L^2}.
\]
In other words,
\[J(\bv^{*})\leq   \varliminf_{m\to \infty} J(\bv_m)=\inf_{\bv\in U_{\text{ad}}}J(\bv),\]
which indicates that $\bv^*$ is an optimal solution to  problem $(P)$.
\end{proof}

\section{Optimality Conditions}\label{opt_cond}
Now we derive the first order  necessary  optimality conditions  for problem $(P)$ by using a variational inequality  (cf.~\cite{lions1971optimal}), that is,  
if $\bv$ is an optimal solution to problem $(P)$, then
\begin{align} 
J'(\bv)\cdot (\psi-\bv)\geq 0, \quad  \psi\in U_{\text{ad}}. \label{var_ineq}
\end{align}
To establish the  G$\hat{a}$teaux differentiability of  $J(\bv)$, we first check the G$\hat{a}$teaux differentiability of $T$ with respect to $\bv$. 
  Let  $z$ be the G$\hat{a}$teaux of $T$ with respect to $\bv$ in the direction of $h\in U_{\text{ad}}$, i.e., $z = T'(\bv )\cdot h$. Then $z$ satisfies 
 \begin{equation}\label{z-equation}
	\begin{split}
		-\kappa \Delta z + \bv  \cdot \nabla z  + h \cdot \nabla T&=0,\\\
		 z|_{\Gamma}&=0.
	\end{split}
\end{equation}
Using the $L^2$-estimate as in Lemma \ref{lem1} with the help of Lemma \ref{lem0} and \eqref{EST_TH1}, we get
 \begin{align}
\kappa \|\nabla z\|^2_{L^2}&\leq \left|\int_{\Omega} (h\cdot \nabla  T) z\,dx\right| 
\leq C \|\nabla h\|_{L^2}\|\nabla T\|_{L^2} \|\nabla z\|_{L^2}, \label{EST_zL2}
\end{align}
which implies 
\begin{align}
 \|\nabla z\|_{L^2}  &\leq \frac{C}{\kappa} \|\nabla h\|_{L^2} \|\nabla T\|_{L^2}
 \leq  \frac{C}{\kappa^2} \|f\|_{L^2}\|\nabla h\|_{L^2}. \label{EST_zH1}
\end{align}
Therefore, $T(\bv)$ is  G$\hat{a}$teaux differentiable for $\bv\in U_{\text{ad}}$, so is $J(\bv)$.


\subsection{First Order Optimality Conditions}

Let $A=-\mathbb{P}\Delta$ be the Stokes operator   with
 $$
 D(A)=\{ H^1_0(\Omega)\cap H^2(\Omega)\colon \nabla \cdot \bv=0\},
$$ 
where $\mathbb{P}\colon L^2(\Omega)\to \{\bv\in L^2(\Omega)\colon \nabla \cdot \bv=0 \ \text{and}\ 
\bv\cdot \bn|_{\Gamma}=0\}$ is the Leray projector.  Note that $A$ is a strictly positive and self-adjoint operator. Moreover,  define operator $D\colon L^2(\Omega)\to L^2(\Omega)$  by $DT = T - \langle T\rangle$.  Then  $D$ is  a bounded linear operator.  The cost functional now  can be  rewritten as
 \begin{align}\label{J(v)}
 J(\bv)&=\frac{1}{2}(D^*DT, T) +\frac{\gamma}{2}(A\bv,\bv ),
  \end{align}
  where $D^*$ is the $L^2$-adjoint operator of $D$.
\begin{remark}\label{D}
Here we present some basic properties of operator $D$. For any $T, \psi\in L^2(\Omega)$, since $\langle T\rangle$ and $\langle \psi\rangle $ are constants, we have 
 $$\frac{1}{|\Omega|}\int_{\Omega}T\langle \psi\rangle \, dx
 =\langle T\rangle \langle \psi\rangle
 =\frac{1}{|\Omega|}\int_{\Omega}\langle T\rangle \psi\, dx.
 $$
Therefore,
\begin{align*}
(DT, \psi)&=\int_{\Omega}(T-\langle T\rangle )\psi\, dx=\int_{\Omega}T\psi\, dx -\int_{\Omega}\langle T\rangle \psi\, dx\nonumber\\
&=\int_{\Omega}T\psi\, dx -\int_{\Omega}T\langle \psi\rangle \, dx=(T, \psi-\langle \psi\rangle )=(T, D\psi),
\end{align*}
which says that $D$ is a self-adjoint operator on $L^2(\Omega)$, i.e., $D=D^*$.
Moreover, since 
$$ \langle T -  \langle T\rangle\rangle =\frac{1}{|\Omega|} \int_{\Omega} (T-\langle T\rangle)\, dx=\langle T\rangle-\langle T\rangle=0,$$ 
it is straightforward  to verify that 
\begin{align*}
D^*DT&=D(DT)=D(T-\langle T\rangle )=T-\langle T\rangle - \langle T -  \langle T\rangle\rangle =DT,
\end{align*}
which  implies that $D^2=D$, and hence the operator norm $\|D\|\leq 1$.
\end{remark}

Now let $q$ be the adjoint state associated with $T$. Then it is easy to verify that $q$ satisfies 
\begin{equation}\label{q-equation1}
	\begin{split}
		-\kappa \Delta q  - \bv \cdot \nabla q &=D^*DT \quad \mbox{in } \Omega,\\\
		 q|_{\Gamma}&=0.
	\end{split}
\end{equation}
Moreover, thanks to \eqref{EST_TH1} and $\|D\|\leq 1$,  we have
\begin{align}
\|\nabla q\|_{L^2}&\leq   \frac{C}{\kappa}  \|T\|_{L^2}\leq  \frac{C}{\kappa^2}  \|f\|_{L^2}.
 \label{EST_qH1}
\end{align}

The following theorem establishes the first order necessary optimality conditions for solving    the optimal solution. 

\begin{thm}
Assume that  $\bv^{\text{opt}}$ is an optimal solution to problem $(P)$. Let  $(T^{opt}, q^{opt} )$ be the corresponding solution to the state equations \eqref{sta_T}--\eqref{sta_BC} and  the adjoint system \eqref{q-equation1}. Then $(\bv^{\text{opt}}, T^{opt}, q^{opt})$  satisfies 
\begin{eqnarray}\label{case2}
\begin{cases}
	-\kappa \Delta T+ {\bv} \cdot \nabla T=f \quad 
	\text{ and }\quad T|_{\partial\Omega} = 0,\\
	-\kappa \Delta q - {\bv}  \cdot \nabla   q  = D^*DT \quad 	\text{ and }\quad q|_{\partial\Omega} = 0,\\
-\gamma\Delta {\bv} +\nabla p = q \nabla  T, \quad \nabla\cdot \bv = 0  \quad  \text{ and } \quad  \bv|_{\partial\Omega} = 0.
\end{cases}
\end{eqnarray}
\end{thm}
\begin{proof}
In light of  \eqref{J(v)},  \eqref{q-equation1}, and  \eqref{2EST_tri}, the  G\^{a}teaux  derivative  of $J$  becomes 
\begin{align*}
		J'(\bv) \cdot h =&(D^*DT, z) + \gamma(A \bv, h)\nonumber\\
		=&(-\kappa \Delta q  - \bv \cdot \nabla q, z) + \gamma(A\bv, h)\nonumber\\
		=& (q,    - \kappa \Delta  z+ \bv\cdot  \nabla z) +  \gamma(A\bv, h). 
\end{align*}
Using \eqref{z-equation} we get 
\begin{align*}
	J'(\bv) \cdot h =&-( q, h \cdot \nabla T)+ \gamma(A\bv, h).
\end{align*}
If $\bv^{\text{opt}}$ is the optimal solution, then $J'(\bv^{\text{opt}})\cdot h\geq 0$ for any $h\in U_{\text{ad}}$. This yields the  following 
optimality condition 
\begin{align}
 \gamma A\bv^{\text{opt}}-\mathbb{P}(q\nabla T)=0. \label{opt_v}
\end{align}
In other words,  there exists $p\in L^2(\Omega)$ with $\int_{\Omega}p\,dx=0$ such that 
\begin{align*}
-\gamma\Delta \bv^{\text{opt}}+\nabla p=q\nabla T,
\end{align*}
which completes the proof.
\end{proof}

\subsection{Second Order Optimality Conditions}

In this section, we discuss  the  second order optimality conditions for characterizing the optimal velocity field.  In particular, it can be shown that the cost functional  $J$ has a strict local minimizer when the control weight $\gamma>0$ is sufficiently large.

\begin{thm}\label{SONC}
Let $(\bv, T, q)$ satisfy the first order necessary optimality system  \eqref{case2}.  If $\gamma>0$ is sufficiently large,  then  there exists some constant $\delta>0$ such that 
\begin{align}
J''(\bv)\cdot (h, h)\geq \delta\|h\|_{U_{\text{ad}}},   \label{opt_2rd}
\end{align}
for $ h\in U_{\text{ad}}$ satisfying  \eqref{z-equation}.
\end{thm}
\begin{proof}
Let $h_i\in U_{\text{ad}}$  and $z_i=T'(\bv)\cdot h_i, i=1,2$. Then  we have 
\begin{equation*}
	\begin{split}
		-\kappa \Delta z_i  + \bv \cdot \nabla z_i + h_i \cdot \nabla T&=0 \quad \mbox{in } \Omega,\\\
		 z_i|_{\Gamma}&=0.
	\end{split}
\end{equation*}
Moreover, let  $Z=z'_1(\bv)\cdot h_2$. 
 Then $Z$ satisfies 
\begin{align}
		-\kappa \Delta Z +  h_2 \cdot \nabla z_1 +\bv \cdot \nabla Z+h_1 \cdot \nabla z_2&=0 \quad \mbox{in } \Omega,\label{Z-equation}\\
		 Z|_{\Gamma}&=0. \nonumber
\end{align}
Again applying an $L^2$-estimate for $Z$ and using \eqref{EST_zH1}, we can easily verify that 
\begin{align}
\|\nabla Z\|_{L^2}&\leq 
\frac{C}{\kappa}(\| \nabla h_2\|_{L^2}\|\nabla z_1\|_{L^2}+\|\nabla  h_1\|_{L^2}\|\nabla z_2\|_{L^2})\nonumber\\
& \leq\frac{C}{\kappa^3}\|f\|_{L^2}\| \nabla h_1\|_{L^2}\|\nabla h_2\|_{L^2}, \label{EST_ZH1}
\end{align}
which implies that $T(\bv)$ is twice G$\hat{a}$teaux differentiable for $\bv\in U_{\text{ad}}$, so is $J(\bv)$.

 Now differentiating $J'(\bv)\cdot h_1$ once again in the direction $h_2\in U_{\text{ad}}$ gives
\begin{align}\label{optimality_3}
		J''(\bv) \cdot (h_1,h_2) =(D^*Dz_2, z_1) +(D^*DT, Z)+ \gamma(Ah_2, h_1).	
\end{align} 
To further analyze  the second  term involving $Z$, we take  the  inner product of \eqref{Z-equation} with $q$ and apply 
 \eqref{2EST_tri}. We get
\begin{align*}
		&-\kappa (Z, \Delta q)- ( z_1, h_2 \cdot \nabla q) 
		-( Z, \bv\cdot \nabla q) -( z_2, h_1\cdot \nabla q)=0.
\end{align*}
With the help of   the adjoint equation  \eqref{q-equation1},  we obtain   
\begin{align*}
		( z_1, h_2 \cdot \nabla q)+( z_2, h_1\cdot \nabla q)
		=(Z, D^*DT).
\end{align*}
Therefore,
\begin{align*}
		J''(\bv) \cdot (h_1,h_2) =&	(D^*Dz_2, z_1) +( z_1, h_2 \cdot \nabla q)+( z_2, h_1\cdot \nabla q)
                                                  + \gamma(Ah_2, h_1).
\end{align*} 
Setting $h_1=h_2=h$ and $z_1=z_2=z=T'(\bv)\cdot h$ follows
\begin{align}\label{optimality_4}
		J''(\bv) \cdot (h, h) =\|Dz\|^2_{L^2}+2( z, h\cdot \nabla q)
		+\gamma\|A^{1/2}h\|^2_{L^2}.
\end{align} 
Furthermore, by \eqref{1EST_tri}, \eqref{EST_zH1} and \eqref{EST_qH1}, we get
\begin{align*}
|\int_{\Omega}z h\cdot \nabla q\, dx| 
\leq  C \|\nabla z\|_{L^2}\|\nabla h\|_{L^2}\|\nabla q\|_{L^2}
\leq   \frac{C}{\kappa^4} \|f\|^2_{L^2}\|A^{1/2} h\|^2_{L^2} 
\end{align*}
and 
\begin{align*}
		\|Dz\|_{L^2} \le C \|\nabla z\|_{L^2}  \le \dfrac{C}{\kappa^2} \| f\|_{L^2} \| A^{1/2} h\|_{L^2}.
\end{align*} 
Consequently, 
\begin{align}\label{EST_HJ1}
		|J''(\bv) \cdot (h, h)| &\leq  \frac{C}{\kappa^4} \|f\|^2_{L^2}\|A^{1/2} h\|^2_{L^2}	
		+\gamma\|A^{1/2} h\|^2_{L^2}
		=( \frac{C}{\kappa^4} \|f\|^2_{L^2}
		+\gamma)\|A^{1/2} h\|^2_{L^2}
\end{align} 
and 
\begin{align}\label{EST_HJ2}
	J''(\bv) \cdot (h, h) \geq &-2|( z, h\cdot \nabla q)|
		+\gamma\|A^{1/2} h\|^2_{L^2}
     = (\gamma-\frac{C}{\kappa^4}\|f\|^2_{L^2})\|A^{1/2} h\|^2_{L^2}.
	\end{align}  
Therefore, if $\gamma$ is large enough such that  
\begin{align}
\gamma-\frac{C}{\kappa^4}\|f\|^2_{L^2}\geq \delta>0, \label{cond_gamma}
\end{align}
  then  \eqref{opt_2rd} holds.
\end{proof}

\begin{lemma}\label{lem2}
There exists a constant $C>0$ such that
\begin{align}
		|(J''(\bv_1)  -J''(\bv_2))\cdot (h, h)| 
		\leq  \frac{C}{\kappa^5}\|\bv_1-\bv_2\|_{H^1}\|f\|^2_{L^2}\| h\|^2_{H^1},  \label{cond_J2nd}
\end{align}
for any $h, \bv_i\in U_{\text{ad}}, i=1,2$.
\end{lemma}
\begin{proof}
Let $h, \bv_i \in U_{\text{ad}}$ and  $z_i=T'_i(\bv_i)\cdot h, i=1,2$. Here $T_i$ is the temperature corresponding to $\bv_i$. Then $z_i$ satisfies 
\begin{equation*}
	\begin{split}
		-\kappa \Delta z_i  + \bv_i \cdot \nabla z_i + h \cdot \nabla T_i&=0 \quad \mbox{in } \Omega,\\\
		 z_i|_{\Gamma}&=0. 
	\end{split}
\end{equation*}
Further let $\tilde{z}=z_1-z_2$, $\tilde{\bv}=\bv_1-\bv_2$ and $\tilde{T}=T_1-T_2$. Then 
\begin{equation}\label{zt-equation}
	\begin{split}
		-\kappa \Delta \tilde{z}  + \tilde{\bv} \cdot \nabla z_1 + \bv_2 \cdot \nabla \tilde{z} + h \cdot \nabla \tilde{T}&=0 \quad \mbox{in } \Omega,\\\
		 \tilde{z}|_{\Gamma}&=0.
	\end{split}
\end{equation}
By \eqref{sta_T}-\eqref{sta_BC} and \eqref{EST_TH1} it is easy to check that 
\begin{align}
\|\tilde{T}\|_{H^1}\leq  \frac{C}{\kappa} \|\tilde{\bv}\|_{H^1}\|T\|_{H^1} \leq  \frac{C}{\kappa^2}  \|\tilde{\bv}\|_{H^1} \|f\|_{L^2}.   \label{Tt}
\end{align}
Moreover, applying an $L^2$-estimate to \eqref{zt-equation} yields 
\begin{align}
\|\nabla \tilde{z}\|_{L^2} &\leq \frac{C}{k}(\|\tilde{\bv}\|_{H^1}\|z_1\|_{H^1}+\|h\|_{H^1}\|\tilde{T}\|_{H^1})\nonumber\\
&\leq  \frac{C}{k}\left(\|\tilde{\bv}\|_{H^1}\frac{C}{\kappa^2} \|f\|_{L^2}\| h\|_{H^1}+\|h\|_{H^1}\frac{C}{\kappa^2}  \|\tilde{\bv}\|_{H^1} \|f\|_{L^2}\right)\nonumber\\
&\leq \frac{C}{\kappa^3}\|\tilde{\bv}\|_{H^1}\|f\|_{L^2}\| h\|_{H^1}.\label{EST_ztH1}
\end{align}
Now let  $Z_i=z'_i(\bv_i)\cdot h, i=1,2,$. Then 
\begin{equation*}
	\begin{split}
		-\kappa \Delta Z_i  + 2h \cdot \nabla z_i +\bv_i\cdot \nabla Z_i&=0 \quad \mbox{in } \Omega,\\\
		 Z_i|_{\Gamma}&=0. 
	\end{split}
\end{equation*}
In light of \eqref{EST_ZH1}, we have 
\begin{align}
\|\nabla Z_i\|_{L^2} \leq\frac{C}{\kappa^3}\|f\|_{L^2}\| \nabla h\|^2_{L^2}.\label{2EST_ZH1}
\end{align}
Furthermore,  let $\tilde{Z}=Z_1-Z_2$. Then  
\begin{equation}\label{Zt-equation}
	\begin{split}
		-\kappa \Delta \tilde{Z}  + 2h \cdot \nabla \tilde{z} +\tilde{\bv} \cdot \nabla Z_1+ \bv_2\cdot \nabla\tilde{Z}&=0 \quad \mbox{in } \Omega,\\\
		 \tilde{Z}|_{\Gamma}&=0.
	\end{split}
\end{equation}
Again applying  an $L^2$-estimate to \eqref{Zt-equation} and using \eqref{EST_ztH1}-\eqref{2EST_ZH1} follow
\begin{align}
\|\nabla \tilde{Z}\|_{L^2}& \leq  \frac{C}{\kappa}  (2\|h\|_{H^1} \|\tilde{z}\|_{H^1}+\|\tilde{\bv}\|_{H^1} \|Z_1\|_{H^1})\nonumber\\
& \leq  \frac{C}{\kappa}  \left(2\|h\|_{H^1} \frac{C}{\kappa^3}\|\tilde{\bv}\|_{H^1}\|f\|_{L^2}\| h\|_{H^1}+\|\tilde{\bv}\|_{H^1} \frac{C}{\kappa^3}\|f\|_{L^2}\|  h\|^2_{H^1}\right)\nonumber\\
&\leq \frac{C}{\kappa^4}\|\tilde{\bv}\|_{H^1} \|f\|_{L^2}\| h\|^2_{H^1}.\label{EST_ZtH1}
\end{align}
Finally, applying  \eqref{optimality_3} together with   \eqref{EST_TH1},  \eqref{EST_zH1},   \eqref{EST_ztH1}-\eqref{2EST_ZH1}, \eqref{EST_ZtH1} and $\|D\|\leq 1$ follows
\begin{align*}
		&|(J''(\bv_1)  -J''(\bv_2))\cdot (h, h)| =\|Dz_1\|^2_{L^2} +(D^*DT, Z_1)+\gamma\|A^{1/2}h\|^2_{L^2}
		\nonumber\\
		&\qquad\quad-(\|Dz_2\|^2_{L^2} +(D^*DT, Z_2)+\gamma\|A^{1/2}h\|^2_{L^2})\nonumber\\
		&\qquad=(\|Dz_1\|^2_{L^2}-\|Dz_2\|^2_{L^2})+(D^*DT, Z_1-Z_2)  \nonumber\\
		&\qquad\leq (\|z_1\|_{L^2}+\|z_2\|_{L^2})\|z_1-z_2\|_{L^2}+\|T\|_{L^2} \|Z_1-Z_2\|_{L^2} \nonumber\\
		&\qquad\leq  \frac{C}{\kappa^2} \|f\|_{L^2}\| h\|_{H^1}
		\frac{C}{\kappa^3}\|\tilde{\bv}\|_{H^1}\|f\|_{L^2}\| h\|_{H^1}
		+  \frac{C}{\kappa}  \|f\|_{L^2} \frac{C}{\kappa^4}\|\tilde{\bv}\|_{H^1} \|f\|_{L^2}\| h\|^2_{H^1} \nonumber\\
		&\qquad\leq \frac{C}{\kappa^5}\|\tilde{\bv}\|_{H^1}\|f\|^2_{L^2}\| h\|^2_{H^1},
\end{align*}
which establishes  the desired result.
\end{proof}

Now we are in a position to address the second order sufficient conditions. Let $\mathbf{w}\in U_{\text{ad}}$ and $T$ be the associated  solution to  \eqref{sta_T}--\eqref{sta_BC}. Define 
\begin{align*}
\mathbb{T}(\mathbf{w})=\{\bv\in U_{\text{ad}}&\colon  z=T'(\mathbf{w})\cdot (\bv-\mathbf{w}) \ \text{satisfies}\\
& -\kappa \Delta z + \mathbf{w} \cdot \nabla z  + (\bv-\mathbf{w}) \cdot \nabla T=0,\
		 z|_{\Gamma}=0\}. 
\end{align*}

\begin{corollary}\label{thm_SOSC}
Let  $\bv^*$ satisfy the optimality condition \eqref{opt_v}. If $\gamma>0$ is sufficiently large,
then 
 there exist $ \epsilon, \delta_0>0$ such that 
 the quadratic growth condition 
\begin{align}
J(\bv^{*})+\delta_0\|\bv-\bv^{*}\|^2_{U_{\text{ad}}}\leq J(\bv)
  \label{gap}
\end{align}
holds for all $ \bv\in \mathbb{T}(\bv^{*})$ and  $\|\bv-\bv^{*}\|_{H^1}\leq \epsilon$. In particular, $J$ has a local minimum in $ \mathbb{T}(\bv^{*})$  at $\bv^*$.
\end{corollary}
\begin{proof}
To see the gap between  $J(\bv^{*})$ and $J(\bv)$ for any $ \bv\in \mathbb{T}(\bv^{*})$ satisfying  $\|\bv-\bv^{*}\|_{H^1}\leq \epsilon$, we  apply a Taylor expansion of $J(\bv)$ around $\bv^{*}$. With the help of Theorem \ref{SONC} and Lemma \ref{lem2} and setting $h=\bv-\bv^*$, we have for  $\xi\in (0, 1)$,
\begin{align*}
J(\bv)-J(\bv^{*})&=J'(\bv^{*})\cdot (\bv-\bv^{*})+\frac{1}{2}J''(\bv^{*}+\xi(\bv-\bv^{*}))\cdot (\bv-\bv^{*}, \bv-\bv^{*})\\
&=\frac{1}{2}J''(\bv^{*})\cdot (\bv-\bv^{*}, \bv-\bv^{*})\\
&\quad+ \frac{1}{2} (J''(\bv^{*}+\xi(\bv-\bv^{*} ))-J''(\bv^{*} ))\cdot (\bv-\bv^{*}, \bv-\bv^{*} )\\
&\geq \frac{1}{2}  \delta \|\bv-\bv^{*} \|^2_{H^1}
-  \frac{1}{2}  \frac{C}{\kappa^5}\|\xi(\bv-\bv^{*}) \|_{H^1}\|f\|^2_{L^2}\| \bv-\bv^{*}\|^2_{H^1}\\
&=  \frac{1}{2}\left( \delta - \frac{C}{\kappa^5}\|\xi(\bv-\bv^{*} ) \|_{H^1}\|f\|^2_{L^2}\right)\| \bv-\bv^{*} \|^2_{H^1}\\
&\geq \frac{1}{2} \left(\delta-\frac{C\epsilon}{\kappa^5} \|f\|^2_{L^2}\right) \|\bv-\bv^{*} \|^2_{H^1}.
\end{align*}
Therefore, if  letting   $0<\delta_0\leq \frac{1}{2}(\delta-\frac{C\epsilon}{\kappa^5} \|f\|^2_{L^2})$ 
 or
 $\gamma\geq 2\delta_0 +\frac{C}{\kappa^4}\|f\|^2_{L^2}+ \frac{C\epsilon}{\kappa^5} \|f\|^2_{L^2} $
 for some constants $\delta_0,C>0$, then \eqref{gap} holds, which  completes the proof.
\end{proof}


\section{Numerical Implementation }\label{Sect:Implementation}
In this section, we shall present a detailed numerical 
implementation for solving the optimality system \eqref{case2} based on  a 2D problem. 
The following lemma establishes  the relation between the diffusivity coefficient $\kappa$ and the control weight parameter $\gamma$, which indicates that it is sufficient  to test the numerical examples  for $\kappa=1$. The results for other $\kappa$ values can then be obtained by this relation.

\begin{lemma}\label{lem:sol-k}
Let $[T_{\gamma}, q_{\gamma}, \bv_{\gamma},p_{\gamma}]$
be the solution to  \eqref{case2} corresponding $\kappa=1$ and $\gamma$. Let
$[T_{\kappa, \tilde \gamma}, q_{\kappa,\tilde \gamma}, \bv_{\kappa, \tilde \gamma},p_{\kappa, \tilde \gamma}]$
be the solution to  \eqref{case2} corresponding $\kappa$ and $\tilde\gamma$ where 
 $\tilde \gamma= \dfrac{1}{\kappa^4} \gamma$.
Then the following relation holds:
	\[
	T_{\kappa,\tilde\gamma} =  \dfrac{1}{\kappa} T_{\gamma}, \quad
	q_{\kappa,\tilde\gamma} =  \dfrac{1}{\kappa^2} q_{\gamma}, \quad
	\bv_{\kappa,\tilde\gamma} = \kappa \bv_{\gamma},
	\quad\mbox{ and } \quad
	p_{\kappa,\tilde\gamma} = \dfrac{1}{\kappa^3} p_{\gamma}.
	\]
\end{lemma}
\begin{proof}
Based on \eqref{case2},
	 it is straightforward to verify that  
\[
	-\kappa \Delta T_{\kappa,\tilde\gamma} + {\bv}_{\kappa,\tilde\gamma} \cdot \nabla T_{\kappa,\tilde\gamma}
	=-\Delta T_{\gamma}+{\bv}_{\gamma} \cdot \nabla T_{\gamma} = f,
\]
\[
\begin{split}
	&- \kappa \Delta q_{\kappa,\tilde\gamma} - {\bv}_{\kappa,\tilde\gamma} \cdot \nabla  q_{\kappa,\tilde\gamma} 
=
	\dfrac{1}{\kappa}
	 \left(-\Delta q_{\gamma} - {\bv}_{\gamma} \cdot \nabla  q_{\gamma}  \right)
	=
	\dfrac{1}{\kappa}D^*DT_\gamma
	=D^*DT_{\kappa, \tilde \gamma},
\end{split}
\]
and
\[
	-\tilde\gamma\Delta {\bv}_{\kappa,\tilde\gamma} + \nabla p_{\kappa,\tilde\gamma} = 
	\dfrac{1}{\kappa^3} \left( - \gamma \Delta {\bv}_{\gamma} + \nabla p_{\gamma} \right)
	= \dfrac{1}{\kappa^3} (q_{\gamma} \nabla  T_{\gamma})
	=q_{\kappa, \tilde\gamma} \nabla  T_{\kappa, \tilde\gamma}.
\]
This completes the proof.
\end{proof}

As a byproduct of the above lemma, we also have the following result
\[
	J(\kappa, \tilde \gamma) = \dfrac{1}{\kappa^2} J(\gamma),
\]
and therefore,
\[
	\dfrac{\log(J(\kappa,\tilde\gamma_1)/J(\kappa,\tilde\gamma_2))}{\log(\tilde\gamma_1/\tilde\gamma_2)} = 
	\dfrac{\log(J(\gamma_1)/J(\gamma_2)}{\log(\gamma_1/\gamma_2)}. 
\]

\subsection{Finite Element Formulation}
The weak formulation for the nonlinear system \eqref{case2}  is to find $T \in H_0^1(\Omega), q \in H_0^1(\Omega),
\bv  \in  [H_0^1(\Omega)]^2$ and $p \in L^2(\Omega)$ such that:
\begin{eqnarray}\label{weak-formulation}
\begin{cases}
(\kappa\nabla T, \nabla \phi) +  ( \bv\cdot\nabla T, \phi)   = (f,\phi),\quad \forall \, \phi\in H_0^1,\\
(\kappa\nabla q, \nabla \psi) -  ( \bv\cdot\nabla q, \psi)   - (DT,\phi) =0,\quad \forall \, \psi\in H_0^1,\\
(\gamma \nabla \bv_h,\bw) - (p, \nabla \cdot \bw)- (q\nabla T,\bw) = 0, \quad\forall  \, \bw\in [H_0^1(\Omega)]^2,\\
(\nabla \cdot \bv,\theta) = 0,\quad\forall  \, \theta\in L^2(\Omega).
\end{cases}
\end{eqnarray}
We aim to use finite element method to approximate the system. 
Let $\mathcal{T}_h$ be a partition of the domain $\Omega$ consisting of triangles in two dimensions. 
For every element $\tau\in\mathcal{T}_h$, we denote by $h_\tau$ its diameter and define the mesh size $h=\max\limits_{\tau\in\mathcal{T}_h}h_\tau$ for $\mathcal{T}_h$. On the mesh $\mathcal{T}_h$, we define the continuous finite element spaces as follows,
\begin{eqnarray*}
V_h &=& \{v\in H^1(\Omega): v|_\tau\in\mathbb{P}_2(\tau),\forall \tau\in\mathcal{T}_h\},\\
\textbf{V}_h &=&\{\bv\in [H^1(\Omega)]^2: \bv|_\tau\in[\mathbb{P}_2(\tau)]^2,\forall \tau\in\mathcal{T}_h\},\\
Q_h &=& \{q\in H^1(\Omega)\cap L_0^2(\Omega): q|_\tau\in\mathbb{P}_1(\tau),\forall \tau\in\mathcal{T}_h\}.
\end{eqnarray*}
Here $\mathbb{P}_\ell$ denotes the space of polynomials with degree less than or equal to $\ell$ and $L_0^2(\Omega):=\{\theta\in L^2(\Omega): \int_{\Omega}\theta d\bx = 0\}$.  The corresponding finite element spaces with homogeneous Dirichlet boundary condition are denoted by $V_h^0$ and $\textbf{V}_h^0.$ For the Stokes solver, we  apply the inf-sup stable Taylor-Hood element \cite{BF1991,BS2002}.

Below we introduce the bilinear and trilinear forms. For $\phi,\psi\in V_h$, $\bv,\bw\in \textbf{V}_h$, $\theta\in Q_h$, let
\begin{eqnarray*}
&&\mathcal{A}(\phi,\psi) = \sum_{\tau\in\mathcal{T}_h}\int_\tau\kappa\nabla\phi\cdot\nabla\psi d\bx,\\
&&\mathcal{C}(\bw;\phi,\psi) = \sum_{\tau\in\mathcal{T}_h}\int_\tau(\bw\cdot\nabla\phi)\psi d\bx,\\
&& \mathcal{D}(\bv,\bw) = \sum_{\tau\in\mathcal{T}_h}\int_\tau\gamma\nabla\bv: \nabla\bw d\bx,\\
&& \mathcal{B}(\bw,\theta) = \sum_{\tau\in\mathcal{T}_h}\int_\tau\nabla\cdot\bw\theta d\bx.
\end{eqnarray*}

Now, we are ready to propose the finite element schemes for system  (\ref{case2}) with $D^*DT = T-\langle T\rangle$.
The finite element scheme for the system (\ref{case2}) is to solve:
$T_h\in V_h^0$, $q_h\in V_h^0$, $\bv_h\in\textbf{V}_h^0$ and $p_h\in Q_h$, such that:
\begin{eqnarray}\label{FEM}
\begin{cases}
\mathcal{A}(T_h,\phi)-\mathcal{C}(\bv_h;T_h,\phi) = (f,\phi),\quad \forall \, \phi\in V_h^0,\\
\mathcal{A}(q_h,\psi)+\mathcal{C}(\bv_h;q_h,\psi) - (T_h-\langle T_h\rangle,\psi) = 0,\quad\forall  \, \psi\in V_h^0,\\
\mathcal{D}(\bv_h,\bw) -\mathcal{B}(\bw,p_h)- (q_h\nabla T_h,\bw) = 0, \quad\forall  \, \bw\in\textbf{V}_h^0,\\
\mathcal{B}(\bv_h,\theta) = 0,\quad\forall  \, \theta\in Q_h.
\end{cases}
\end{eqnarray}

\subsection{Picard and Newton iterative Solvers}
Note that \eqref{FEM} is a nonlinear system {involving a  Stokes problem}.  To tackle the nonlinearity, we combine both the Picard and Newton iterative solvers to achieve the required computational efficiency.

For the Picard iterative method, we seek to find $(T^{k+1}, q^{k+1},\bv^{k+1},p^{k+1})$ based on the previously given approximation $(T^{k}, q^{k},\bv^{k},p^{k})$. The idea simply replaces the unknown nonlinear terms by the
known solutions in the previous step. 
The nonlinear system can be linearized as follows:
\begin{eqnarray}\label{picard}
\begin{cases}
	-\kappa \Delta T^{k+1} +{\bv^k} \cdot \nabla T^{k+1}=f,\text{ and }T^{k+1}|_{\partial\Omega} = 0,\\
	-\kappa \Delta q^{k+1}-{\bv^k}  \cdot \nabla   q^{k+1}  = T^{k+1} - \dfrac{1}{|\Omega|}\int_{\Omega}T^{k+1}, 	\text{ and }q^{k+1}|_{\partial\Omega} = 0,\\
-\gamma\Delta {\bv^{k+1}}+\nabla p^{k+1} = q^{k+1} \nabla  T^{k+1},\quad \nabla\cdot \bv^{k+1} = 0, \mbox{ and } \bv^{k+1}|_{\partial\Omega} = 0.
\end{cases}
\end{eqnarray}
The finite element solution to \eqref{picard} is then to find  $(T_h^{k+1}, q_h^{k+1},\bv_h^{k+1},p_h^{k+1}) \in
V_h^0 \times V_h^0 \times \mathbf{V}_h^0 \times Q_h$ such that
\begin{eqnarray}\label{picard-FEM}
\begin{cases}
\mathcal{A}(T_h^{k+1},\phi)-\mathcal{C}(\bv_h^k;T_h^{k+1},\phi) = (f,\phi),\quad \forall \, \phi\in V_h^0,\\
\mathcal{A}(q_h^{k+1},\psi)+\mathcal{C}(\bv_h^{k};q_h^{k+1},\psi) - (T_h^{k+1}-\langle T_h^{k+1}\rangle,\psi) = 0,\quad\forall  \, \psi\in V_h^0,\\
\mathcal{D}(\bv_h^{k+1},\bw) -\mathcal{B}(\bw,p_h^{k+1})- (q_h^{k+1}\nabla T_h^{k+1},\bw) = 0, \quad\forall  \, \bw\in\textbf{V}_h^0,\\
\mathcal{B}(\bv_h^{k+1},\theta) = 0,\quad\forall  \, \theta\in Q_h.
\end{cases}
\end{eqnarray}
Note that the system \eqref{picard-FEM} can be solved sequentially. For the Picard's method in the finite element scheme, we set the following initial guess:
 $( T_h^0, q_h^0,\bv_h^0, p_h^0)$ such that
\begin{equation}\label{initial-guess}
\begin{cases}
\begin{split}
	&\bv_h^0 =0, \quad p_h^0=0,\\
	&\mathcal{A}(T_h^{0},\phi)= (f,\phi),\quad \forall \, \phi\in V_h^0,\\
	&\mathcal{A}(q_h^0,\psi) = (T_h^0-\langle T_h^0\rangle,\psi) = 0,\quad\forall  \, \psi\in V_h^0.
\end{split}
\end{cases}
\end{equation}

We now derive the formulation for the Newton's method in the PDE level.
Given an approximation to the solution field, 
$\{T^k, q^k, \bv^k,p^k\}$, we aim to find a perturbation 
$\{\delta T, \delta q, \delta \bv,\delta p\}$ so that
\[
    \{T^{k+1}, q^{k+1}, \bv^{k+1},p^{k+1}\} =
    \{T^k, q^k, \bv^k,p^k\} + \{\delta T, \delta q, \delta \bv,\delta p\}.
\]
and that
\begin{equation*}
\begin{cases}
	-\kappa \Delta T^{k+1}+{\bv}^{k+1} \cdot \nabla T^{k+1}=f, 
	\forall x\in \Omega,\text{ and }T^{k+1}|_{\partial\Omega} = 0,\\
	-\kappa \Delta q^{k+1}-{\bv}^{k+1}  \cdot \nabla   q^{k+1} 
	  - T^{k+1} + \langle T^{k+1}\rangle = 0,
	\forall x\in \Omega	\text{ and }q^{k+1}|_{\partial\Omega} = 0,\\
-\gamma\Delta {\bv}^{k+1} +\nabla p^{k+1} - q^{k+1} \nabla  T^{k+1}= 0,\  \nabla\cdot\bv^{k+1}|_{\Omega} = 0
\quad \forall x\in\Omega \mbox{ and } \bv^{k+1}|_{\partial\Omega} = 0.
\end{cases}
\end{equation*}
This above PDE system is  still a nonlinear system.
The idea to obtain a linear system is to assume that $\delta\cdot$ quantities are sufficiently small so that we can linearize the problem  with respect to those $\delta\cdot$ quantities using Taylor's expansion. 
Eventually we obtain the following linear system by dropping the higher order nonlinear terms in terms of $\delta\cdot$ quantities.
\begin{eqnarray}\label{Newton-PDE}
\begin{cases}
-\kappa\Delta T^{k+1} + \bv^{k+1}\cdot\nabla T^{k} +\bv^{k}\cdot\nabla T^{k+1} = f +\bv^k\cdot\nabla T^k,\ T^{k+1}|_{\partial\Omega}=0,\\
-\kappa\Delta q^{k+1} - \bv^{k+1}\cdot\nabla q^k - \bv^{k}\cdot\nabla q^{k+1}-T^{k+1}+\langle T^{k+1}\rangle = -\bv^k\cdot\nabla q^k,\ q^{k+1}|_{\partial\Omega}=0,\\
-\gamma\Delta\bv^{k+1}+\nabla p^{k+1}-q^{k+1}\nabla T^k-q^k\nabla T^{k+1} = -q^k\nabla T^k,
\bv^{k+1}|_{\partial\Omega}=0 \\
\nabla\cdot\bv^{k+1} = 0.
\end{cases}\label{Iter-Div-Free}
\end{eqnarray}
The finite element solution to \eqref{Newton-PDE} is then to find  $(T_h^{k+1}, q_h^{k+1},\bv_h^{k+1},p_h^{k+1}) \in
V_h^0 \times V_h^0 \times \mathbf{V}_h^0 \times Q_h$ such that
\begin{eqnarray}\label{Newton-FEM}
\begin{cases}
\mathcal{A}(T_h^{k+1},\phi)+\mathcal{C}(\bv_h^k;T_h^{k+1},\phi) +\mathcal{C}(\bv_h^{k+1};T_h^{k},\phi) = (f,\phi)+\mathcal{C}(\bv_h^{k};T_h^{k},\phi),\quad \forall \, \phi\in V_h^0,\\
\mathcal{A}(q_h^{k+1},\psi)-\mathcal{C}(\bv_h^{k};q_h^{k+1},\psi) - \mathcal{C}(\bv_h^{k+1};q_h^{k},\psi)  - (T_h^{k+1}-\langle T_h^{k+1}\rangle,\psi) = -\mathcal{C}(\bv_h^{k};q_h^{k},\psi) ,\quad\forall  \, \psi\in V_h^0,\\
\mathcal{D}(\bv_h^{k+1},\bw) -\mathcal{B}(\bw,p_h^{k+1})- (q_h^{k}\nabla T_h^{k+1},\bw) - (q_h^{k+1}\nabla T_h^{k},\bw)= 
- (q_h^{k}\nabla T_h^{k},\bw), \quad\forall  \, \bw\in\textbf{V}_h^0,\\
\mathcal{B}(\bv_h^{k+1},\theta) = 0,\quad\forall  \, \theta\in Q_h.
\end{cases}
\end{eqnarray}

\begin{remark}
Comparing to  Picard's method,  Newton's method has a faster convergence rate. However, its initial condition should be chosen wisely. For Picard's method, our numerical experiments  show that it can yield a satisfactory initial solution for the Newton's method very quickly. This suggests that we can use Picard's method at the first stage to obtain a good initial guess and then apply Newton's method to obtain the  converged numerical solutions.   The numerical experiments presented in the rest of this work are  conducted  using the combined Picard-Newton solver.
\end{remark}

\subsection{Numerical Algorithm}
In this subsection, we summarize our numerical method in the following algorithm.
\begin{algorithm}[h]\label{algorithm}
\begin{itemize}
\item Choose values in $\epsilon_1$, $\epsilon_2$, $n_1$, and $n_2$.
    	\item Set the initial guess $( T_h^0, q_h^0,\bv_h^0, p_h^0)$ as in \eqref{initial-guess}.
        \item Compute the cost functional:
            \begin{eqnarray}
            J_0=\frac{\gamma\|\nabla\bv_h^0\|^2}{2}+\frac{\|T_h^0-\langle{T}_h^0\rangle\|^2}{2}.
            \end{eqnarray}
    
   \item For $k = 0,\dots,n_1$, perform the Picard iteration as below: 
    \begin{itemize}
    \item Solve $(T_h^{k+1}, q_h^{k+1},\bv_h^{k+1},p_h^{k+1}) \in
V_h^0 \times V_h^0 \times \mathbf{V}_h^0 \times Q_h$ for \eqref{picard-FEM}.
        \item Compute the cost functional:
            \begin{eqnarray}
            J_k=\frac{\gamma\|\nabla\bv_h^k\|^2}{2}+\frac{\|T_h^k-\langle{T}_h^k\rangle\|^2}{2}.
            \end{eqnarray}
        \item If $\dfrac{|J_k-J_{k-1}|}{J_{k-1}}<\epsilon_1$, \textbf{STOP} and \textbf{OUTPUT} $T_h^k$, $q_h^k,$ 
        $\bv_h^k$, and $p_h^k$.
       \end{itemize}
       
 \item Set $( T_h^0, q_h^0,\bv_h^0, p_h^0) = ( T_h^k, q_h^k,\bv_h^k, p_h^k).$
 
    \item For $k = 0,\dots,n_2$, perform the Newton's iterations as below:
    \begin{itemize}
        \item Solve $(T_h^{k+1}, q_h^{k+1},\bv_h^{k+1},p_h^{k+1}) \in
V_h^0 \times V_h^0 \times \mathbf{V}_h^0 \times Q_h$  for \eqref{Newton-FEM}.
        \item Compute the cost functional:
            \begin{eqnarray}
            J_k=\frac{\gamma\|\nabla\bv_h^k\|^2}{2}+\frac{\|T_h^k-\langle{T}_h^k\rangle\|^2}{2}.
            \end{eqnarray}
        \item If $\dfrac{|J_k-J_{k-1}|}{J_{k-1}}<\epsilon_2$, \textbf{STOP} and 
        \textbf{OUTPUT} $T_h^k$, $q_h^k,$ $\bv_h^k$, and $p_h^k$.
    \end{itemize}
\end{itemize}
\caption{Finite Element Scheme for system (\ref{case2})}\label{Alg:Div-Free}
\end{algorithm}

\section{Numerical Experiments}\label{Sect:Num}
In this section, we shall present several numerical experiments {by employing different heat source profiles} to validate the proposed numerical schemes in Algorithm~\ref{Alg:Div-Free}. 
The domain for all test problems is set to be the unit square, i.e., $\Omega = (0,1)\times(0,1)$. Thanks to Lemma \ref{lem:sol-k}, it is sufficient to test for one  $\kappa$ value. {Without loss of generality}, we perform all our numerical tests only for $\kappa=1$. 
The numerical experiments are performed using the FENICS package \cite{FENICS} on the uniform triangular mesh with $h= 1/100$.

Recall that  as proven in Corollary  \ref{thm_SOSC}, a local minimizer  can be obtained if the control weight  $\gamma$ is sufficiently large. However,  a large control weight may result in a minor convective effect. Our first example shows that if $\gamma$ is set to be too large, ``doing nothing" might be optimal.

\begin{example}\label{DivFree-Test1}
{We first test a symmetric heat distribution}. Let  
 \[
f(x,y)  = 2\pi^2\sin(\pi x)\sin(\pi y).\] 
\end{example}
Set $ \kappa=1$ and $\gamma=1$.  The stop criterion is met at the ninth iteration as shown in 
Fig.~\ref{fig:divfree-1-T-g1}, where  Fig.~\ref{fig:divfree-1-T-g1}a. presents the optimal temperature distribution and   Fig.~\ref{fig:divfree-1-T-g1}b. presents the cost functional values with respect to $\gamma$ for each iteration.  However,  the cost functional does not seem to decay at all. In this case, $\gamma=1$ may be too large so that the convective effect becomes minor and hence, the thermal diffusion plays  a dominant role. Based on this observation, we proceed to test smaller $\gamma$ values and note that  convection becomes effective  when $\gamma\in$[E-7,E-5].  Using   the optimal convection-cooling design,  the cost functional value can be reduced by about 40\% for the current heat source term. 
 The results are illustrated  in Figs.~\ref{fig:divfree-1-T}-\ref{fig:divfree-1-V-Stream}.
 Moreover,  we also test how the cost functional, the variance of the temperature, and the velocity change  with respect to 
different $\gamma$. The results   are plotted in Fig.~\ref{fig:divfree-1-Conv}a. The corresponding  convergence rates  are plotted, respectively,  in 
Fig.~\ref{fig:divfree-1-Conv}b., which are computed using the   following standard formulas
 \begin{align}
 &   r_J(\gamma_i) = \dfrac{\ln(J(\gamma_{i+1})/J(\gamma_i))}
    {\ln(\gamma_{i+1}/\gamma_i)} \label{r_J}\\
&     r_{T}(\gamma_i) = \dfrac{\ln(\|T(\gamma_{i+1})-\langle T(\gamma_{i+1})\rangle\|_{L^2}/\|T(\gamma_{i})-\langle T(\gamma_{i})\rangle\|_{L^2})}
    {\ln(\gamma_{i+1}/\gamma_i)},\quad  \text{and} \label{r_T} \\
 & r_{\bv}(\gamma_i) = \dfrac{\ln(\gamma_{i+1}\|\nabla \bv(\gamma_{i+1})\|^2_{L^2}
    /\gamma_{i}\|\nabla \bv(\gamma_{i})\|^2_{L^2})}
    {\ln(\gamma_{i+1}/\gamma_i)}.  \label{r_v}
  \end{align} 

\begin{figure}[H]
\centering
\begin{tabular}{cc}
\includegraphics[width=.45\textwidth]{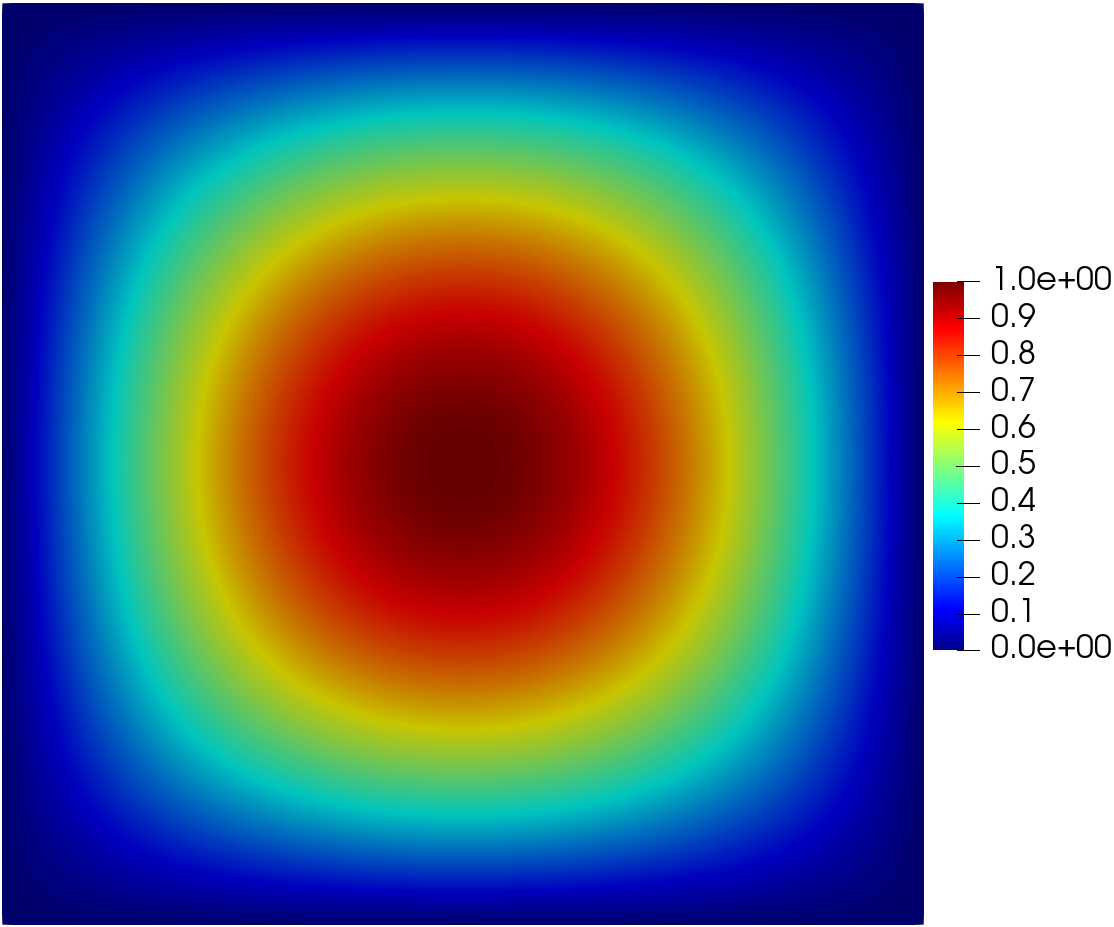}
&\includegraphics[width=.45\textwidth]{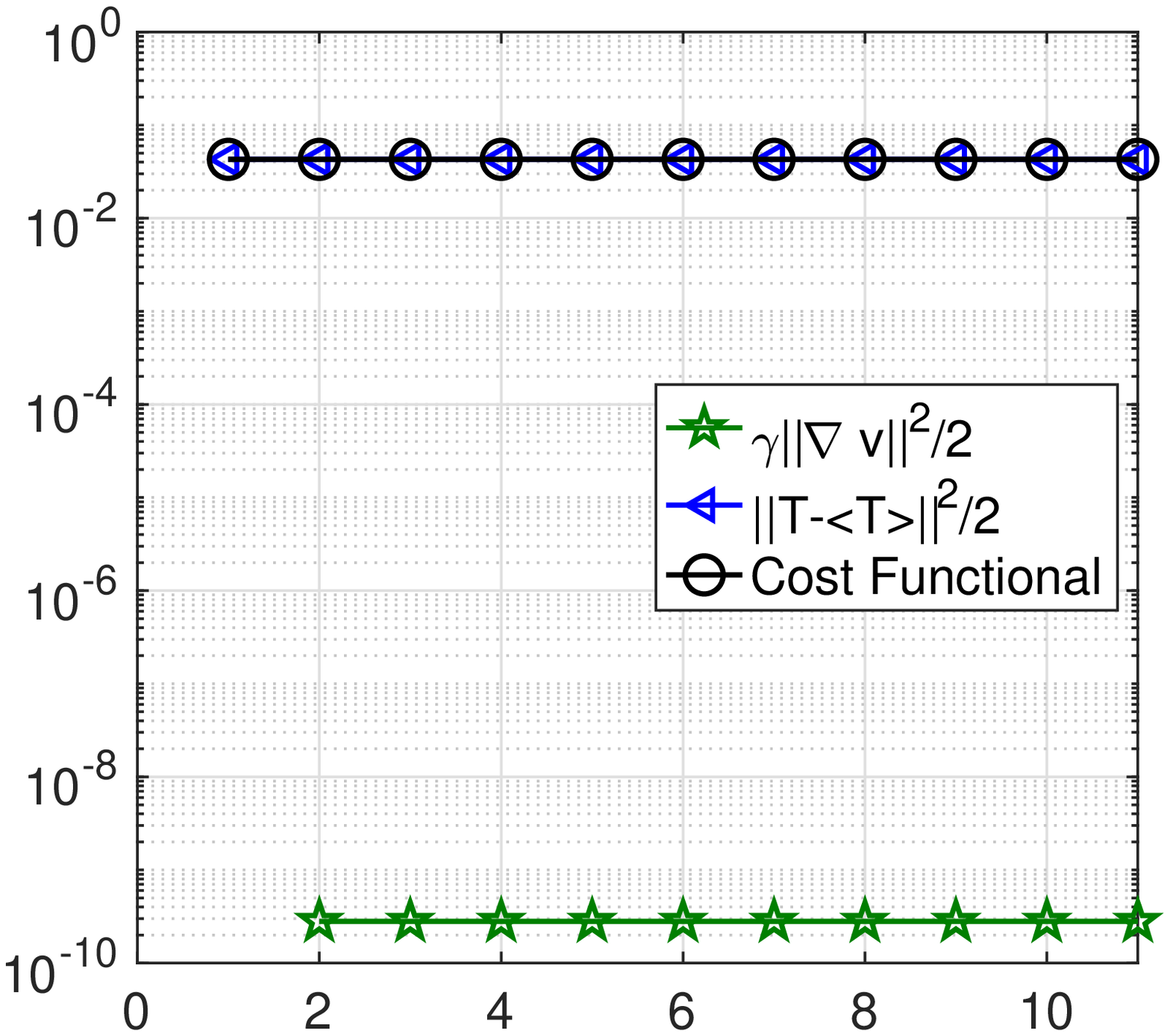}
\end{tabular}
\caption{Example~\ref{DivFree-Test1}: Plots of temperature $T_h$ of $\kappa = 1.0$ and $\gamma = 1.0$ for (a). Optimal heat distribution $T_h^9$; and (b). Convergence profiles for cost functional.}\label{fig:divfree-1-T-g1}
\end{figure}

\begin{figure}[H]
\centering
\begin{tabular}{cc}
\includegraphics[width=.35\textwidth]{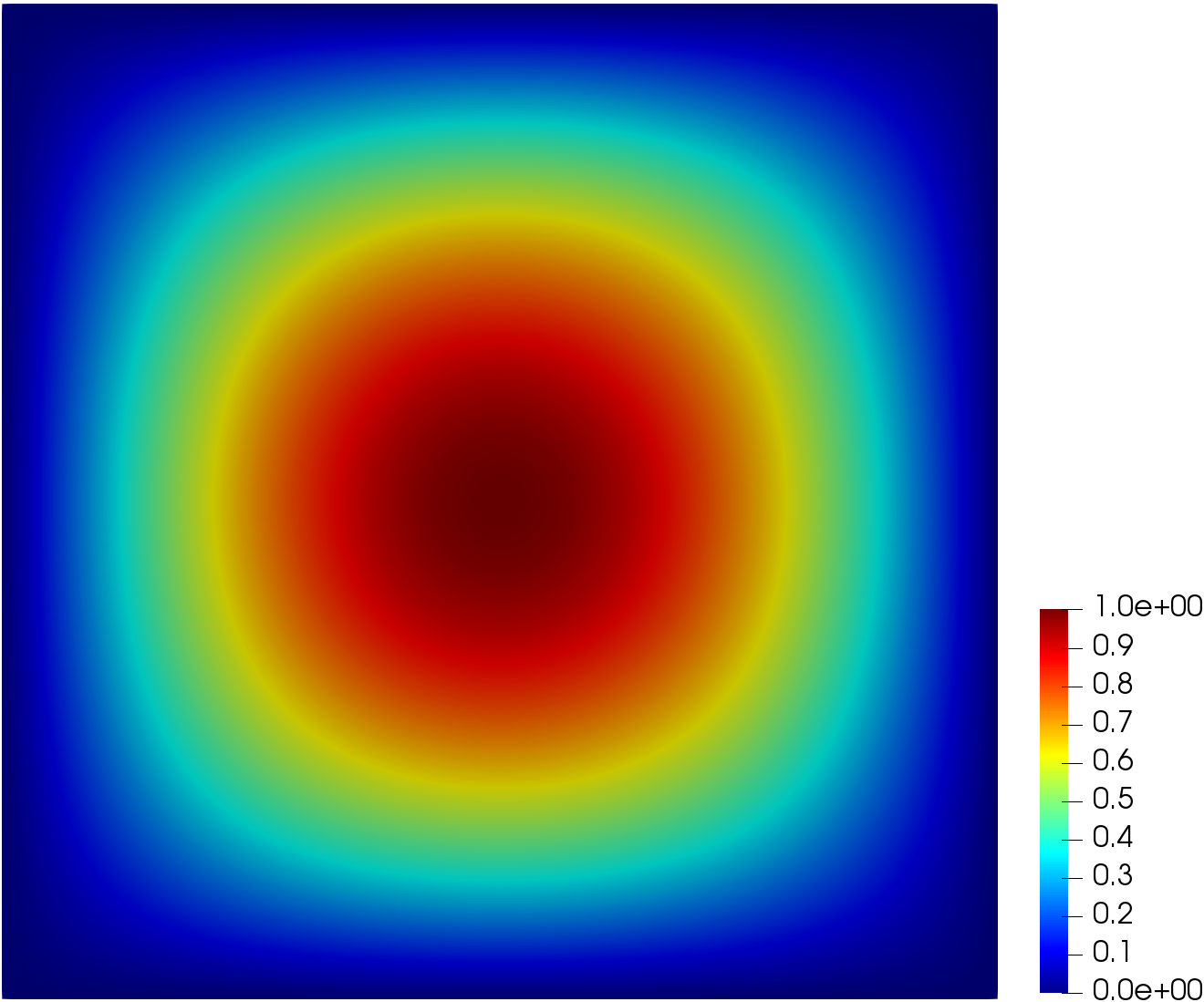}
&\includegraphics[width=.35\textwidth]{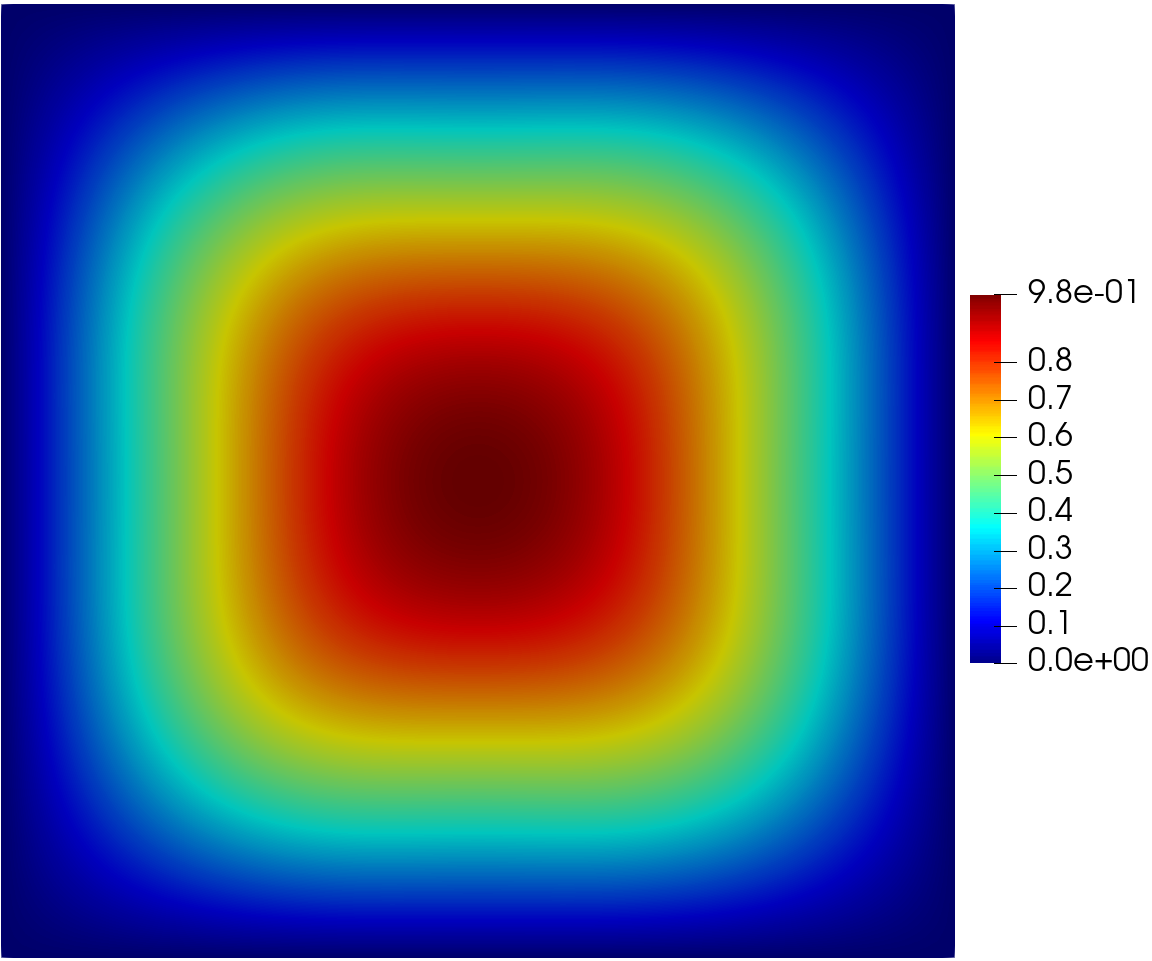}\\
(a) & (b)\\
\includegraphics[width=.35\textwidth]{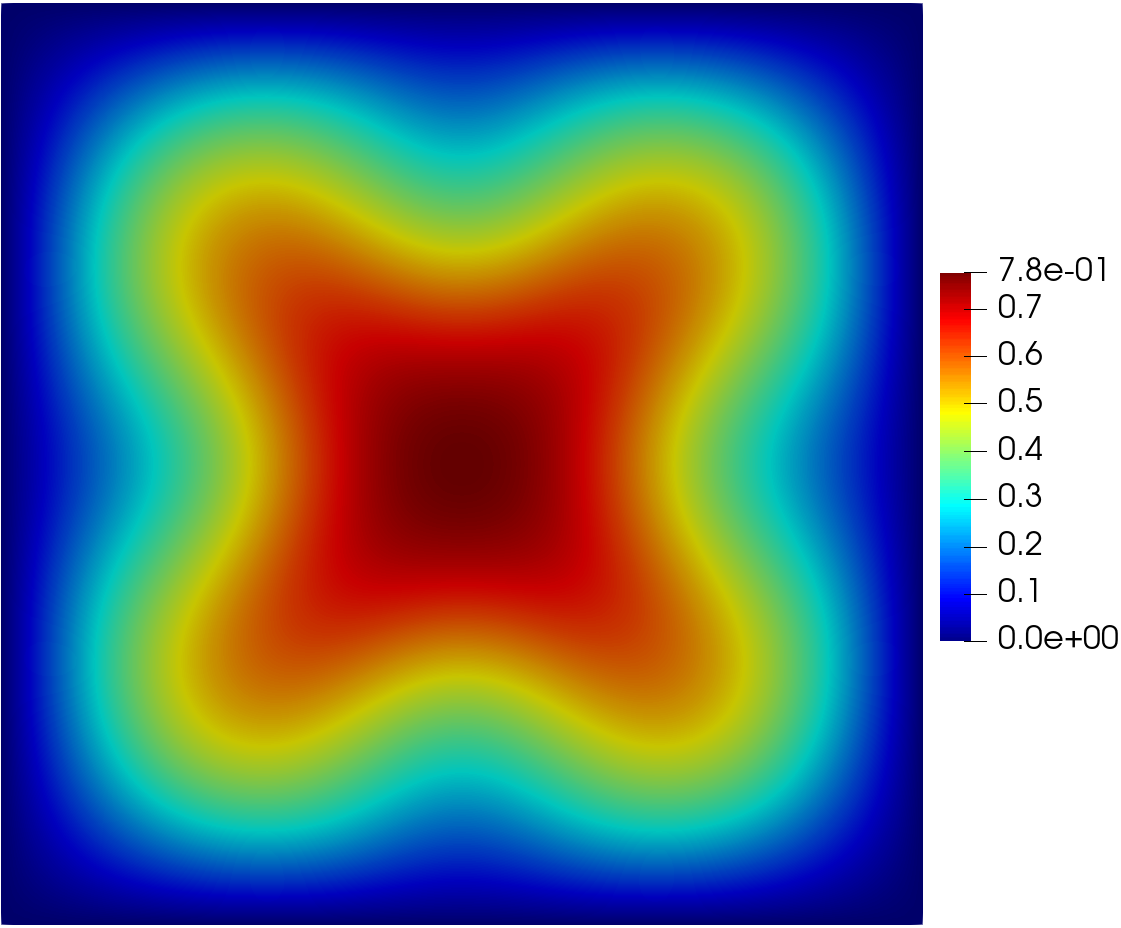}
&\includegraphics[width=.35\textwidth]{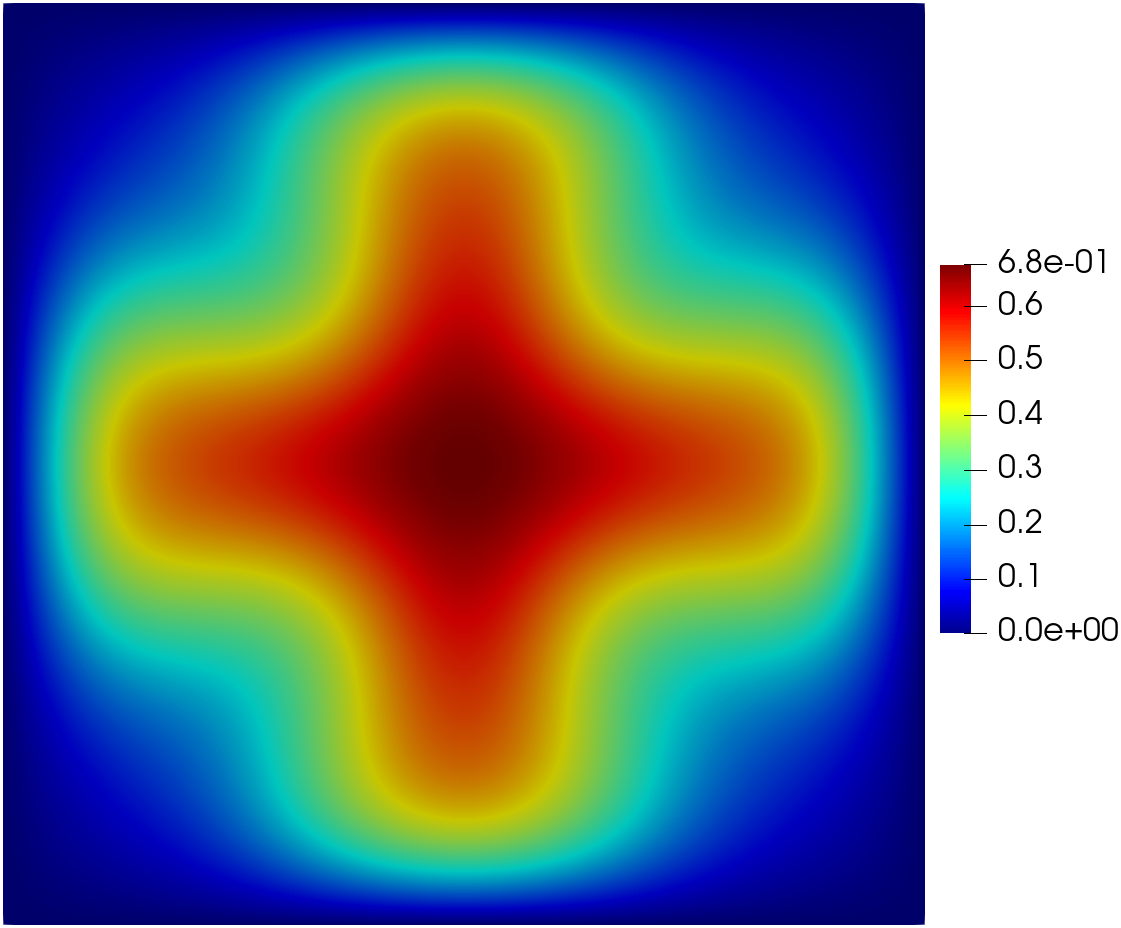}
\\
(c) & (d)\\
\end{tabular}
\caption{Example~\ref{DivFree-Test1}: Plots of temperature $T_h$ of $\kappa = 1.0$ for (a). Initial heat distribution $T_h^0$; and with (b). $\gamma = 3.6$E-6; (c). $\gamma = 8.5$E-7; (d). $\gamma = 3.9$E-7.}\label{fig:divfree-1-T}
\end{figure}



The initial heat distribution $T_h^0$ corresponding to $\bv = 0$ is shown in Fig.~\ref{fig:divfree-1-T}a.
The optimal heat distribution $T_h$ corresponding to $\gamma = 3.6$E-6, 8.5E-7, and 3.9E-7 are plotted in Fig.~\ref{fig:divfree-1-T}b-d. 
For the initial heat distribution, one can observe that the maximum of $T_h^0$ is 1.0.  {Thanks to advection effect,  the ``hot" region, which  is at the center of the domain initially, is now spread  out, but still inherits  certain symmetric pattern. As a result, the heat distribution over the entire domain is evened out.  Note that the maximum of $T_h$ is reduced to 9.8E-1, 7.8E-1, and 6.8E-1 corresponding to $\gamma = 3.6$E-6, 8.5E-7, and $3.9$E-7, respectively. Also, it is shown from these plots that the smaller value in $\gamma$ (which indicates less penalty on the control), the  more effective  is the convection-cooling}.

\begin{figure}[H]
\centering
\begin{tabular}{ccc}
\includegraphics[width=.3\textwidth]{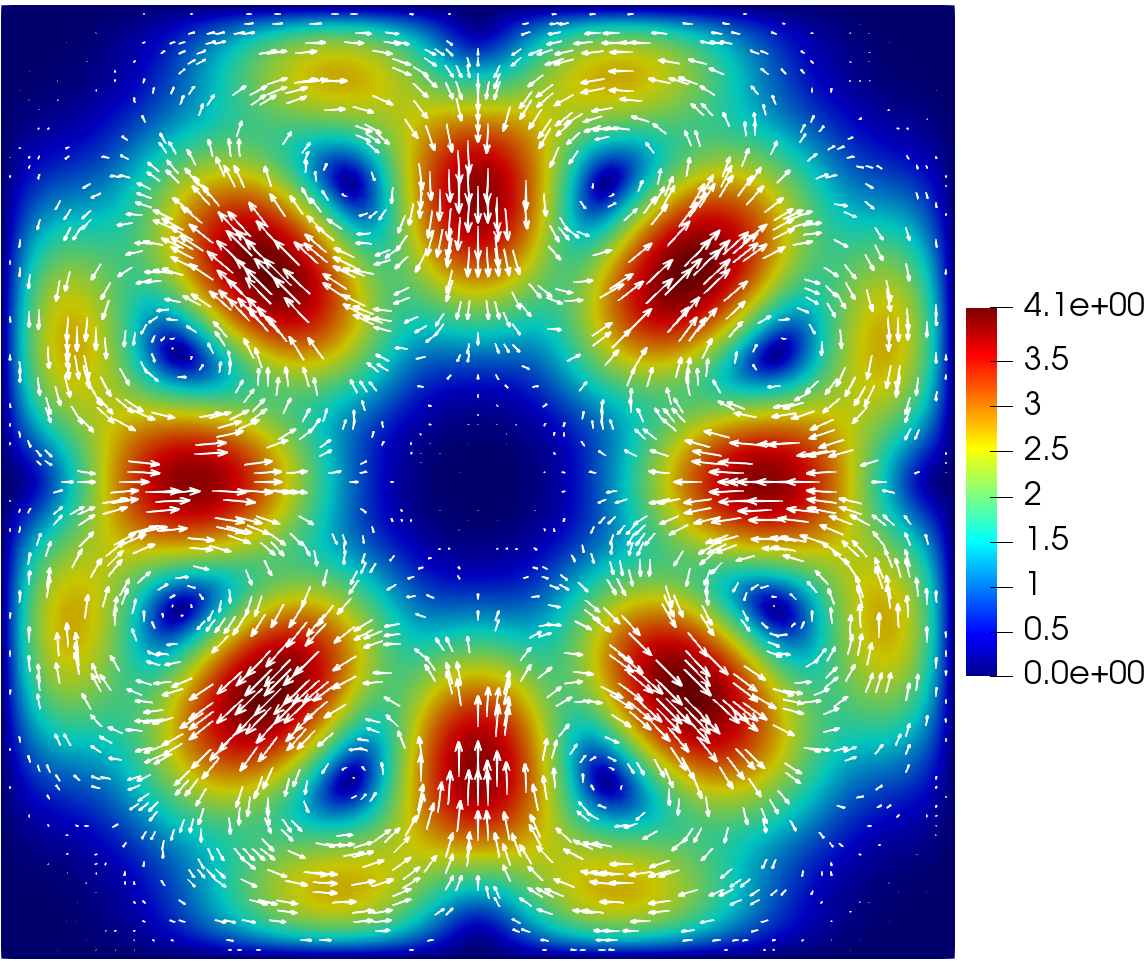}
&\includegraphics[width=.3\textwidth]{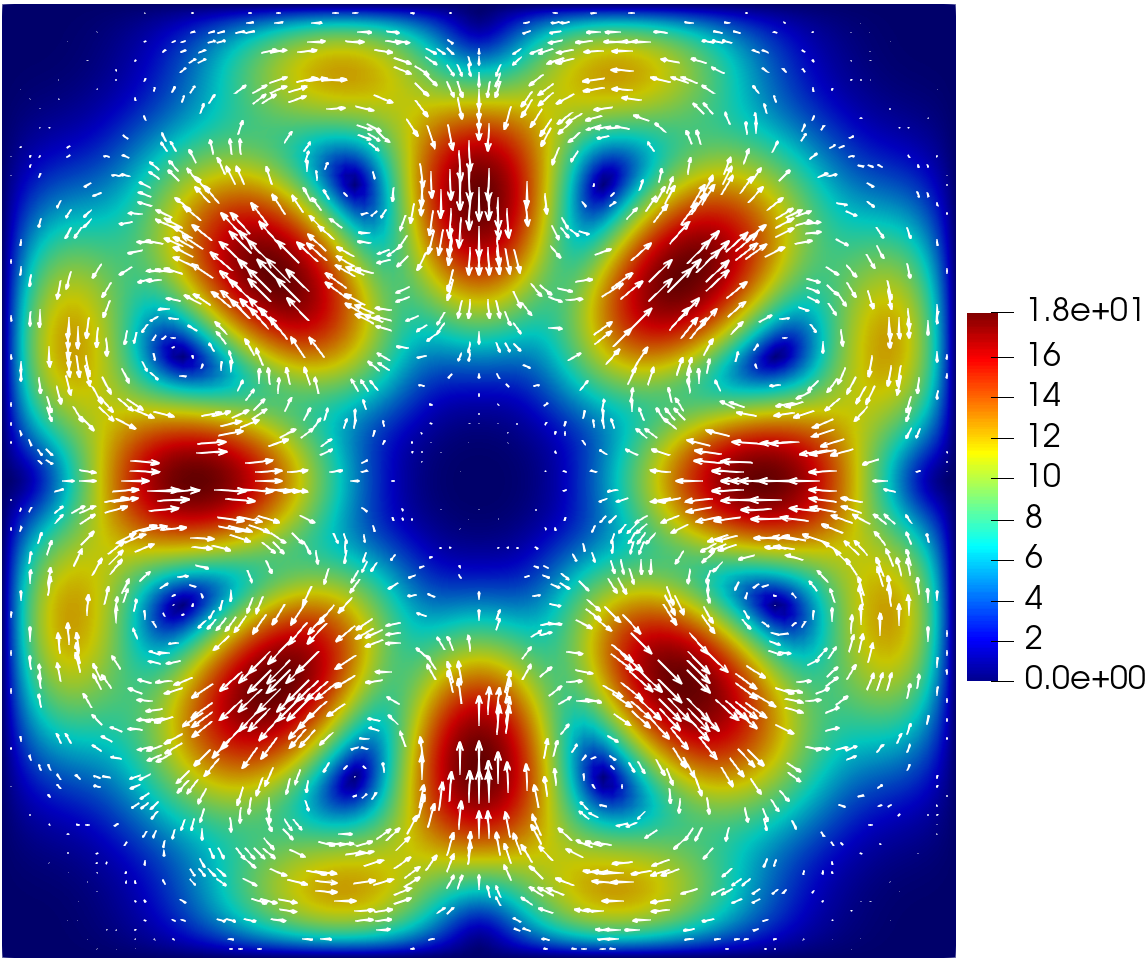}
&\includegraphics[width=.3\textwidth]{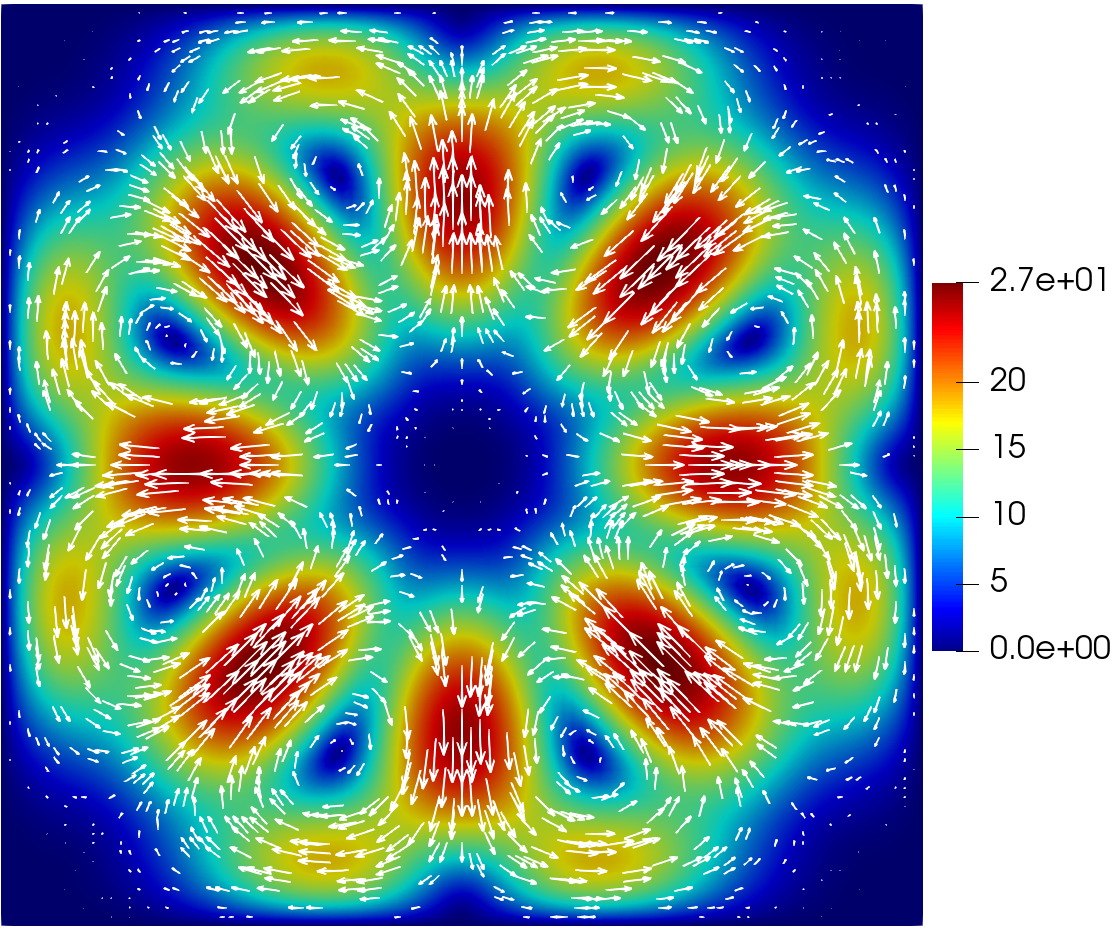}
\\
(a) & (b) &(c)
\end{tabular}
\caption{Example~\ref{DivFree-Test1}: Plots of velocity field ${\bv}_h$ for $\kappa = 1.0$ and (a). $\gamma = 3.6$E-6; (b). $\gamma = 8.5$E-7; (c). $\gamma = 3.9$E-7. Here, the color illustrates the magnitude of velocity $\bv_h$ and the vector plots the field of $\bv_h$.}\label{fig:divfree-1-V}
\end{figure}

\begin{figure}[H]
\centering
\begin{tabular}{ccc}
\includegraphics[width=.3\textwidth]{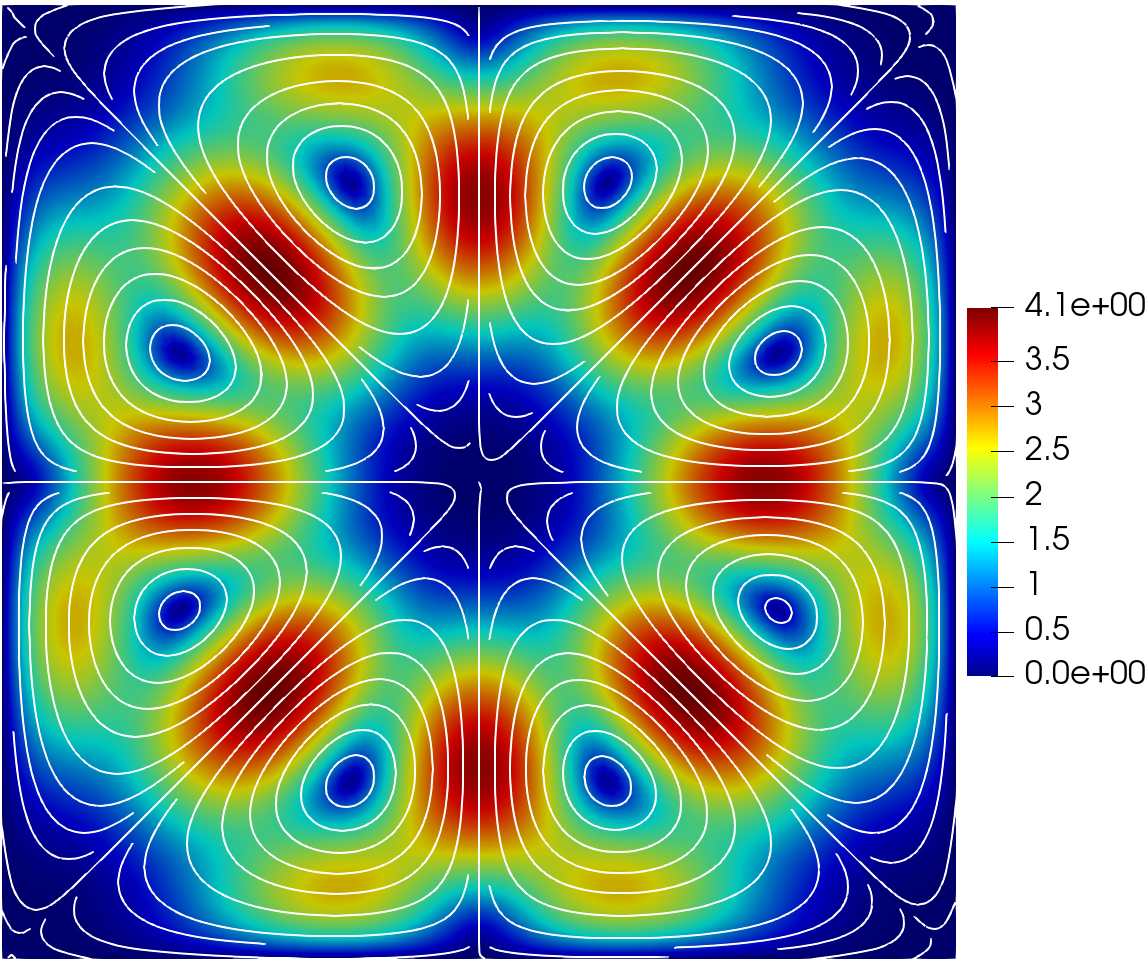}
&\includegraphics[width=.3\textwidth]{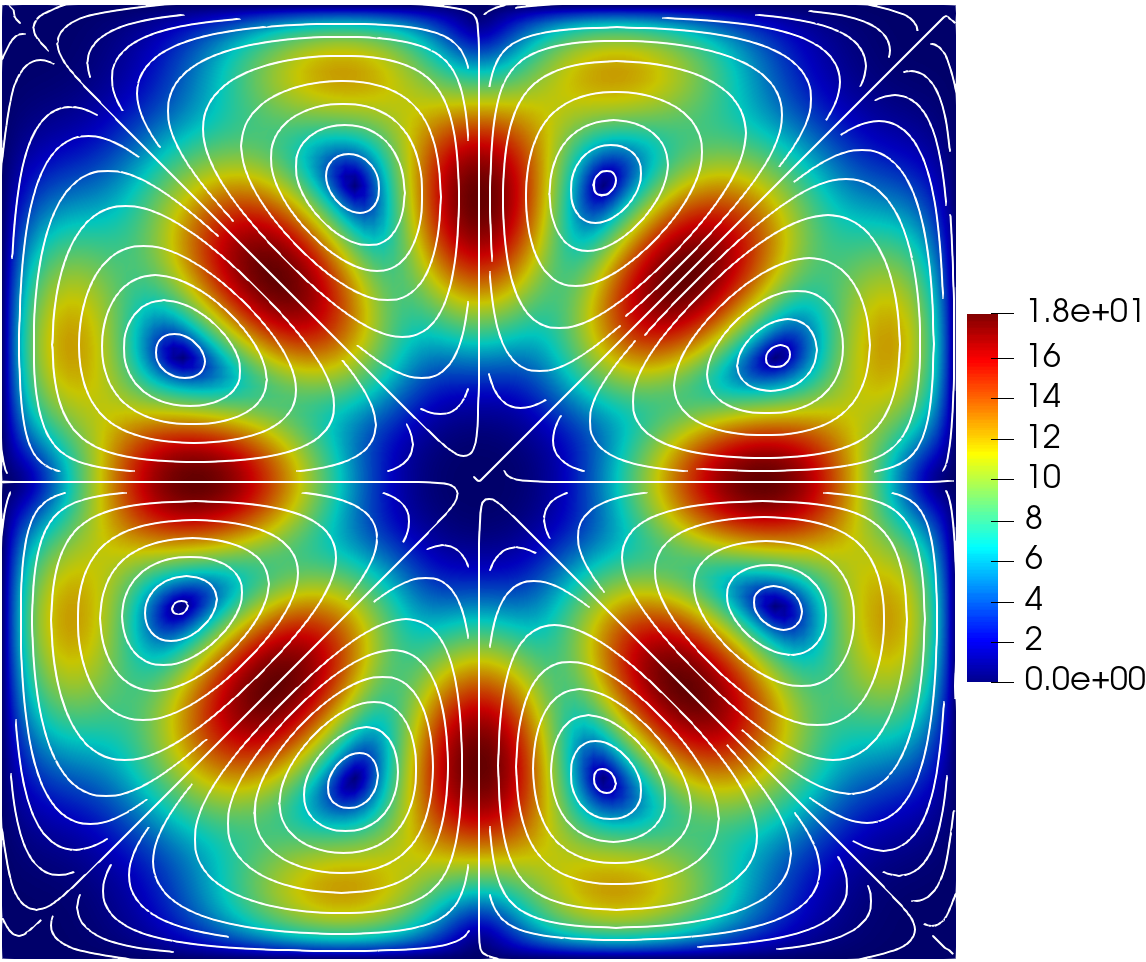}
&\includegraphics[width=.3\textwidth]{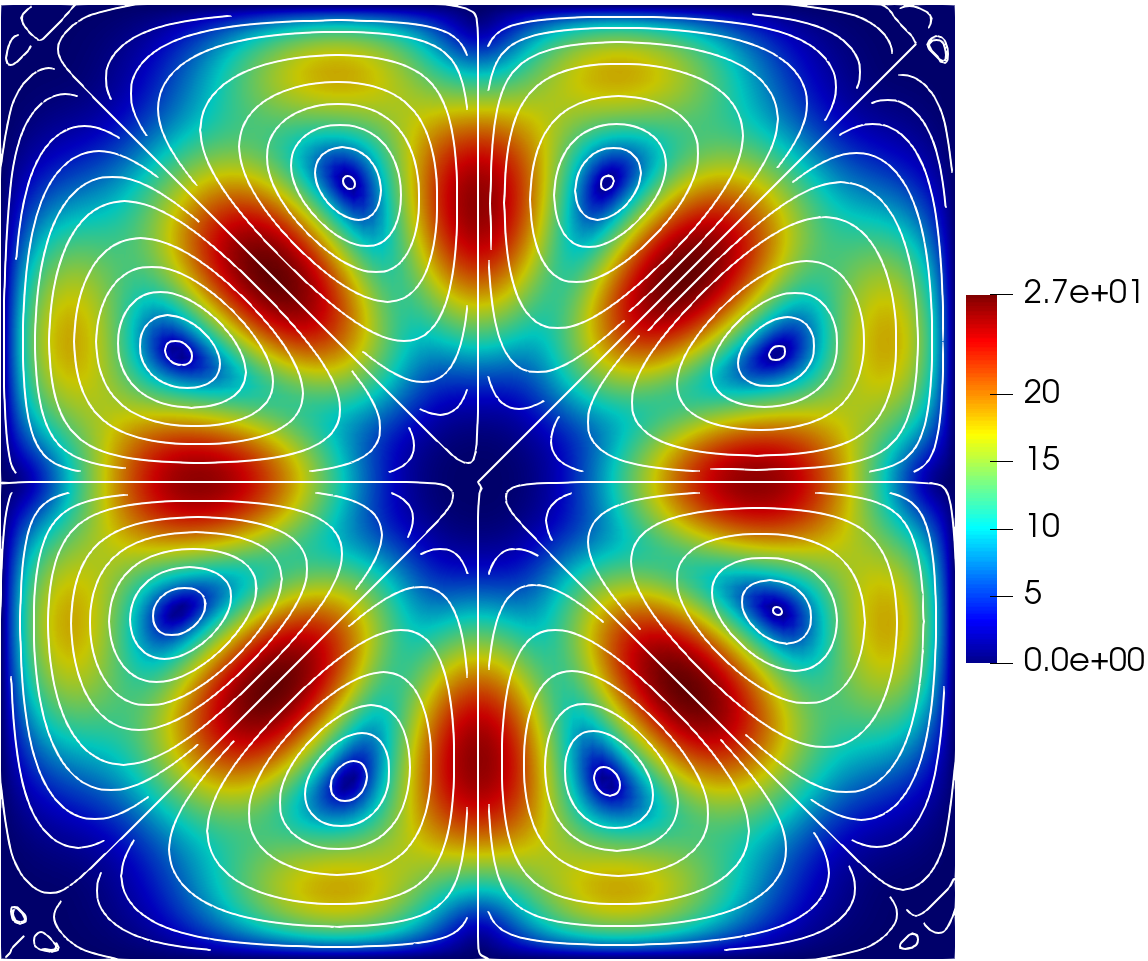}
\\
(a) & (b) &(c)
\end{tabular}
\caption{Example~\ref{DivFree-Test1}: Plots of streamlines of $\bv_h$ for $\kappa = 1.0$ and (a). $\gamma = 3.6$E-6; (b). $\gamma = 8.5$E-7; (c). $\gamma = 3.9$E-7. Here, the color illustrates the magnitude of velocity $\bv_h$ and the curve plots the streamline of $\bv_h$.}\label{fig:divfree-1-V-Stream}
\end{figure}

{On the other hand, as shown in Figs.~\ref{fig:divfree-1-V}-\ref{fig:divfree-1-V-Stream}, the optimal velocity  fields $\bv_h$ and their streamlines  computed by our algorithm for different $\gamma$ well preserve the  divergence-free condition and also present symmetric patterns. This also explains the symmetric pattern of the temperature distribution shown in Fig.~\ref{fig:divfree-1-T}.  Moreover,   the patterns for $\bv_h$ are very similar for different $\gamma$ values. However, the magnitude of $\bv_h$ increases as the $\gamma$ value decreases. }

\begin{figure}[H]
\centering
\begin{tabular}{cc}
\includegraphics[width=.45\textwidth]{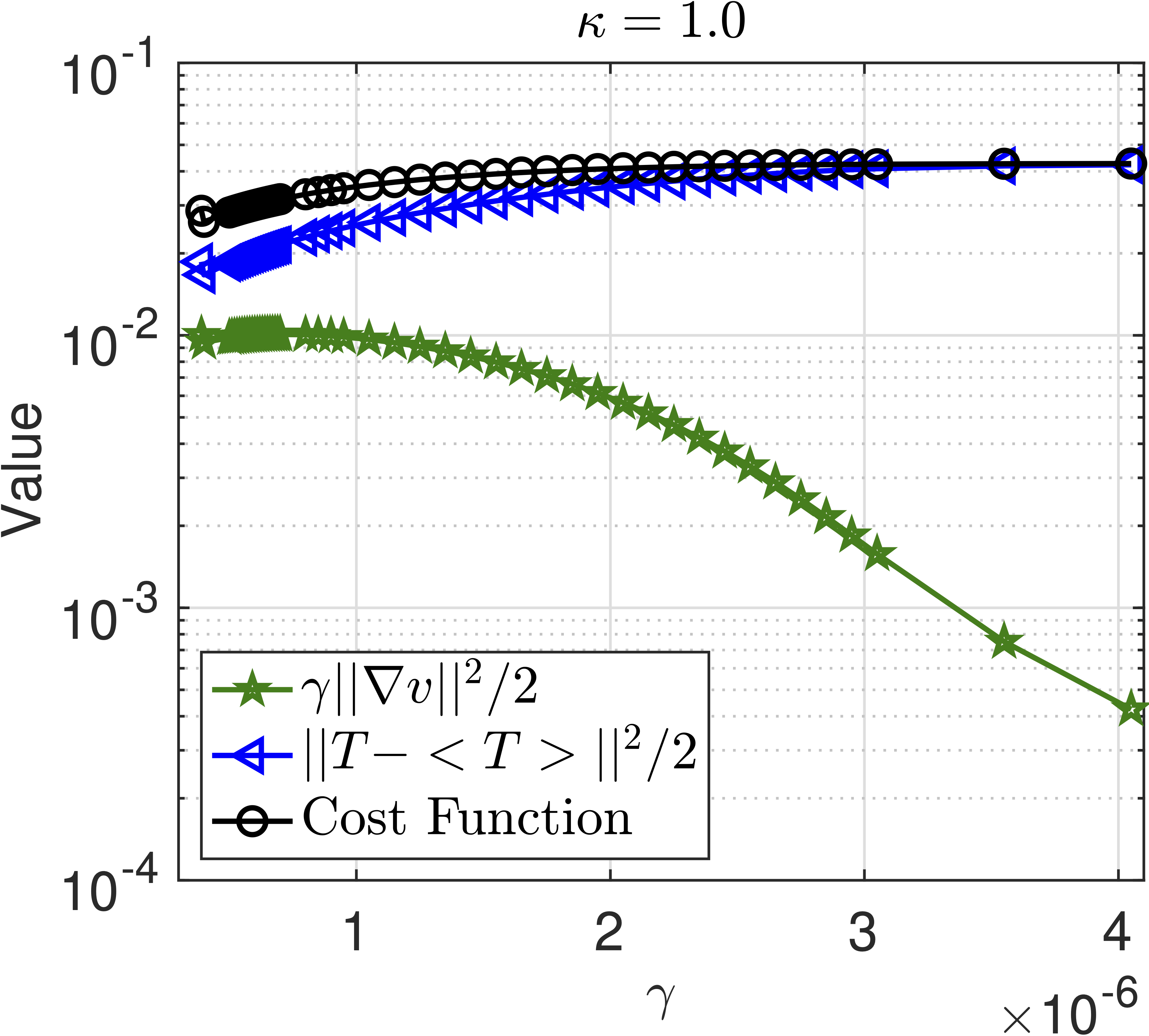}
&\includegraphics[width=.45\textwidth]{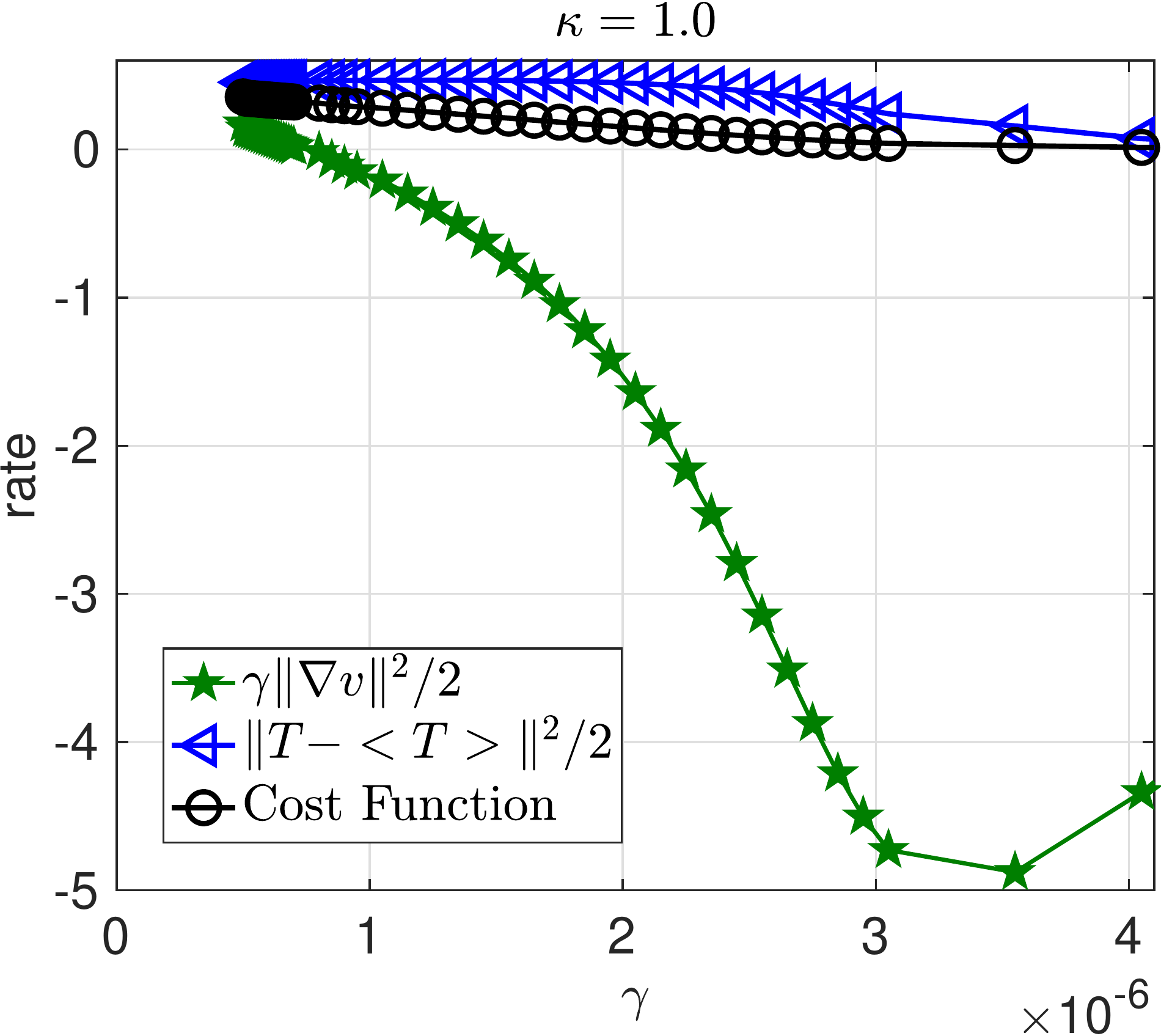}
\\
(a) & (b) 
\end{tabular}
\caption{Example~\ref{DivFree-Test1}: Illustration of results for $\kappa= 1.0$ 
(a). Plot of profiles in the cost functional with respect to $\gamma$ (here $\|T_h^0-\langle T_h^0\rangle\|^2/2 = $4.287E-2); 
(b). Convergence  rates $r_J, r_T$ and $r_{\bv}$ computed by \eqref{r_J}--\eqref{r_v}}. \label{fig:divfree-1-Conv}
\end{figure}
Next, we investigate  the behavior of the cost functional  with respect to $\gamma\in$[3.9E-7, \,4.1E-6]. 
In Fig.~\ref{fig:divfree-1-Conv}a, we plot the cost values versus various $\gamma$ values. It  shows that smaller values in $\gamma$ lead to  smaller cost functional  values. When $\gamma = 4$E-7, we obtain $J_{\min} = 2.60$E-2, which is $39\%$ smaller than the initial value (which is $4.287$E-2).  
In Fig.~\ref{fig:divfree-1-Conv}b, we plot the convergence rates $r_J, r_T$ and $r_{\bv}$ computed by \eqref{r_J}--\eqref{r_v}.
In particular, it can be seen that the convergence rate $r_{J}$ gradually decreases from $0.35$ to almost $0$ as  increasing the values in $\gamma$.

\begin{remark}
We have tested different $\kappa$ and mesh sizes  $h$ for Example \ref{DivFree-Test1} to demonstrate  the numerical robustness, where  different initial guesses for velocity  are also tested.  
 The numerical results are robust on $\kappa$ and refined $h$ for almost all $\gamma$ in the active region.    To reduce the redundancy of the figures, they are  omitted  in the paper.  However, the performance is slightly different when  $\gamma$ is close to its lower limit. This is likely due to the fact that  the continuous problem may fail to have the existence of an optimal control  when $\gamma=0$. In this case, the cost functional loses its coercivity in the control input. 
 \end{remark}

\begin{example}\label{sect:divfree-2}
In this example, we consider an asymmetric  distribution of the hear source.
Let 
\[ f(x,y) = 1000((x-0.5)^2+(y-0.75)^2)x(1-x)y(1-y).\] 
\end{example}
\begin{figure}[H]
\centering
\begin{tabular}{cc}
\includegraphics[width=.35\textwidth]{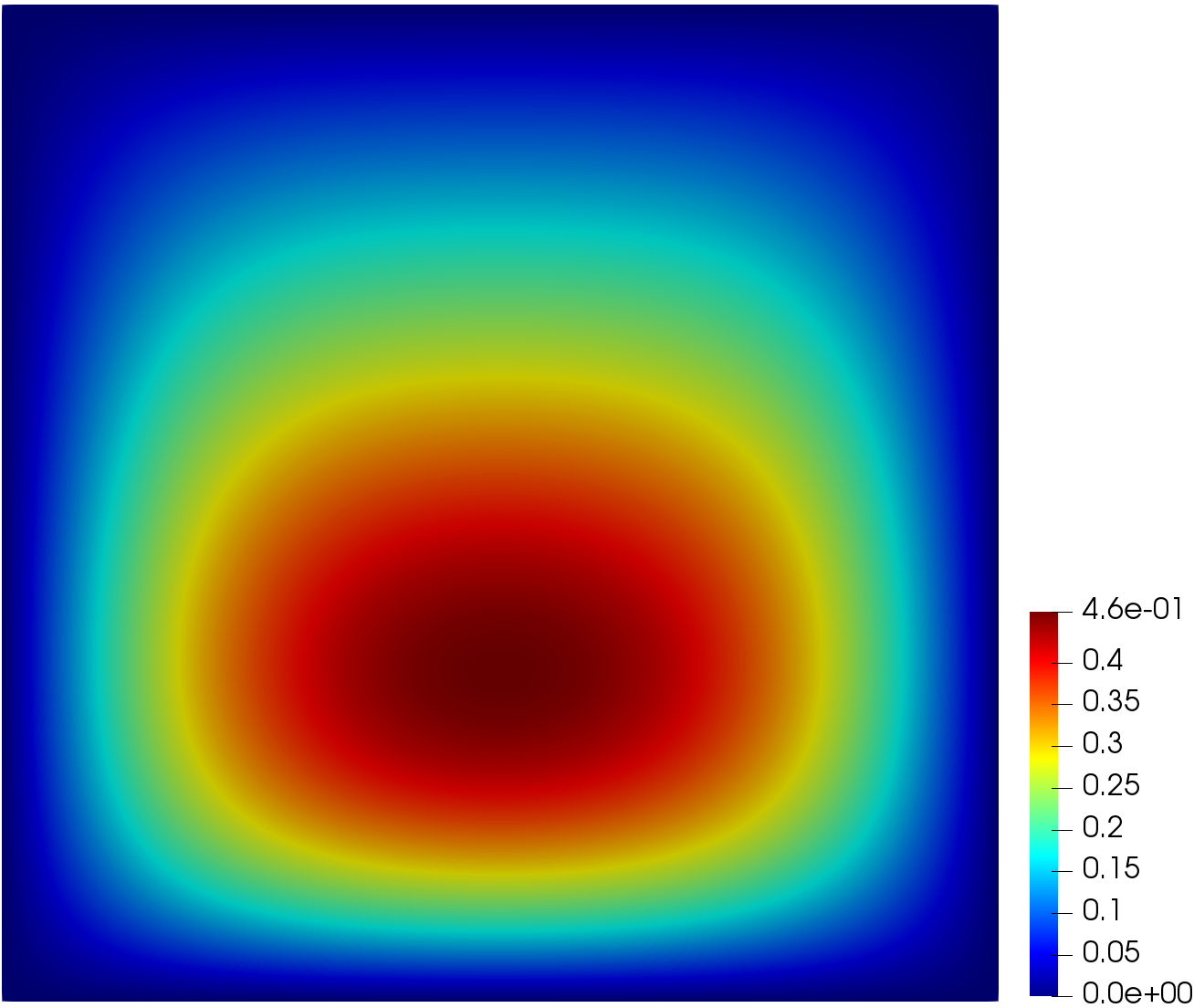}
& \includegraphics[width=.35\textwidth]{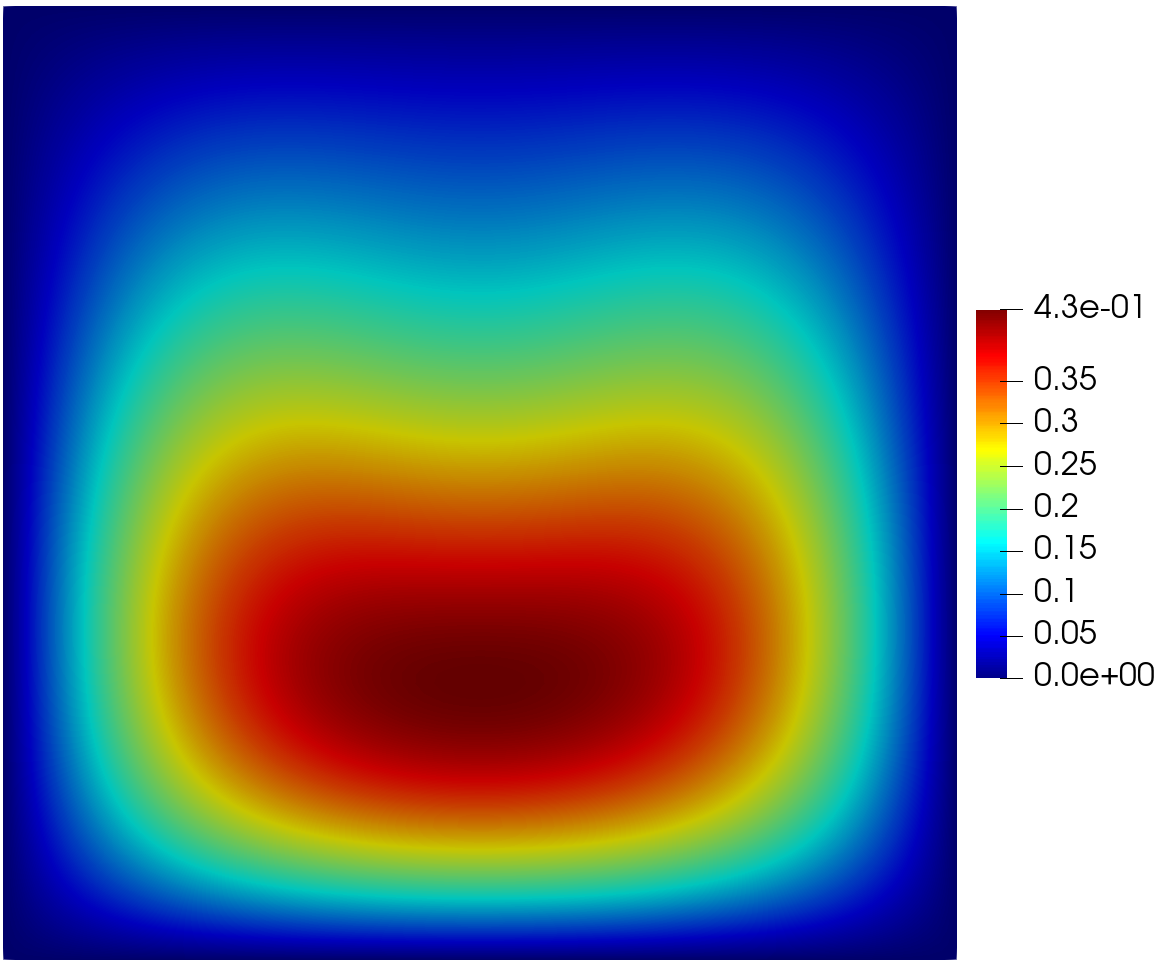}\\
(a) & (b)  \\
\includegraphics[width=.35\textwidth]{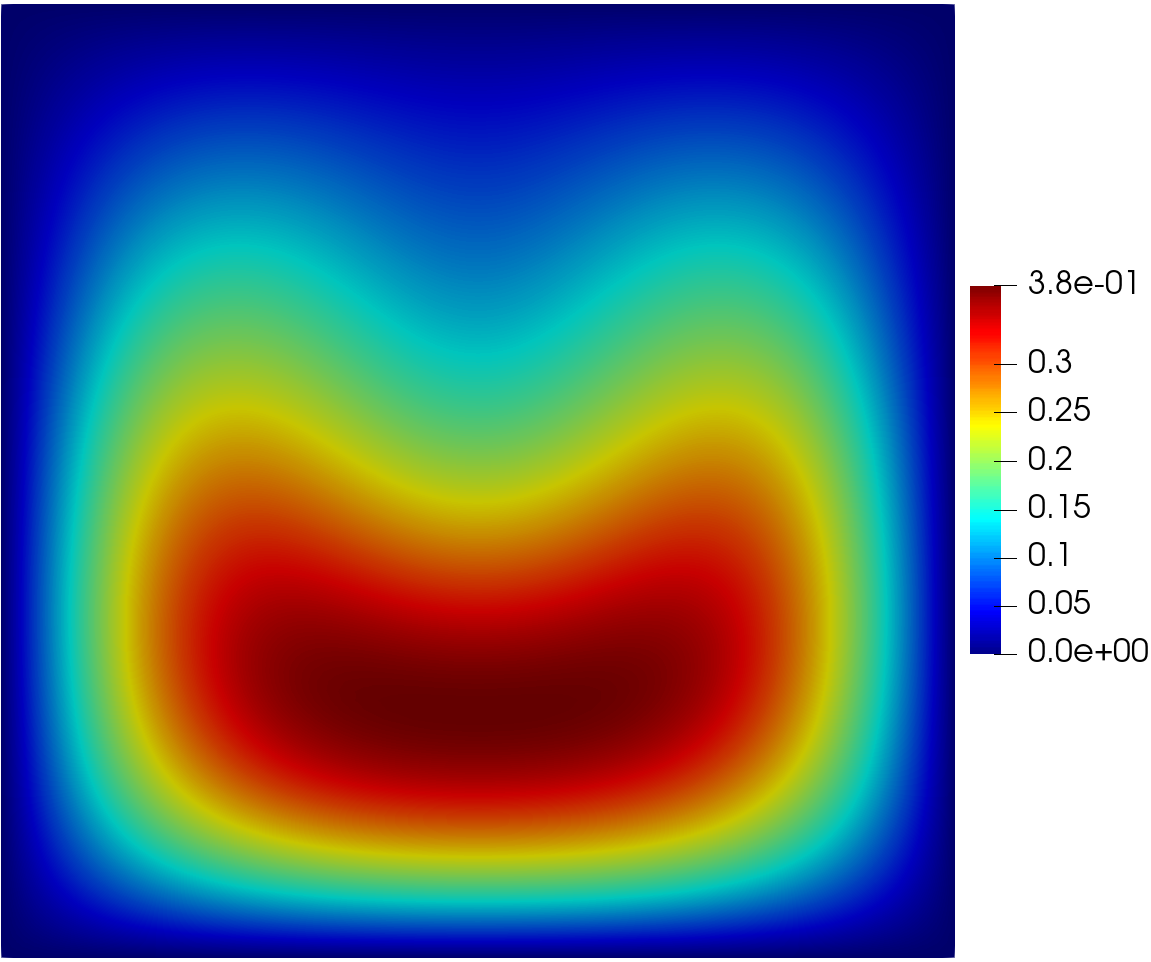}
&\includegraphics[width=.35\textwidth]{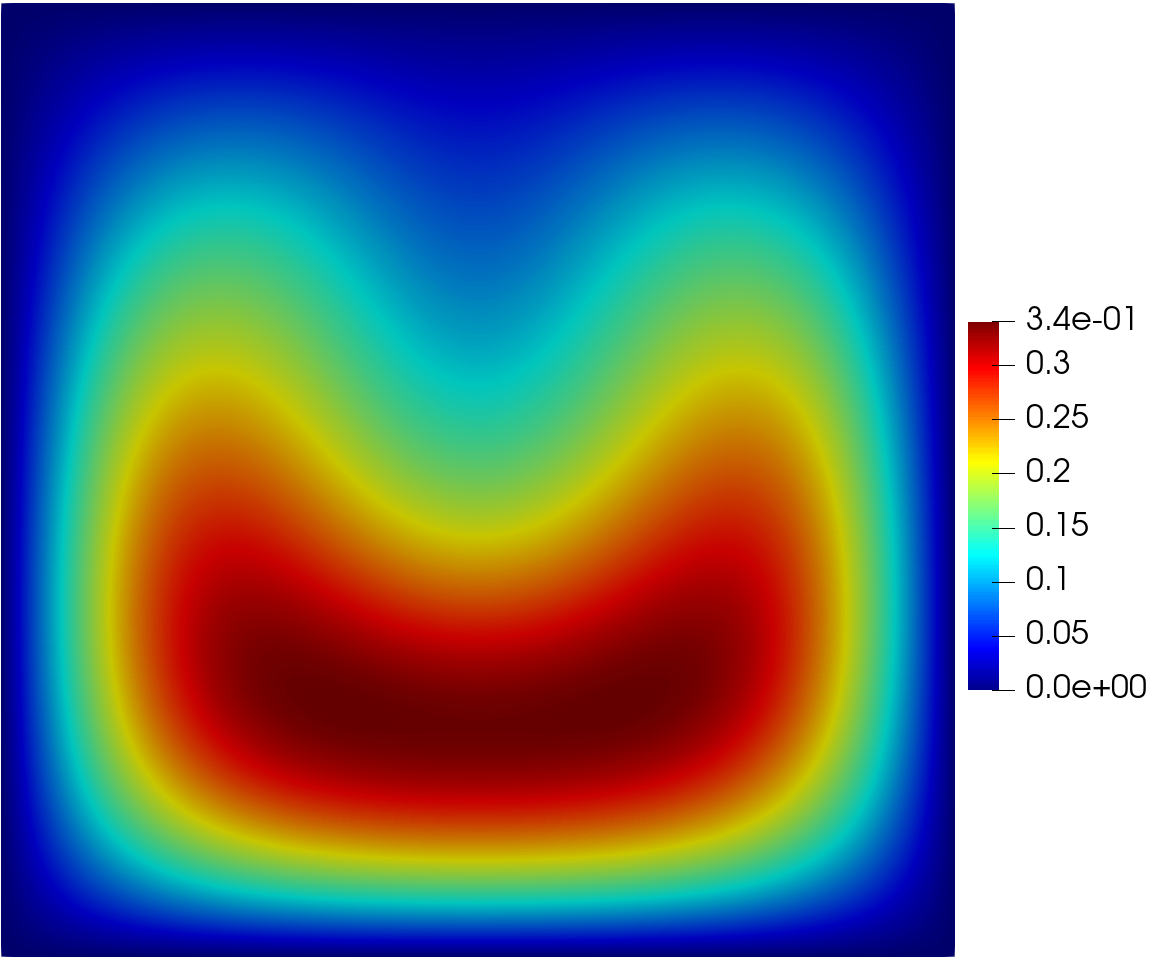}
\\
 (c) &(d)\\
\end{tabular}
\caption{Example~\ref{sect:divfree-2}: Plots of optimal $T_h$  for $\kappa = 1.0$ and  (a). Initial heat distribution $T_h^0$; (b). $\gamma = 1.8$E-6; (c). $\gamma = 8$E-7; (d). $\gamma = 4$E-7.}\label{fig:divfree-2-T}
\end{figure}

The initial heat distribution corresponding to $\gamma = 1.0$ and $\bv=0$ is plotted in Fig.~\ref{fig:divfree-2-T}a. As shown in this figure, the maximum of $T_h^0$ is 4.6E-1. 
The optimal heat distributions corresponding various values in $\gamma$ are  plotted in  Fig.~\ref{fig:divfree-2-T}b-c. 
We observe similar results as in Example \ref{DivFree-Test1}, i.e., the smaller value in $\gamma$ will yield the lower maximum  of the optimal temperature.

 The optimal vector fields and  their streamlines are demonstrated in Fig.~\ref{fig:divfree-2-V}-\ref{fig:divfree-2-V-Stream}.
The profiles of the cost functional are plotted in Fig.~\ref{fig:divfree-2-Conv}. For $\gamma = 4$E-7, we obtain the cost functional value $J_{\min}$ = 6.76E-3, which is 25\% smaller than the initial value (which is $8.97$E-3). 
In this case, we observe that the convergence rate $r_J$ gradually decreases from $0.22$ to almost $0$.

\begin{figure}[H]
\centering
\begin{tabular}{ccc}
\includegraphics[width=.3\textwidth]{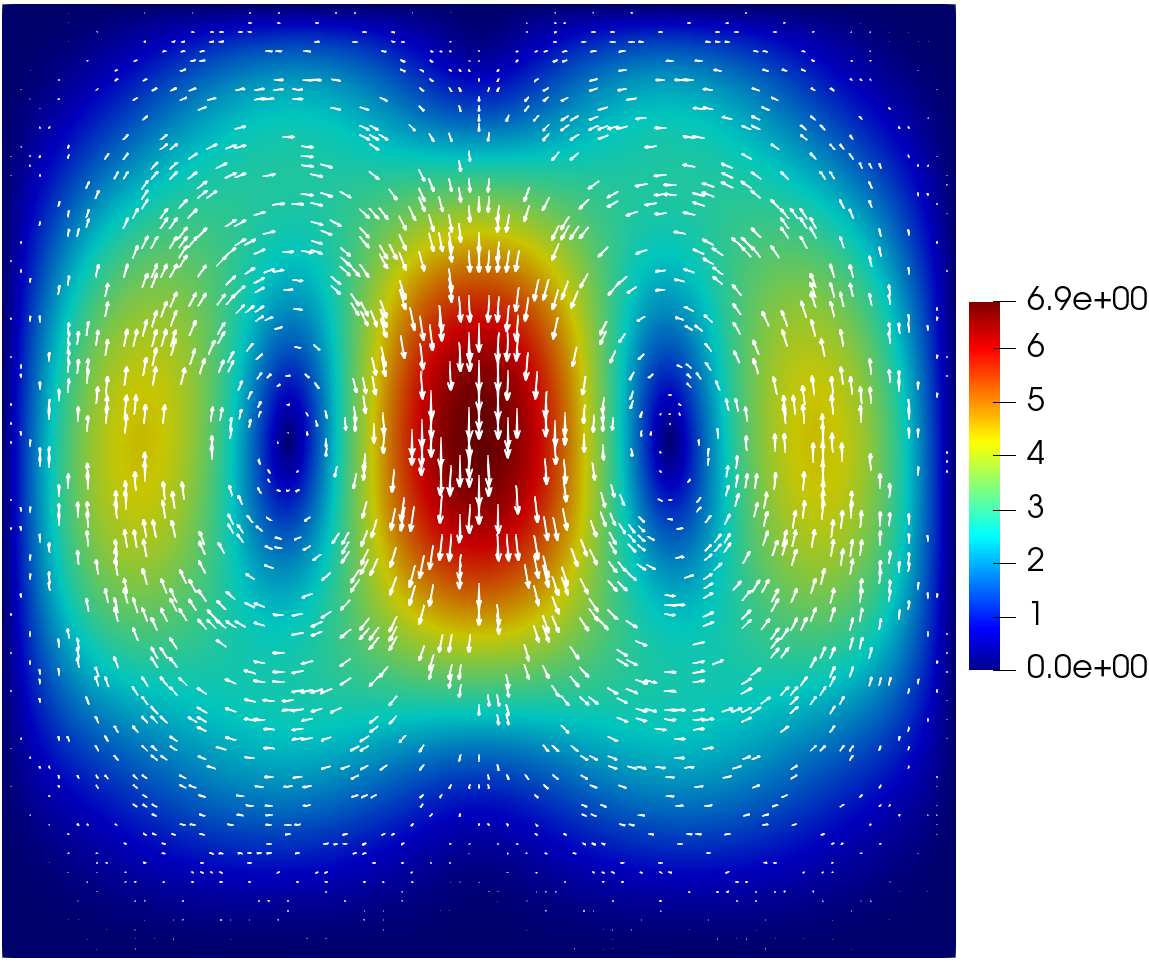}
&\includegraphics[width=.3\textwidth]{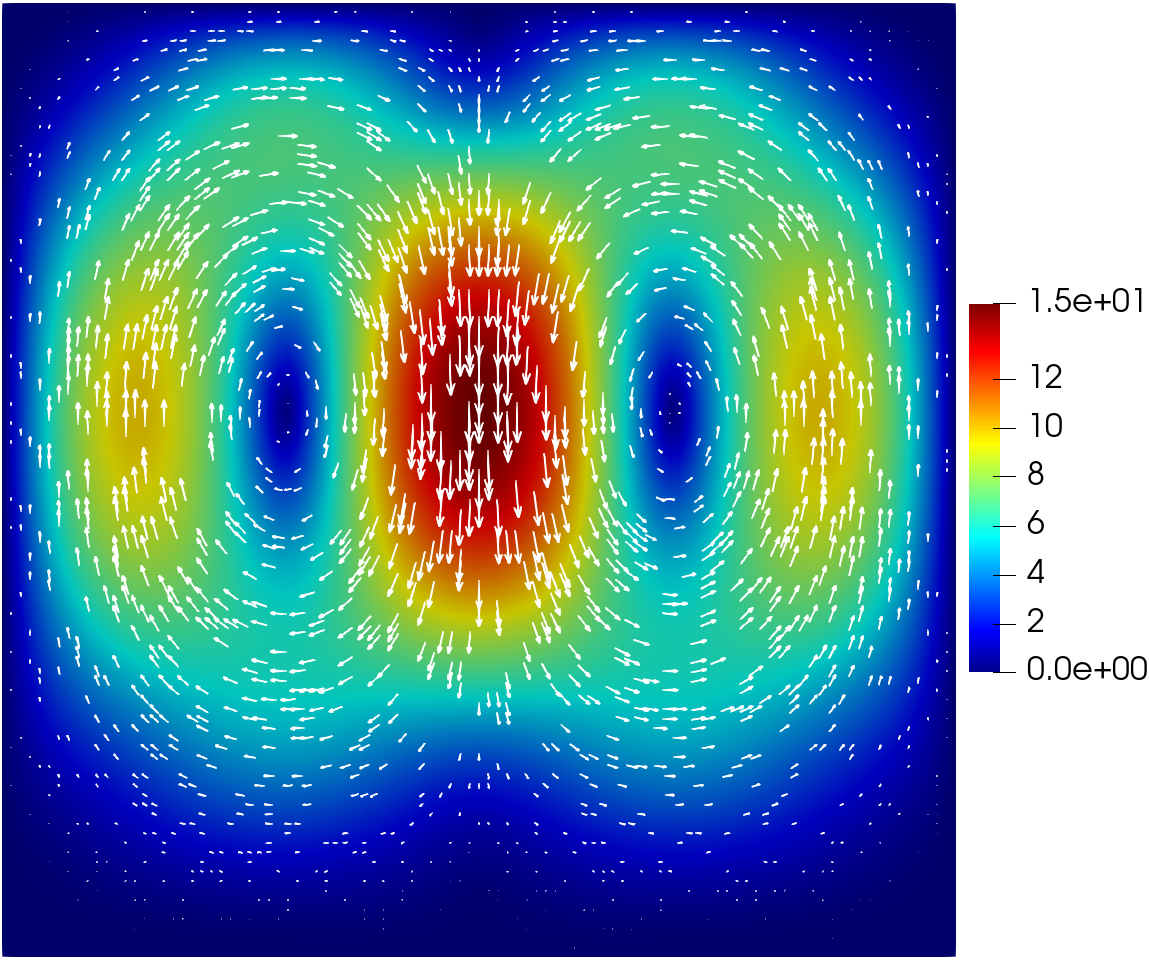}
&\includegraphics[width=.3\textwidth]{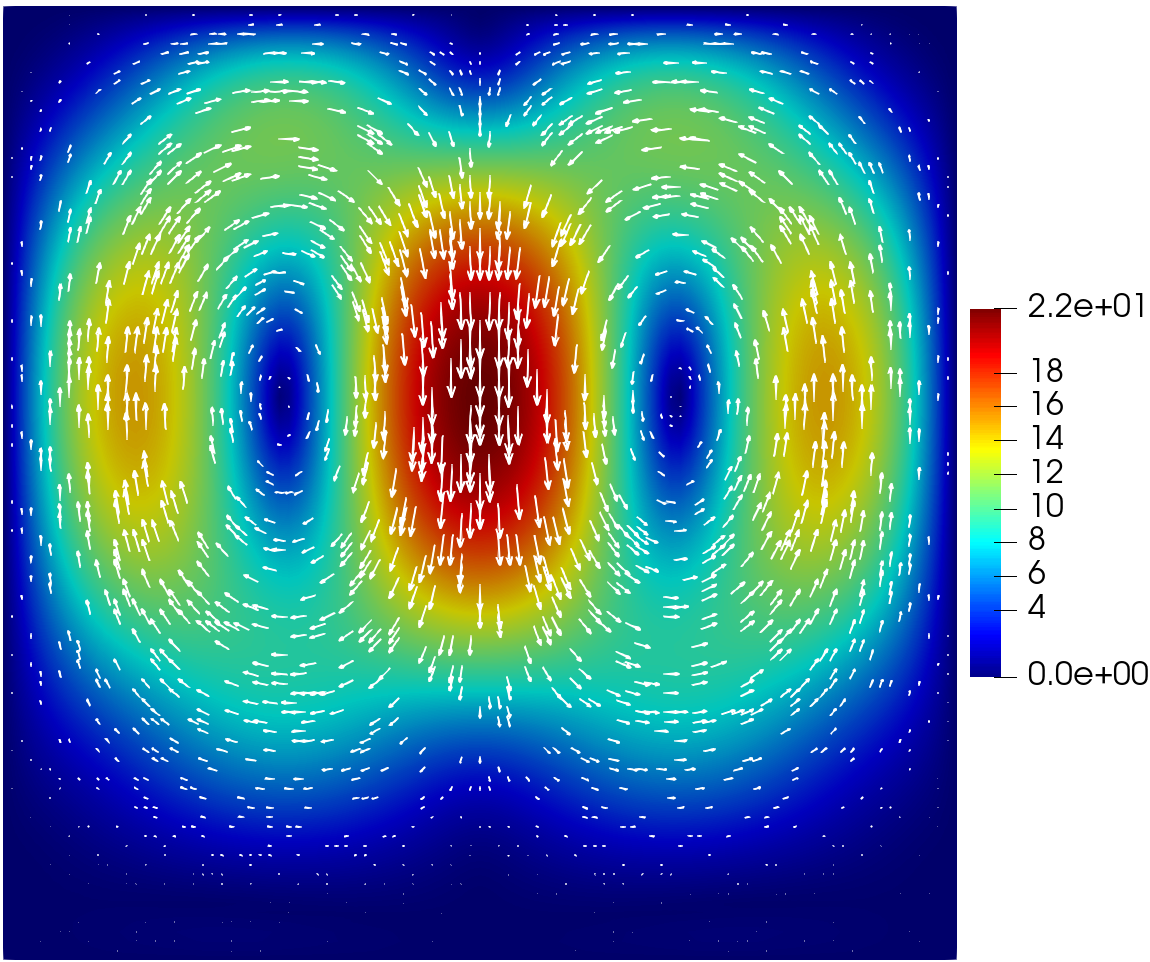}
\\
(a) & (b) &(c)
\end{tabular}
\caption{Example~\ref{sect:divfree-2}: Plots of temperature $T_h$ and vector field ${\bv}$ for $\kappa = 1.0$ and (a) $\gamma = 1.8$E-6; (b)$\gamma = 8$E-7; (c). $\gamma = 4$E-7. Here, the color illustrates the magnitude of velocity $\bv_h$ and the vector plots the field of $\bv_h$.}\label{fig:divfree-2-V}
\end{figure}

\begin{figure}[H]
\centering
\begin{tabular}{ccc}
\includegraphics[width=.3\textwidth]{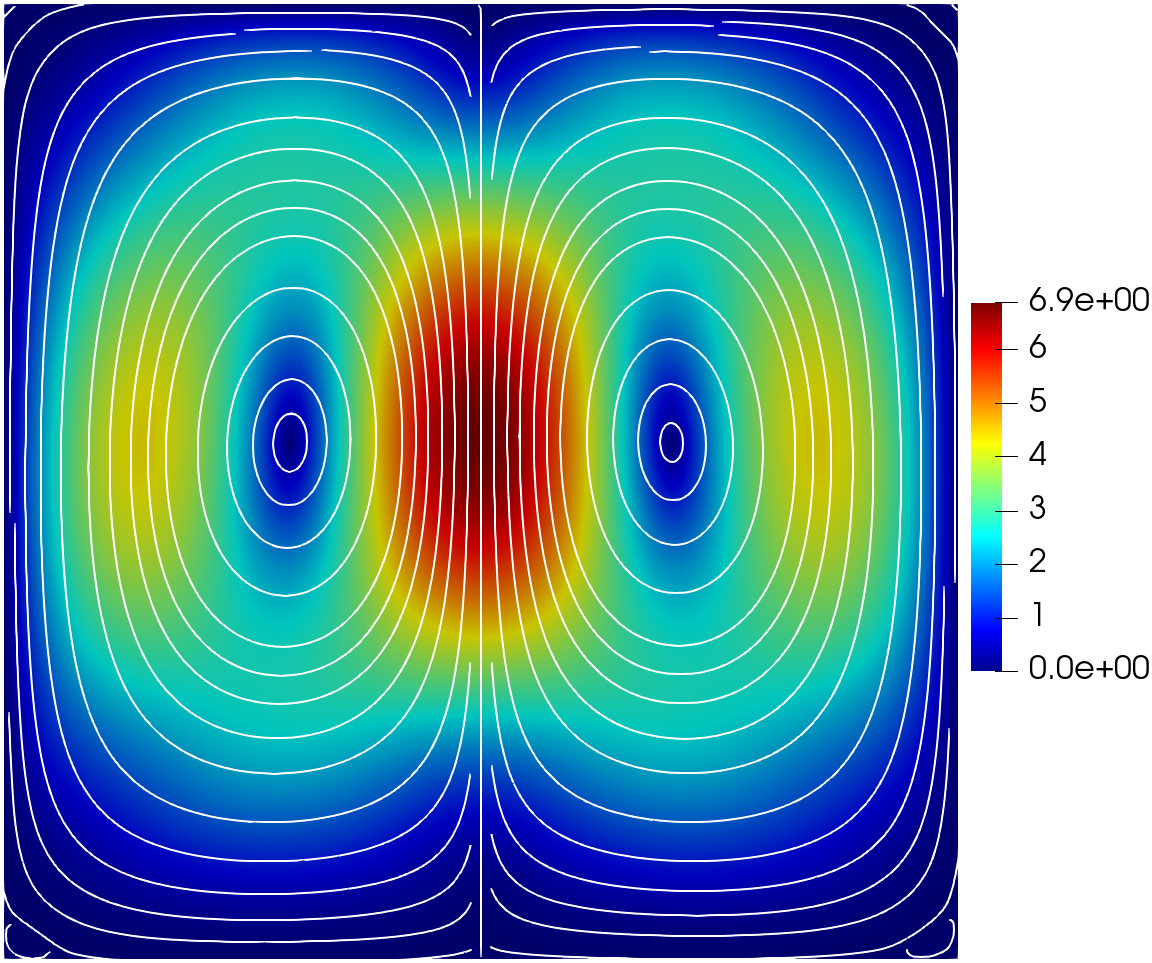}
&\includegraphics[width=.3\textwidth]{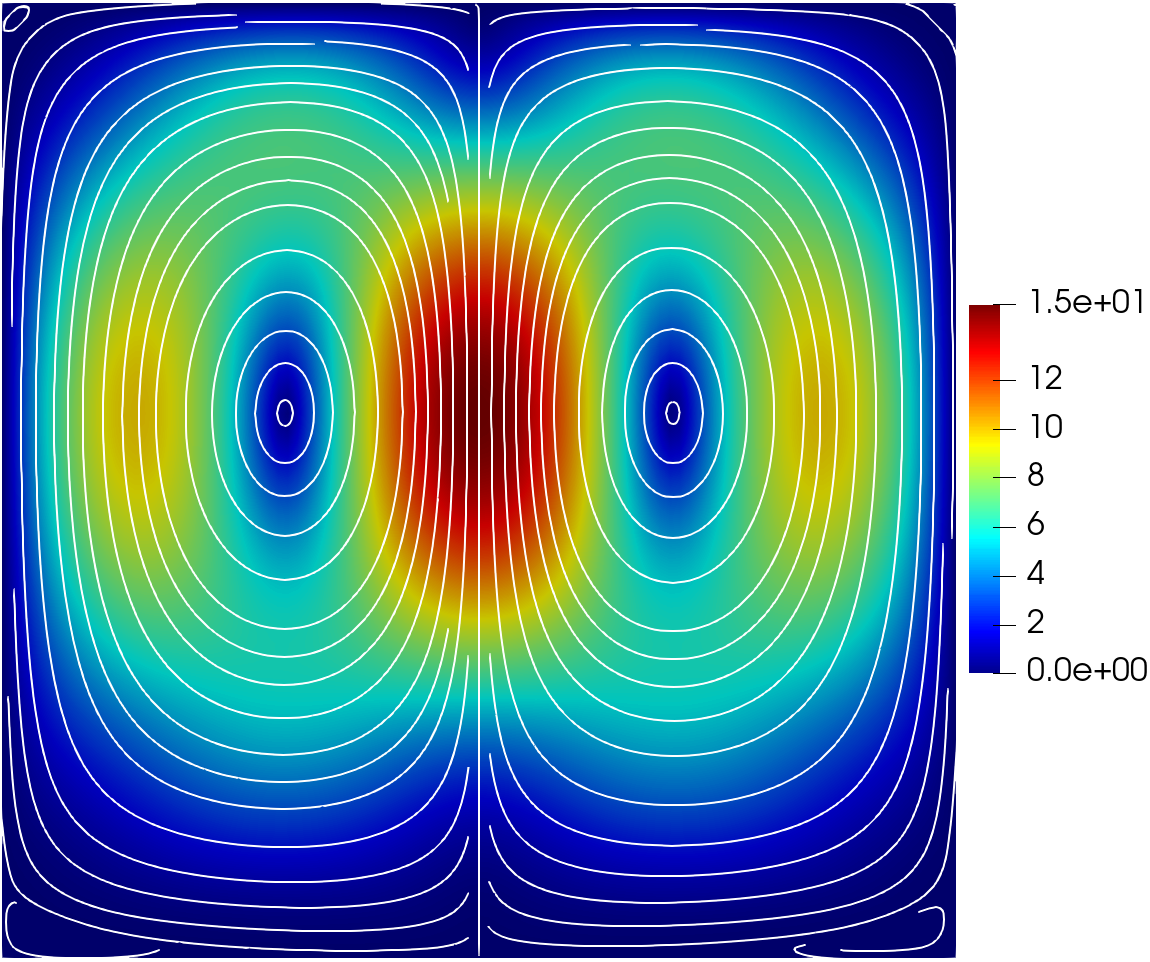}
&\includegraphics[width=.3\textwidth]{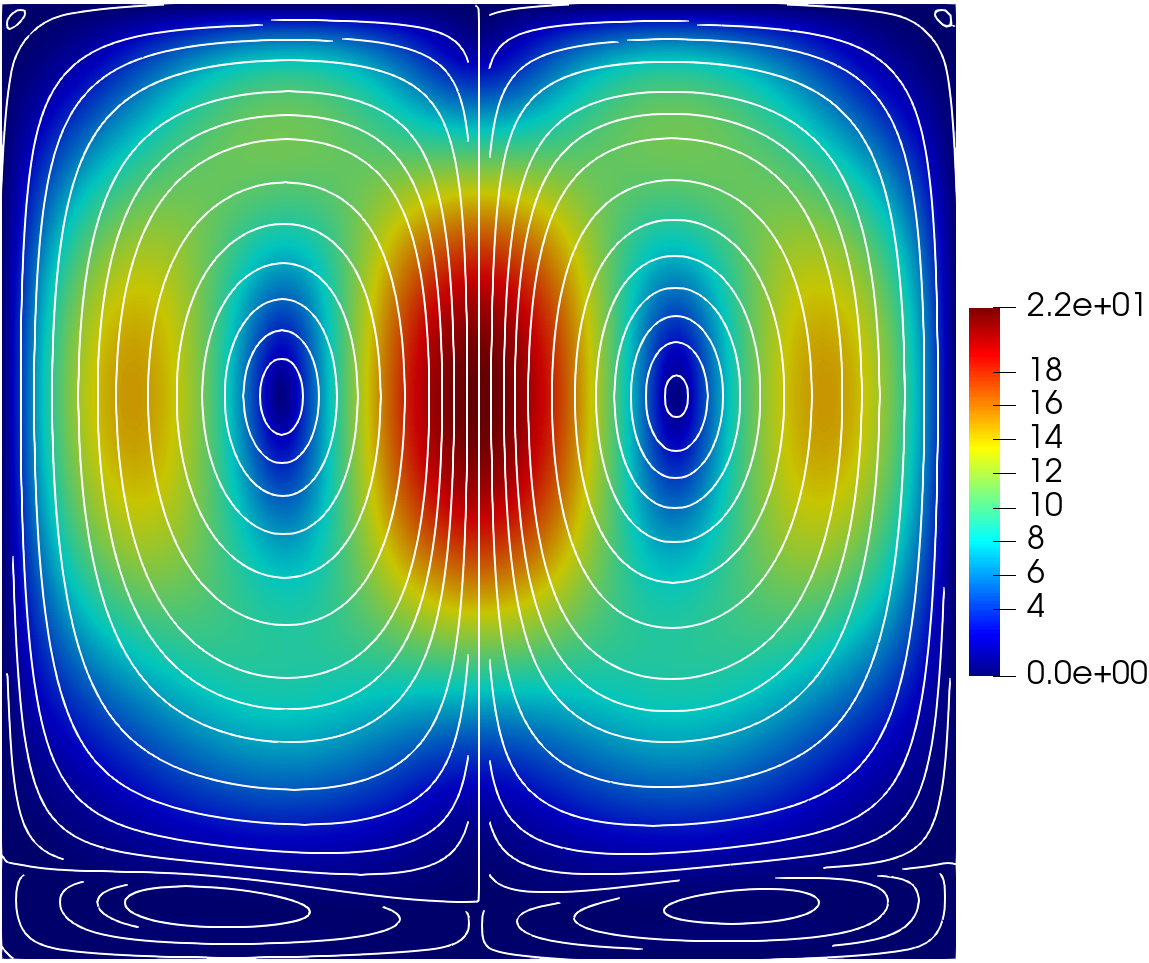}
\\
(a) & (b) &(c)
\end{tabular}
\caption{Example~\ref{sect:divfree-2}: Plots of temperature $T_h$ and vector field ${\bv}$ for $\kappa = 1.0$ and (a) $\gamma = 1.8$E-6; (b)$\gamma = 8$E-7; (c). $\gamma = 4$E-7. Here, the color illustrates the magnitude of velocity $\bv_h$ and the curve plots the streamline of $\bv_h$.}\label{fig:divfree-2-V-Stream}
\end{figure}

\begin{figure}[H]
\centering
\begin{tabular}{cc}
\includegraphics[width=.45\textwidth]{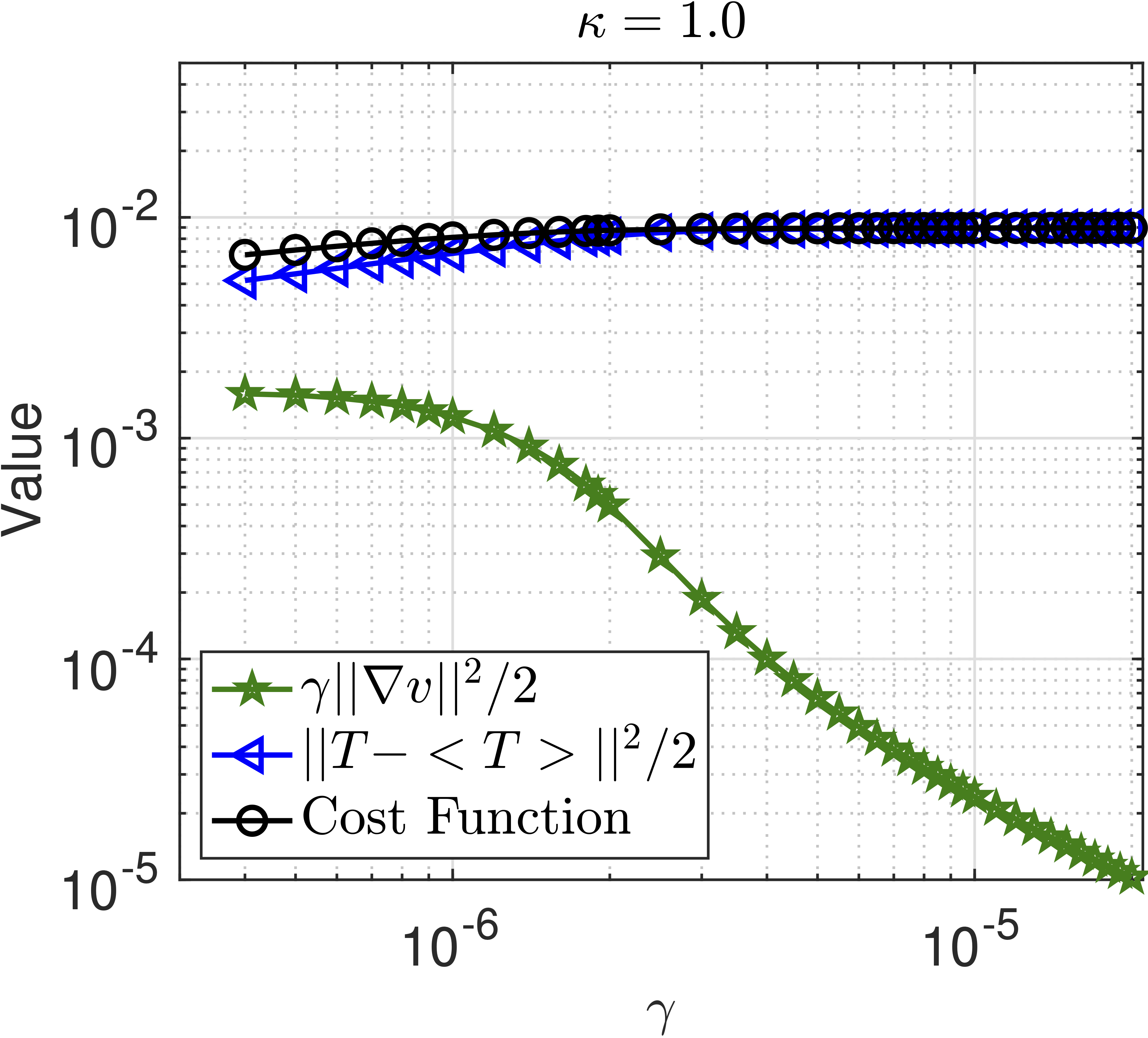}
&\includegraphics[width=.45\textwidth,height = .4\textwidth]{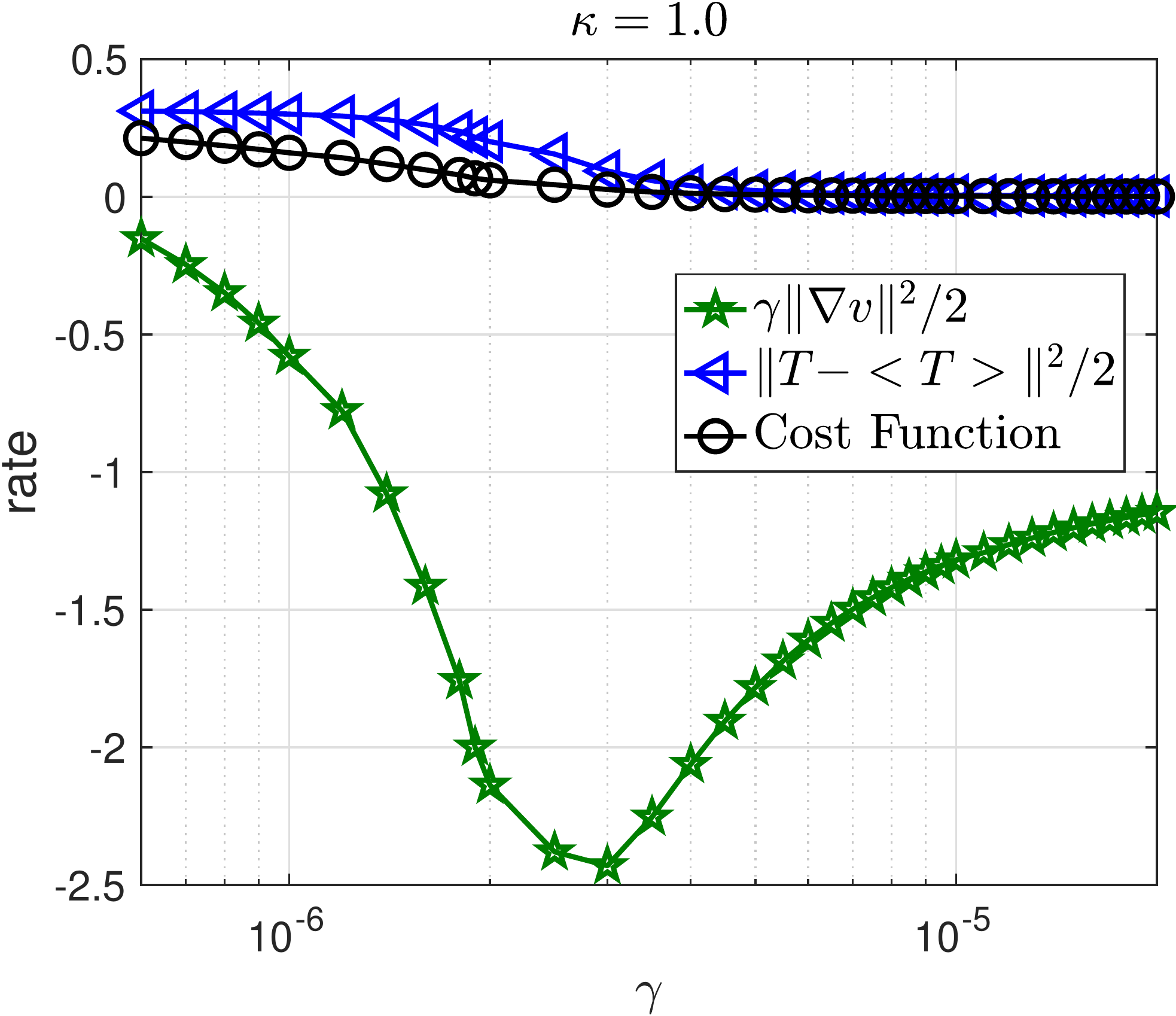}
\\
(a) & (b) 
\end{tabular}
\caption{Example~\ref{sect:divfree-2}: Illustration of results for $\kappa= 1.0$: 
(a). Plots of profiles in the cost functional with respect to $\gamma$ (here $\|T_h^0-\langle T_h^0\rangle\|^2/2 = 8.97$E-3); 
(b). Convergence  rates $r_J, r_T$ and $r_{\bv}$ computed by \eqref{r_J}--\eqref{r_v}}.\label{fig:divfree-2-Conv}
\end{figure}

\begin{example}\label{DivFree-Test3}
{In this example, we continue to  examine  an asymmetric  distribution of the heat source, where the heat source is  centered at the upper right corner.  We especially examine  the behavior of the velocity field subject to such a heat distribution with a sharp peak.
} Let
\[
	f(x,y) =100\exp(-100(x-0.75)^2-100(y-0.75)^2).
\]
\end{example}

\begin{figure}[H]
\centering
\begin{tabular}{cc}
\includegraphics[width=.35\textwidth]{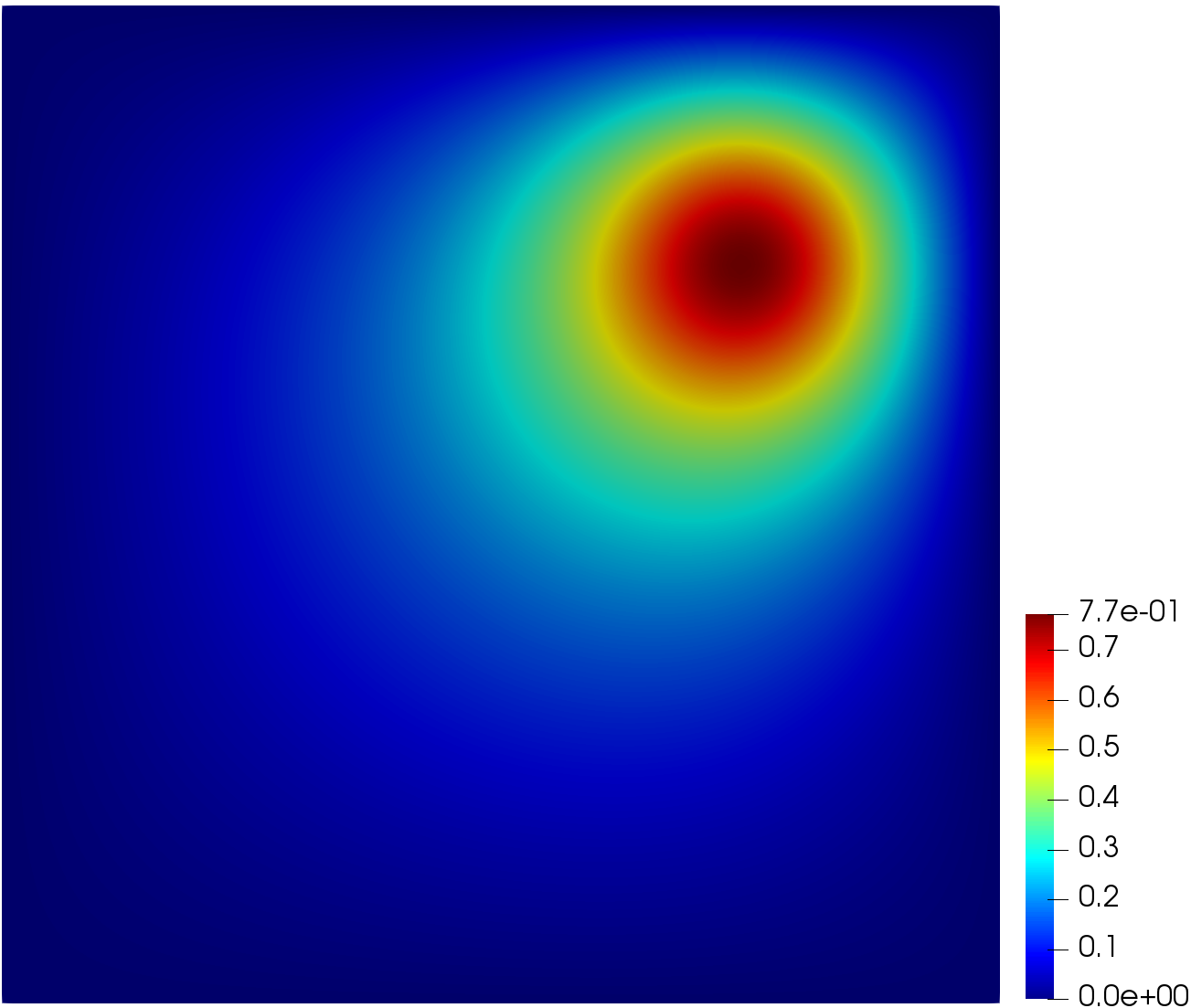}
&\includegraphics[width=.35\textwidth]{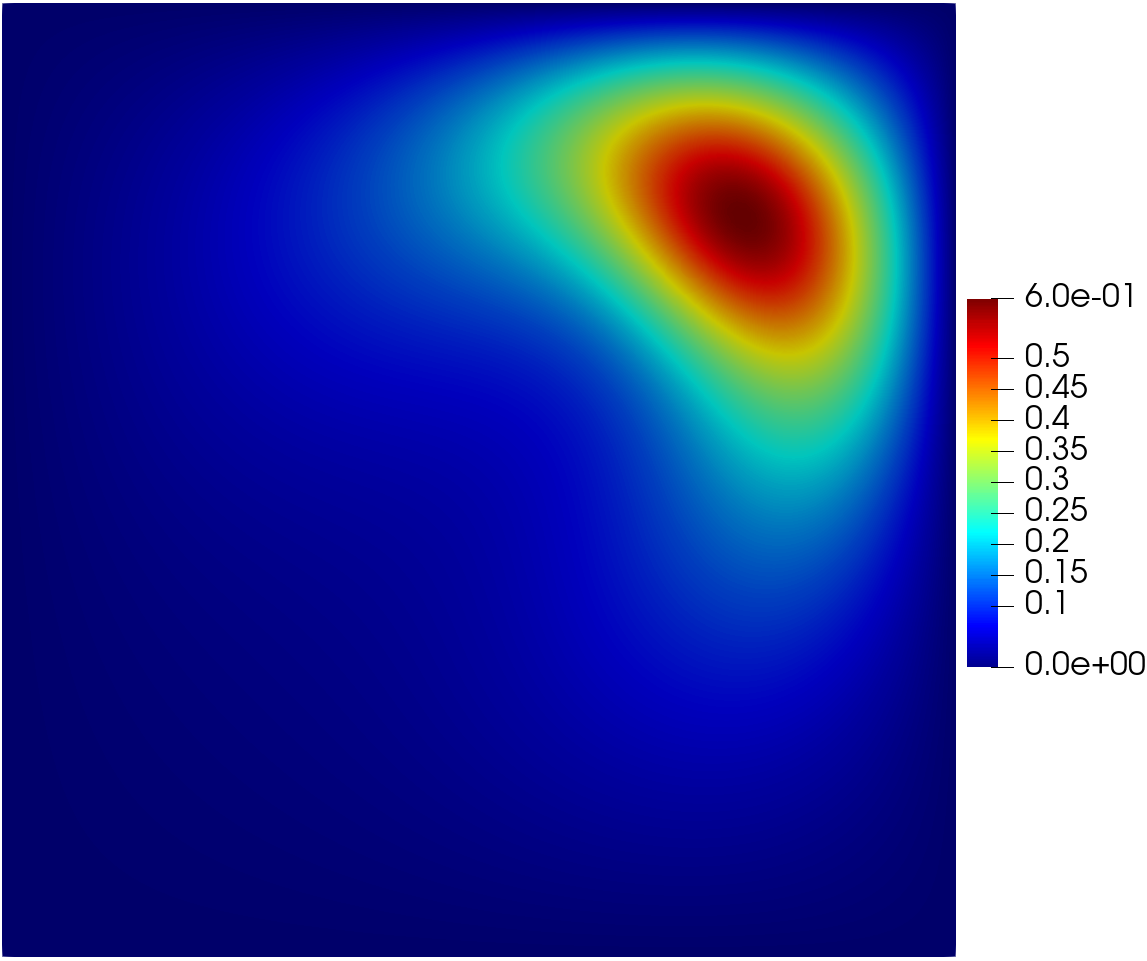}\\
(a) & (b)\\
\includegraphics[width=.35\textwidth]{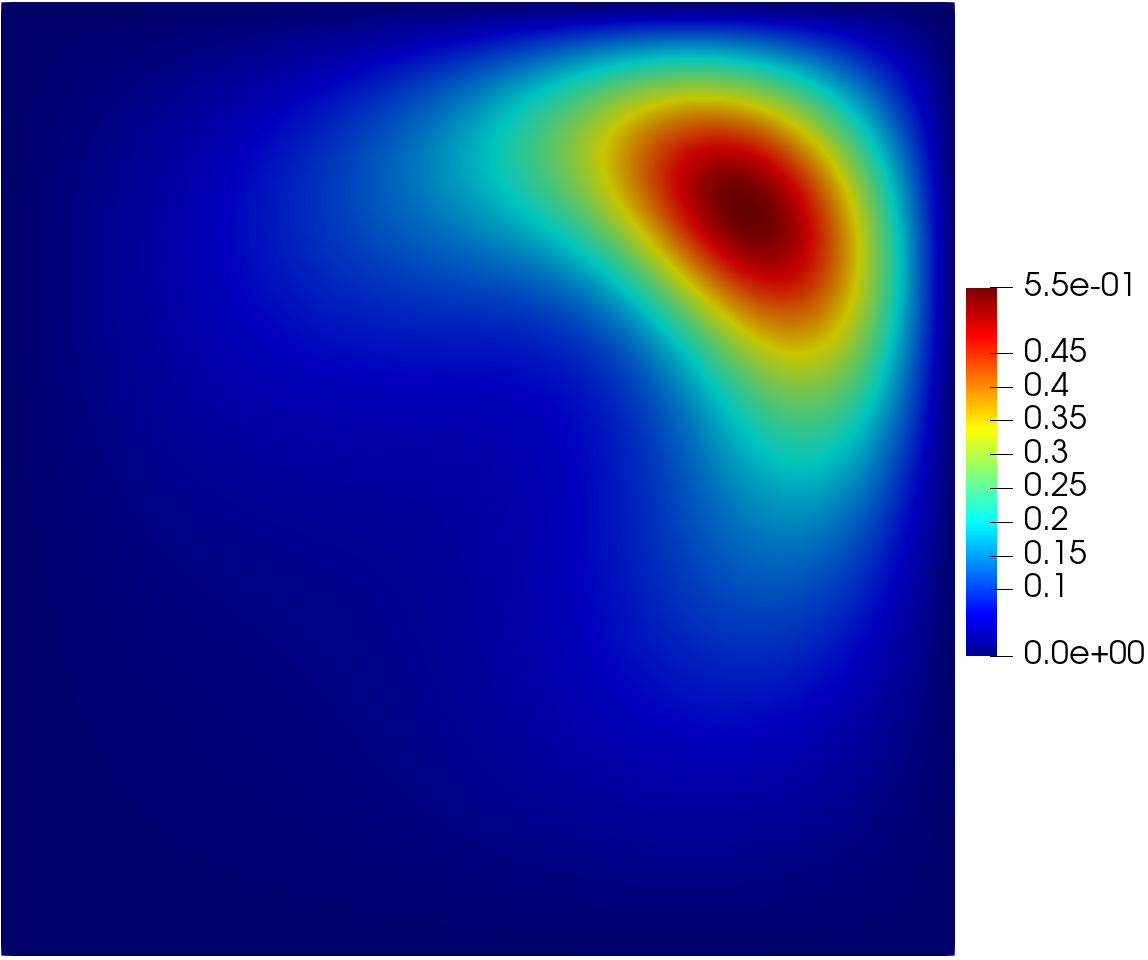}
&\includegraphics[width=.35\textwidth]{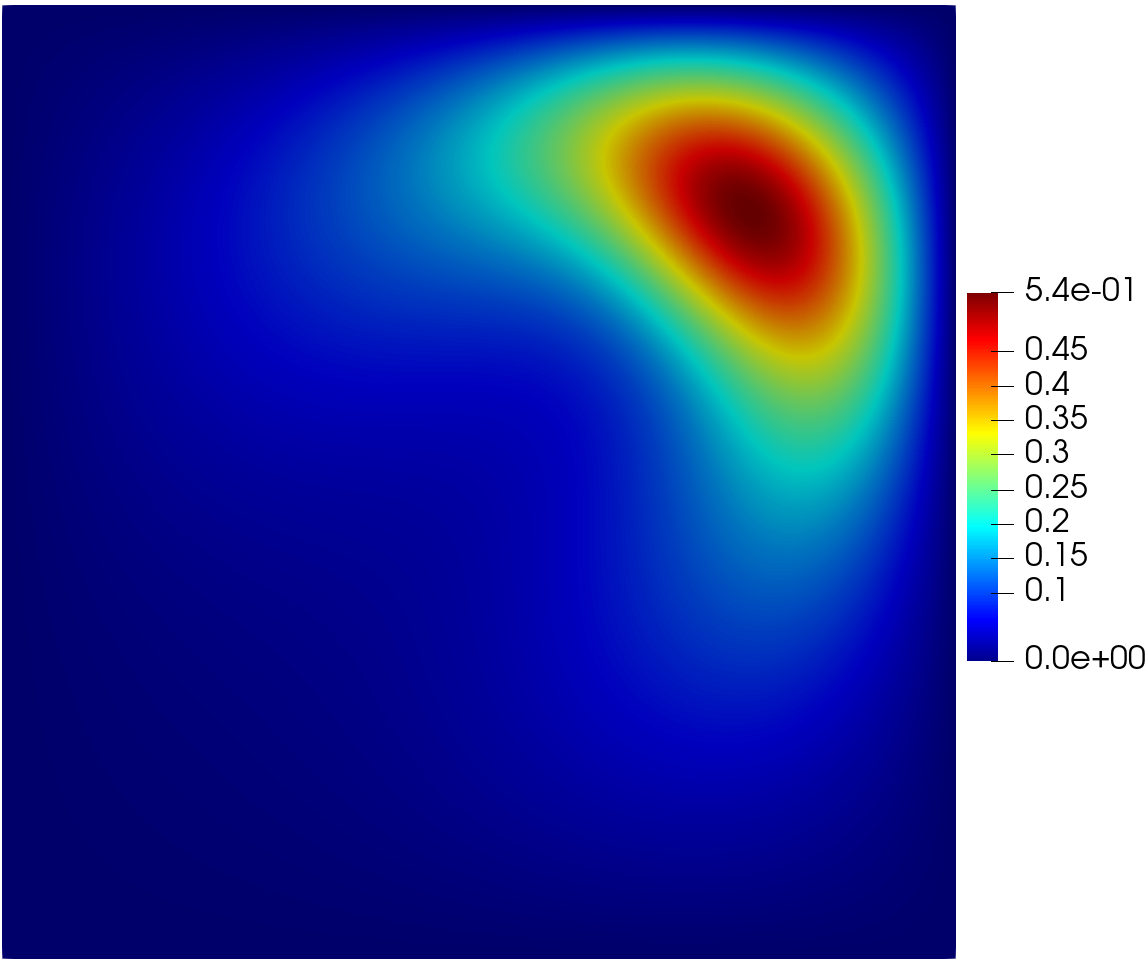}
\\
(c) & (D)\\
\end{tabular}
\caption{Example~\ref{DivFree-Test3}: Plots of optimal $T_h$  with $\kappa = 1.0$ of (a). Initial heat distribution $T_h^0$; and (b). $\gamma = 6$E-7; (c). $\gamma = 3.7$E-7; (d). $\gamma = 3.3$E-7.}\label{fig:divfree-3-T}
\end{figure}

The initial heat distribution corresponding to $\gamma = 1.0$ and $\bv=0$ is plotted in Fig.~\ref{fig:divfree-3-T}a.
As shown in this figure, the maximum of $T_h^0$ is 7.7E-1. 
The numerical optimal solutions for heat distribution $T_h$ are plotted in Fig.~\ref{fig:divfree-3-T} for $\gamma =$6E-7, 3.7E-8, and 3.3E-8. As we  can observe in Fig.~\ref{fig:divfree-3-Conv}a,  the maximum value   of the heat distribution  is reduced from $\max T_h^0 = 0.77$ to $\max T_h = 0.6$, $\max T_h = 0.55$, and $\max T_h=0.54$ corresponding to $\gamma$ =6E-7, 3.7E-7, and 3.3E-7, respectively. Similar to former examples, smaller value in $\gamma$ indicates a more effective cooling process.

{Fig.~\ref{fig:divfree-3-V}-\ref{fig:divfree-3-V-Stream} illustrate the velocity fields and the corresponding streamlines. Based on the direction  fields we observe that for each case the velocity tends to ``blow" the heat source further to the upper right corner, however due to divergence-free, the heat distribution is stretched toward to the cooler region. 
For this example, the velocity fields associated with  different values of $\gamma$  also share  a similar pattern. The profiles of the cost functional  are plotted in Fig.~\ref{fig:divfree-3-Conv}. For $\gamma$ = 3.3E-7, we obtain the cost function value $J_{\min}$ = 7.74E-3, which is 38\% smaller than the initial value (1.24E-2). 
In this case, we find that the convergence rate $r_J$ gradually decreases from $0.29$ to almost $0$. }


\begin{figure}[H]
\centering
\begin{tabular}{ccc}
\includegraphics[width=.3\textwidth]{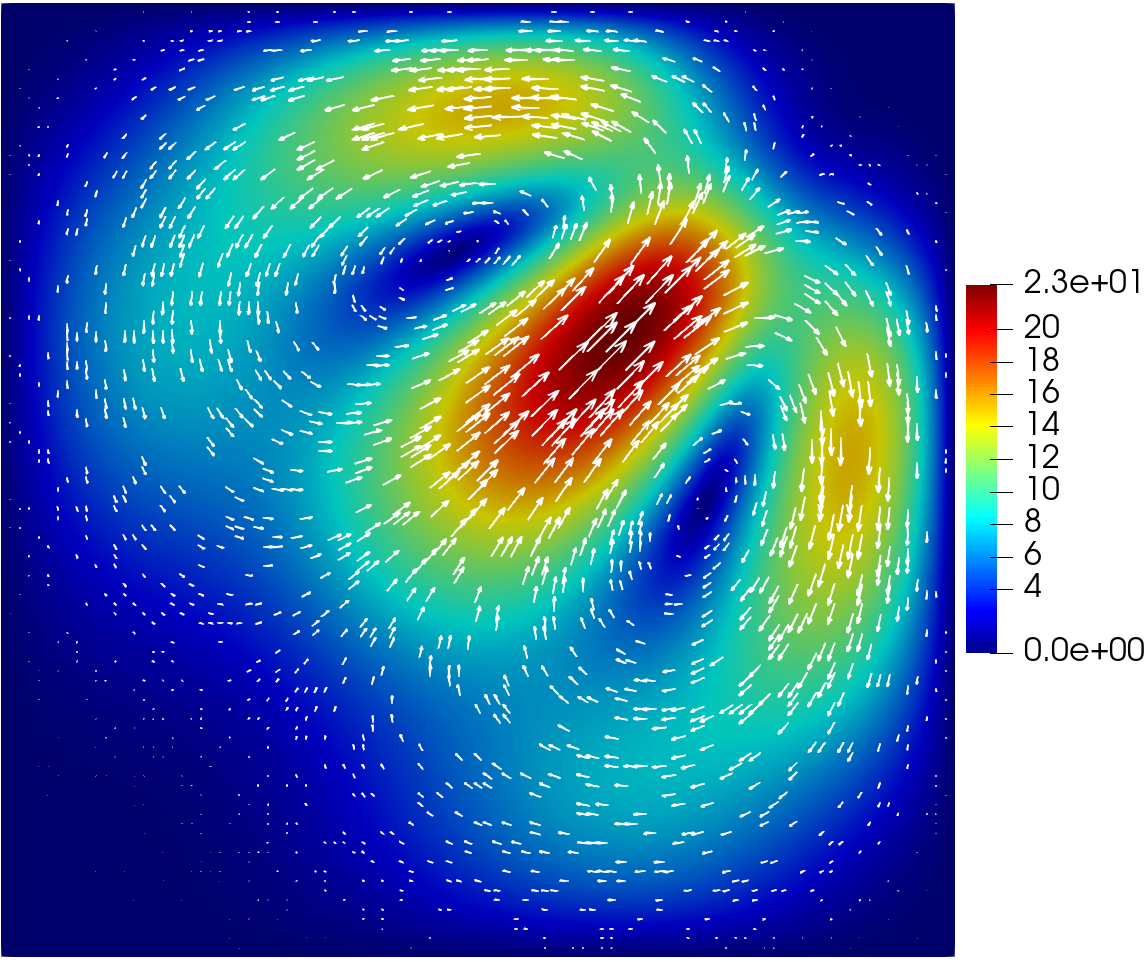}
&\includegraphics[width=.3\textwidth]{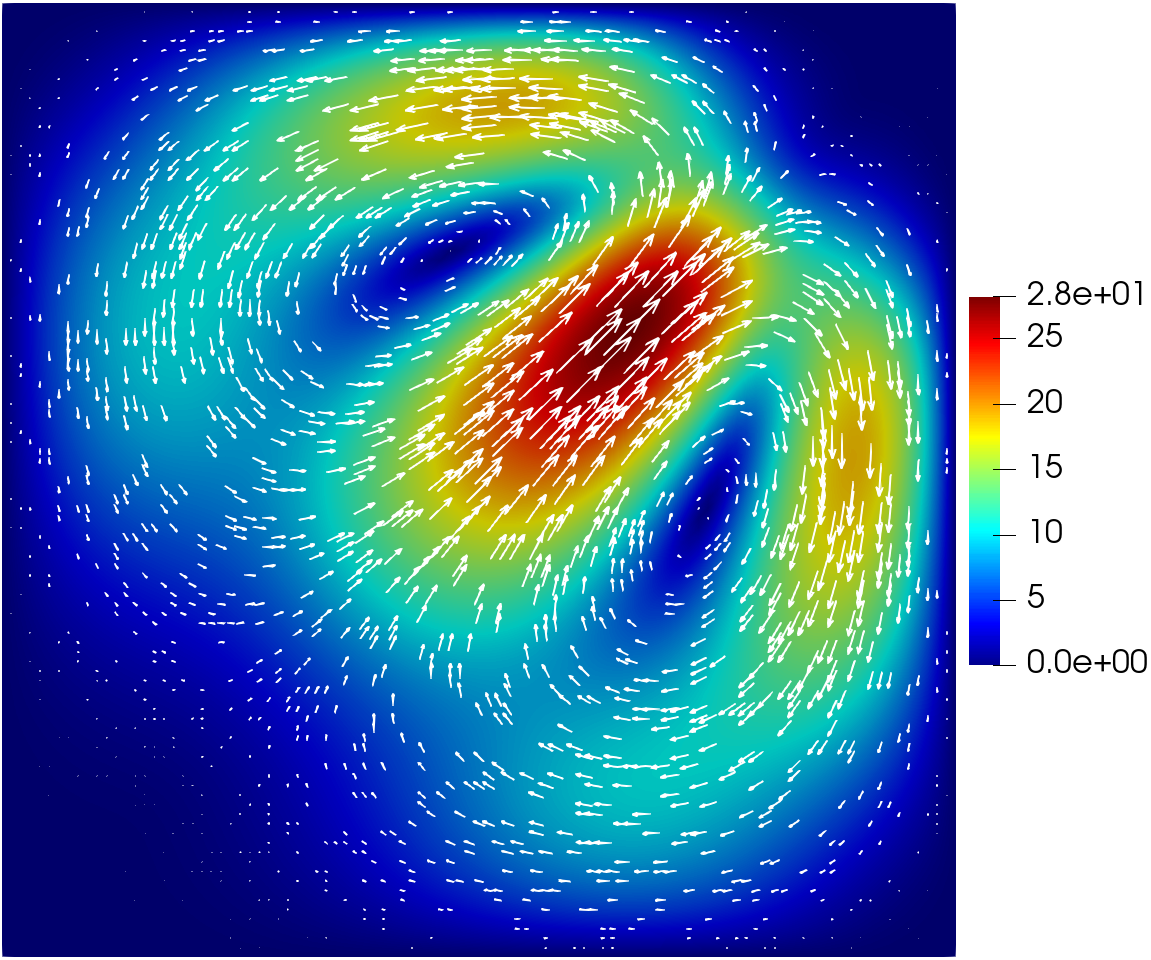}
&\includegraphics[width=.3\textwidth]{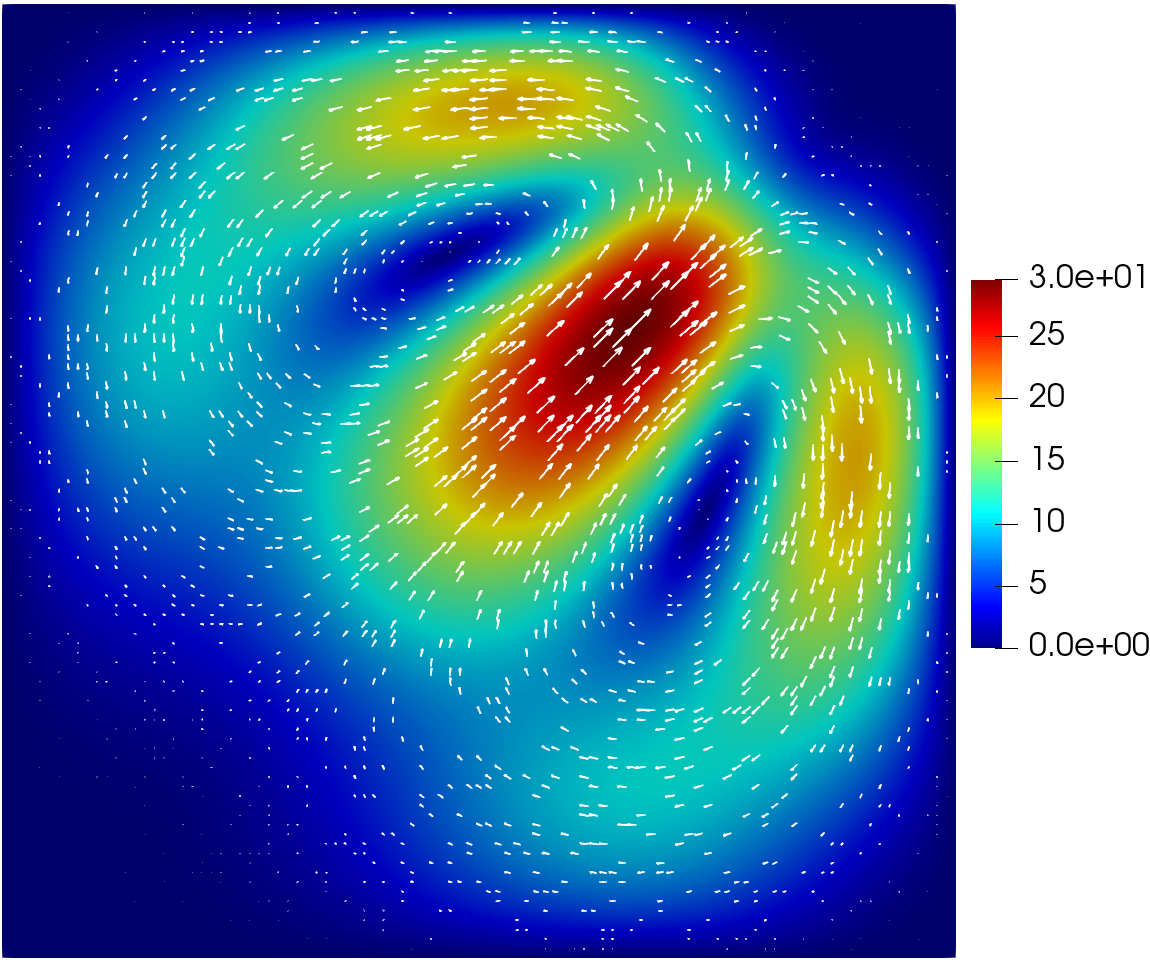}
\\
(a) & (b) &(c)
\end{tabular}
\caption{Example~\ref{DivFree-Test3}: Plots of optimal $\bv_h$  for $\kappa = 1.0$ and (a). $\gamma = 6$E-7; (b). $\gamma = 3.7$E-7; (c). $\gamma = 3.3$E-7. Here, the color illustrates the magnitude of velocity $\bv_h$ and the vector plots the field of $\bv_h$.}\label{fig:divfree-3-V}
\end{figure}

\begin{figure}[H]
\centering
\begin{tabular}{ccc}
\includegraphics[width=.3\textwidth]{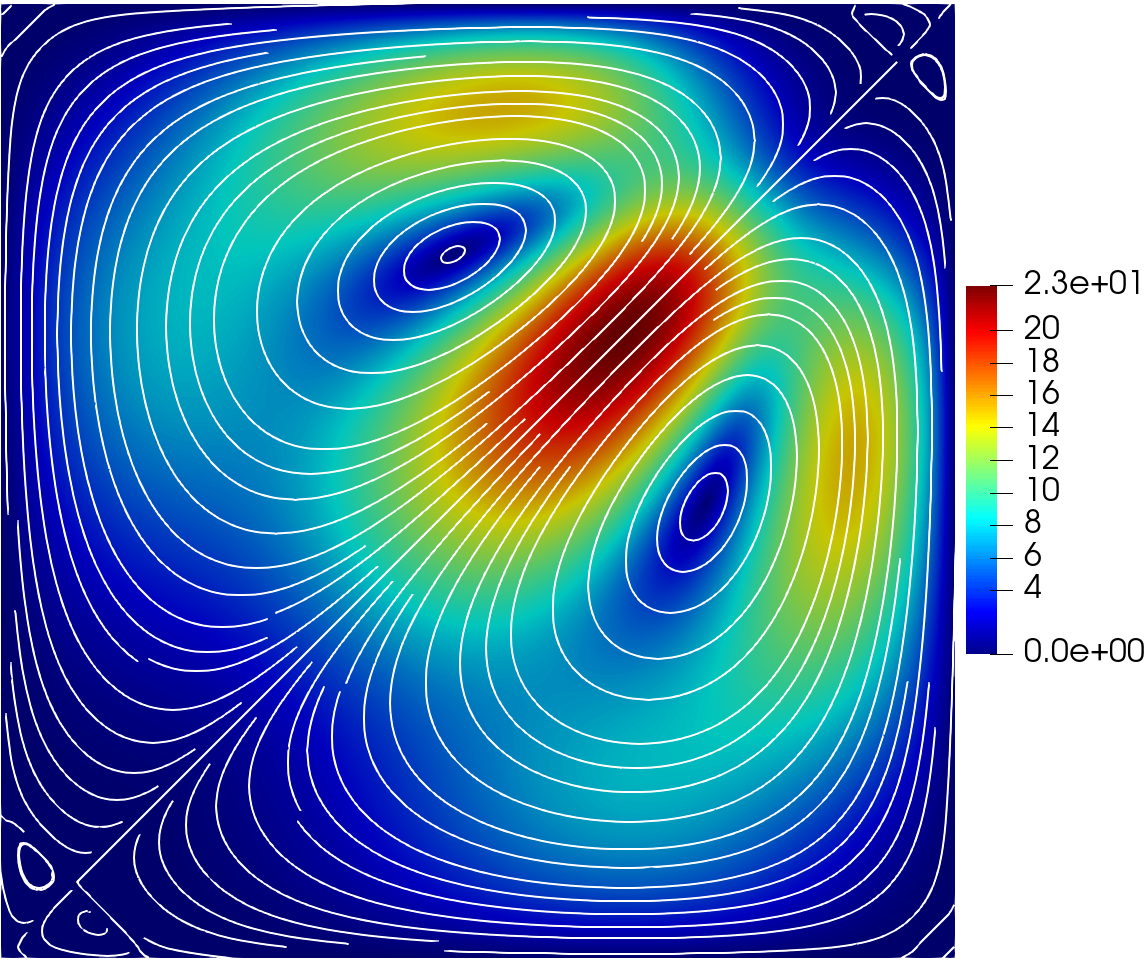}
&\includegraphics[width=.3\textwidth]{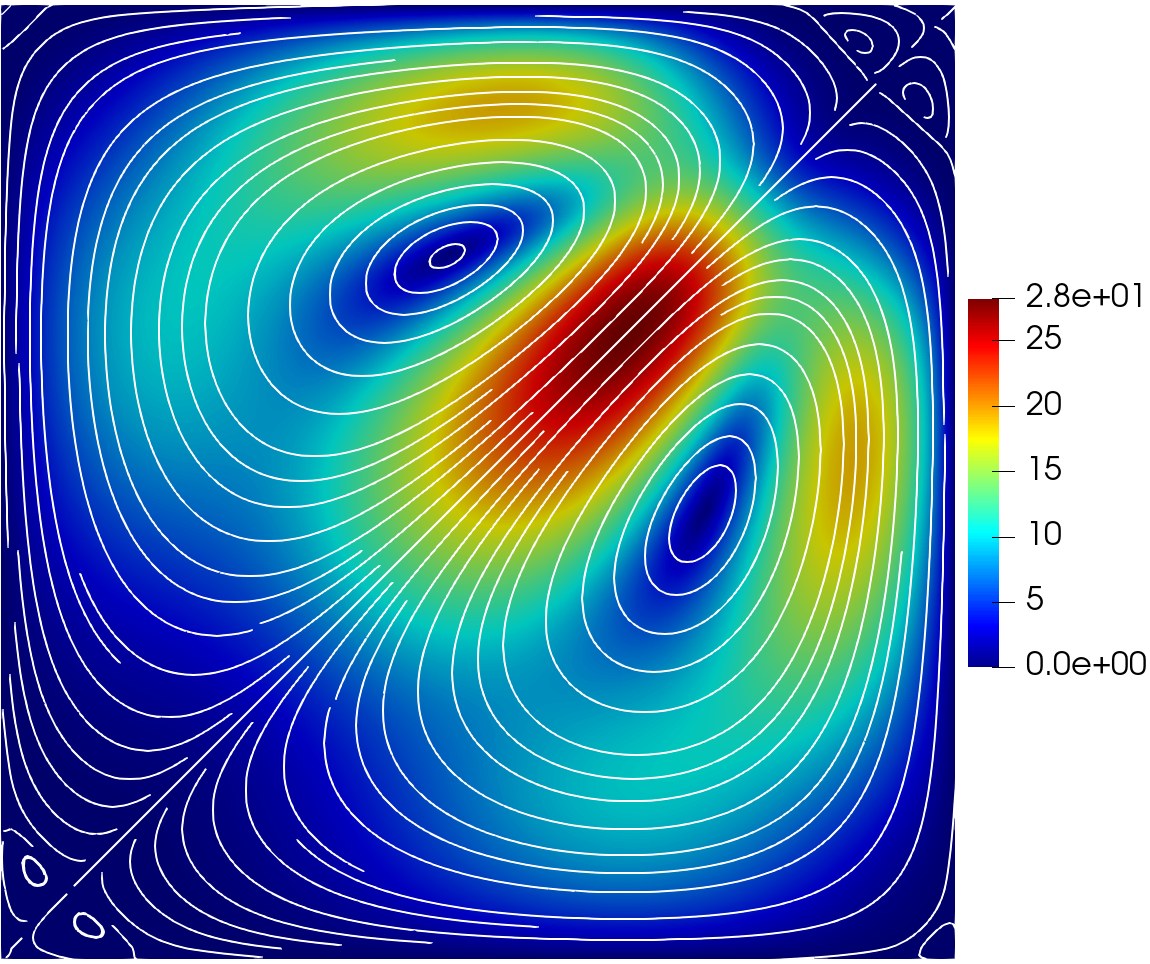}
&\includegraphics[width=.3\textwidth]{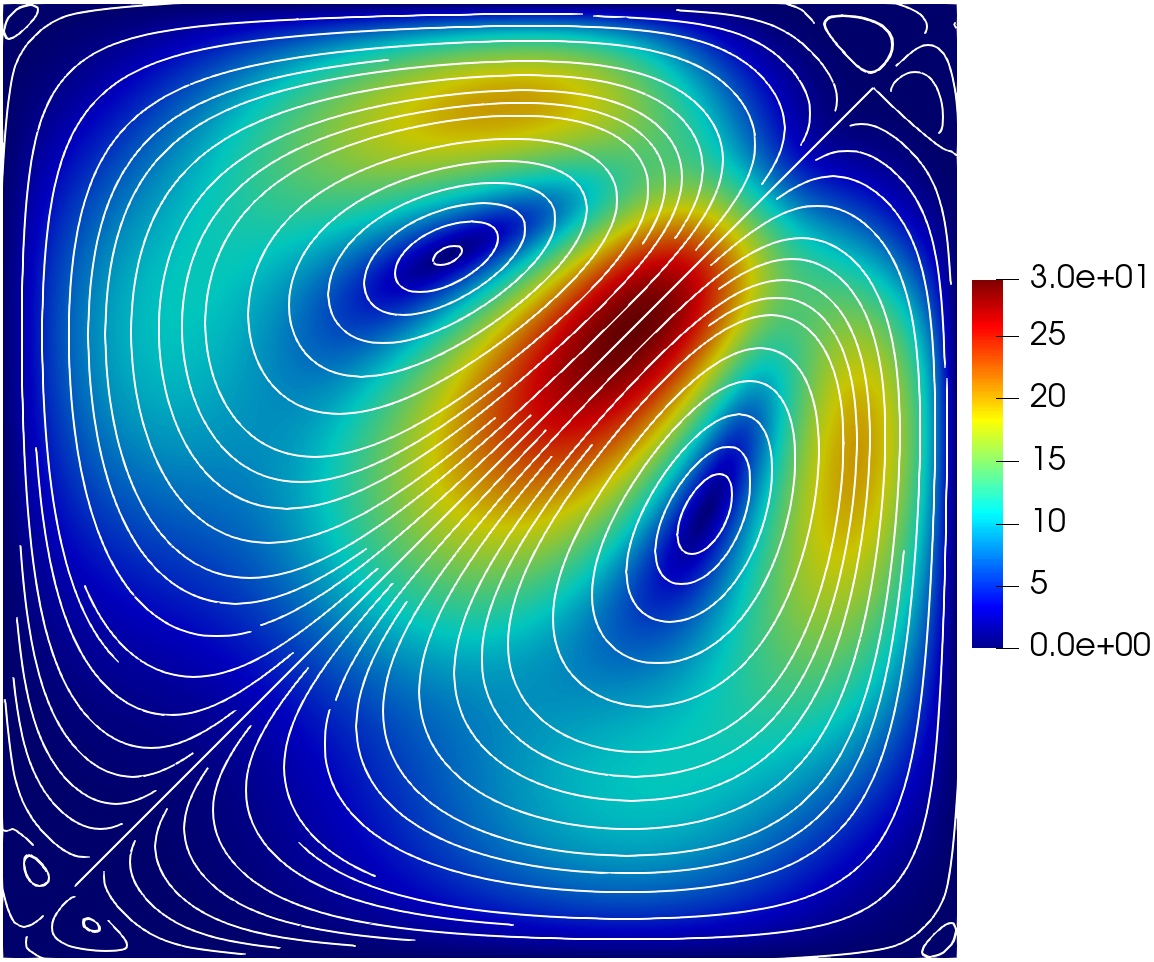}
\\
(a) & (b) &(c)
\end{tabular}
\caption{Example~\ref{DivFree-Test3}: Plots of optimal $T_h$  for $\kappa = 1.0$ and (a). $\gamma = 6$E-7; (b). $\gamma = 3.7$E-7; (c). $\gamma = 3.3$E-7. Here, the color illustrates the magnitude of velocity $\bv_h$ and the curve plots the streamline of $\bv_h$.}\label{fig:divfree-3-V-Stream}
\end{figure}

\begin{figure}[H]
\centering
\begin{tabular}{cc}
\includegraphics[width=.45\textwidth]{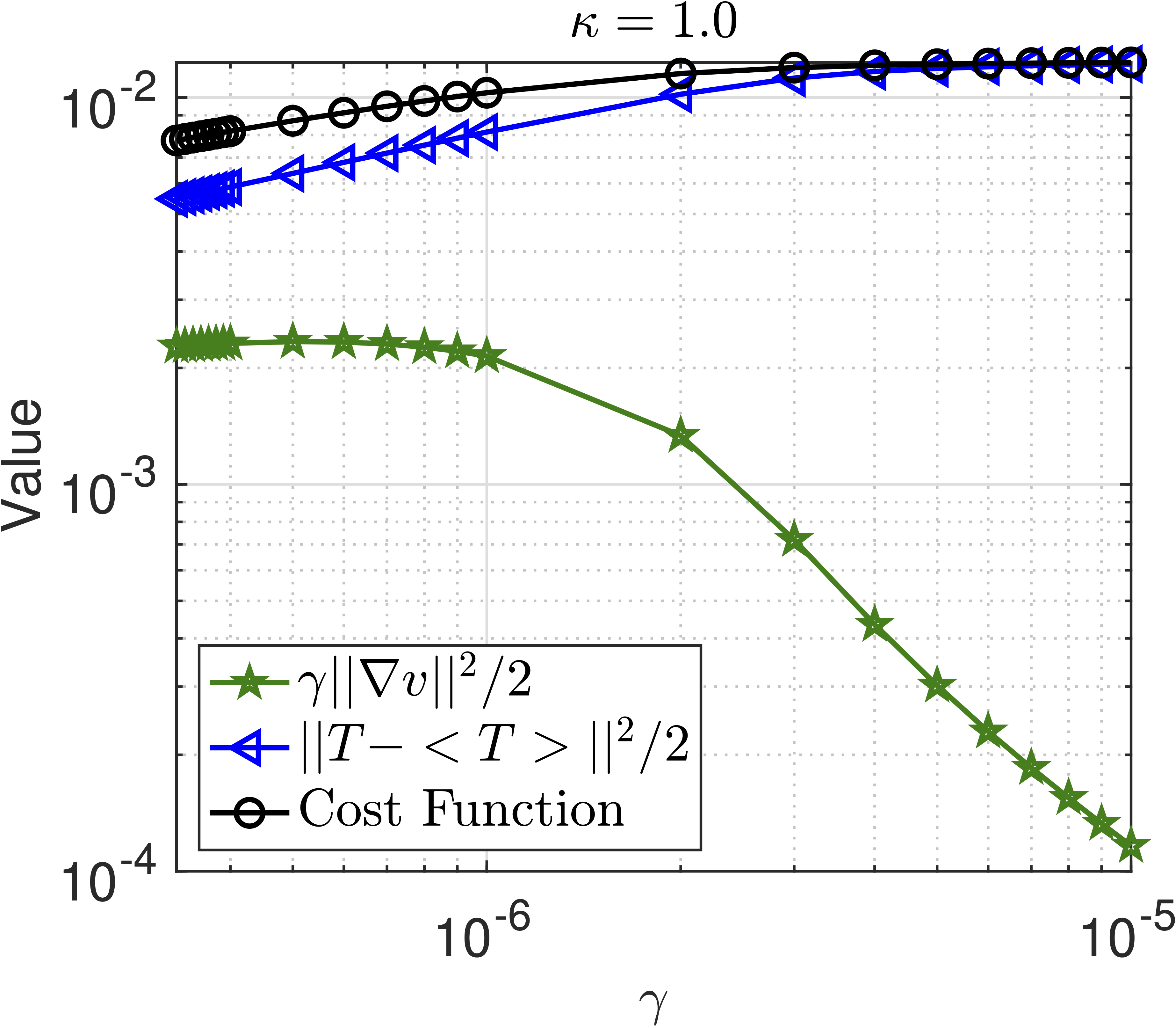}
&\includegraphics[width=.45\textwidth]{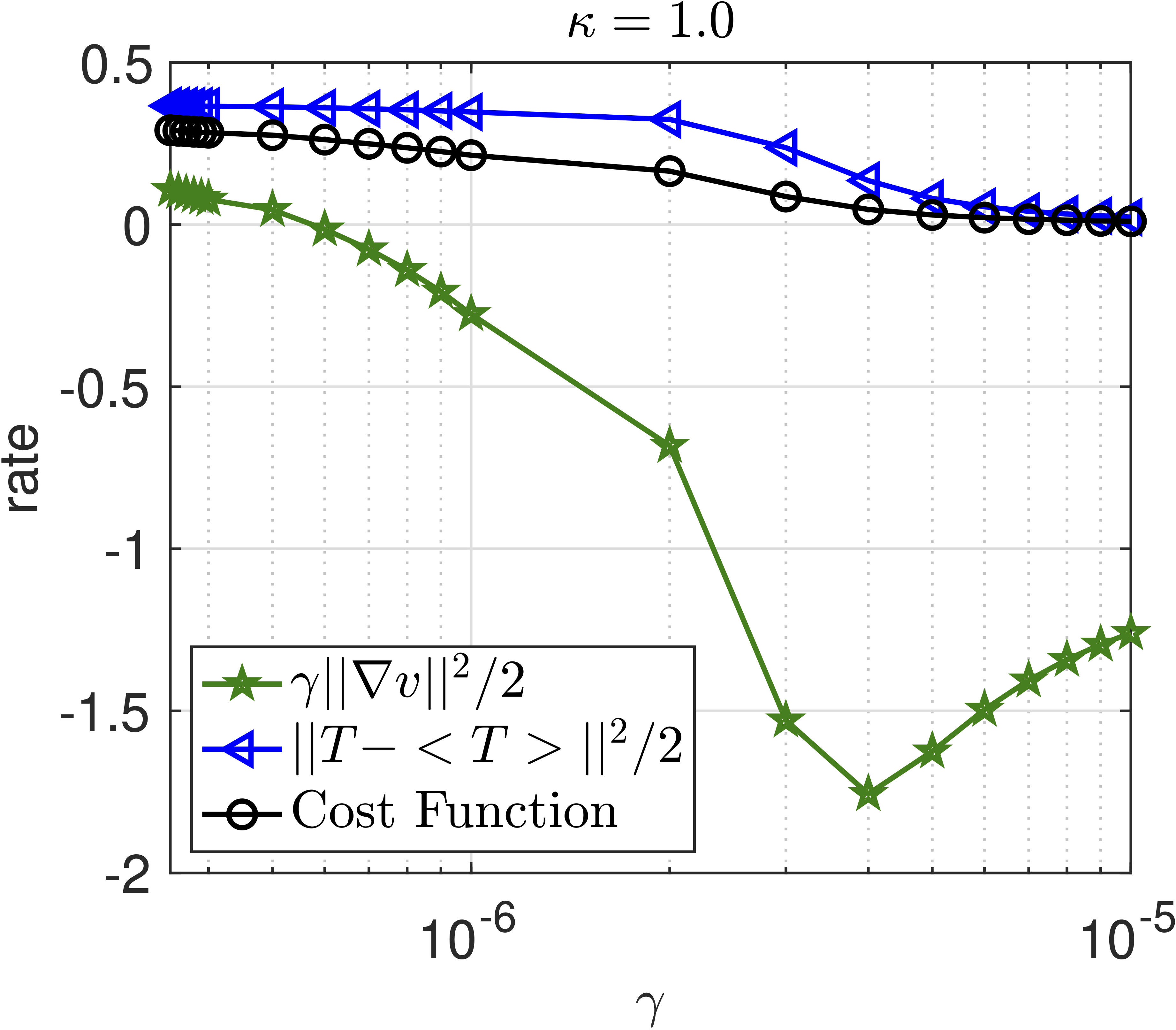}
\\
(a) & (b) 
\end{tabular}
\caption{Example~\ref{DivFree-Test3}: Illustration of results for $\kappa= 1.0$ (a). Plot of initial temperature $T_h^0$ (here $\|T_h^0-\langle T_h^0\rangle\|^2/2 = 1.24$E-2); (b) Plot of profiles in the cost functional with respect to $\gamma$; 
(c) Convergence  rates $r_J, r_T$ and $r_{\bv}$ computed by \eqref{r_J}--\eqref{r_v}}.\label{fig:divfree-3-Conv}
\end{figure}

\begin{example}\label{DivFree-Test4}
In the last  example, we consider that there is  a heat source as well as a heat sink and examine how the velocity behaves in  an environment with such heat distributions. Let 
\[
f(x,y) = 75\exp(-(9x-2)^2/4-(9y-2)^2/4) -75\exp(-(9x-4)^2/4-(9y-7)^2/4).
\]
\end{example}

\begin{figure}[H]
\centering
\begin{tabular}{cc}
\includegraphics[width=.35\textwidth]{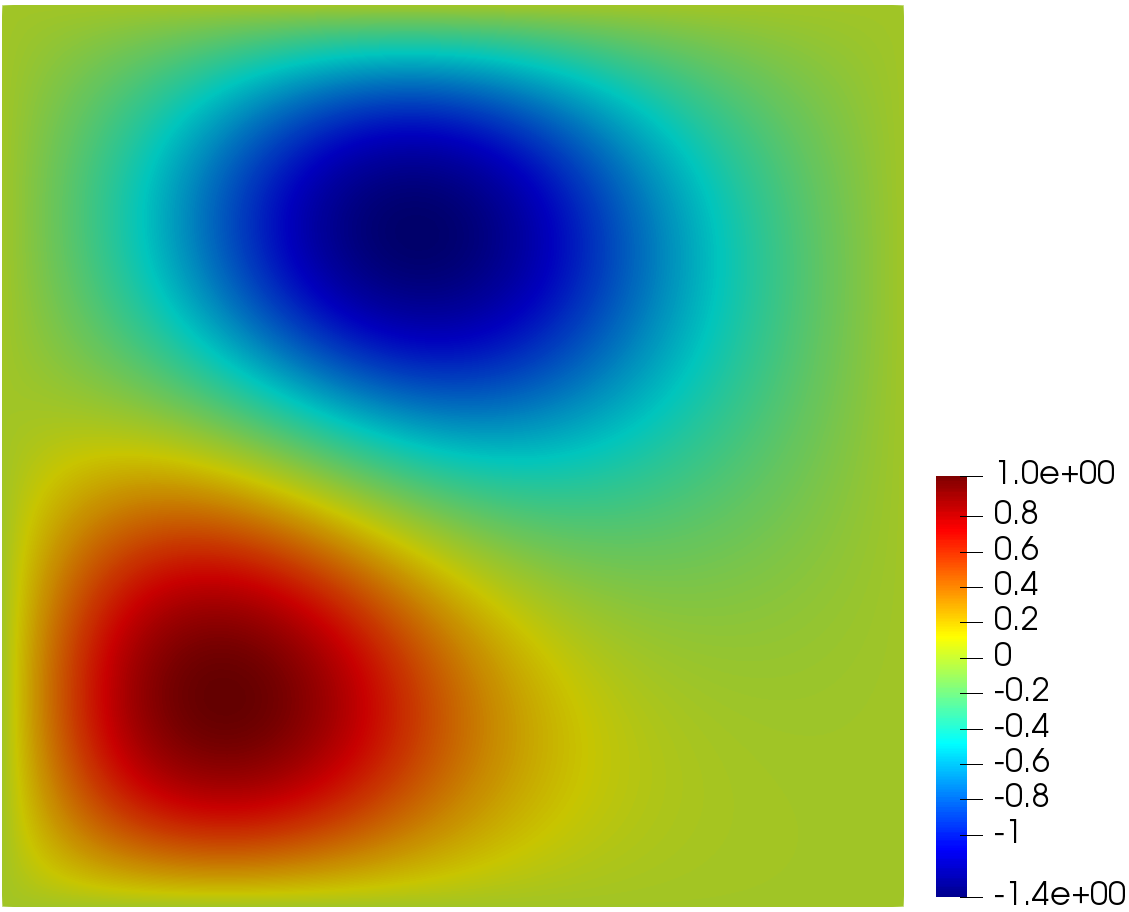}
&\includegraphics[width=.35\textwidth]{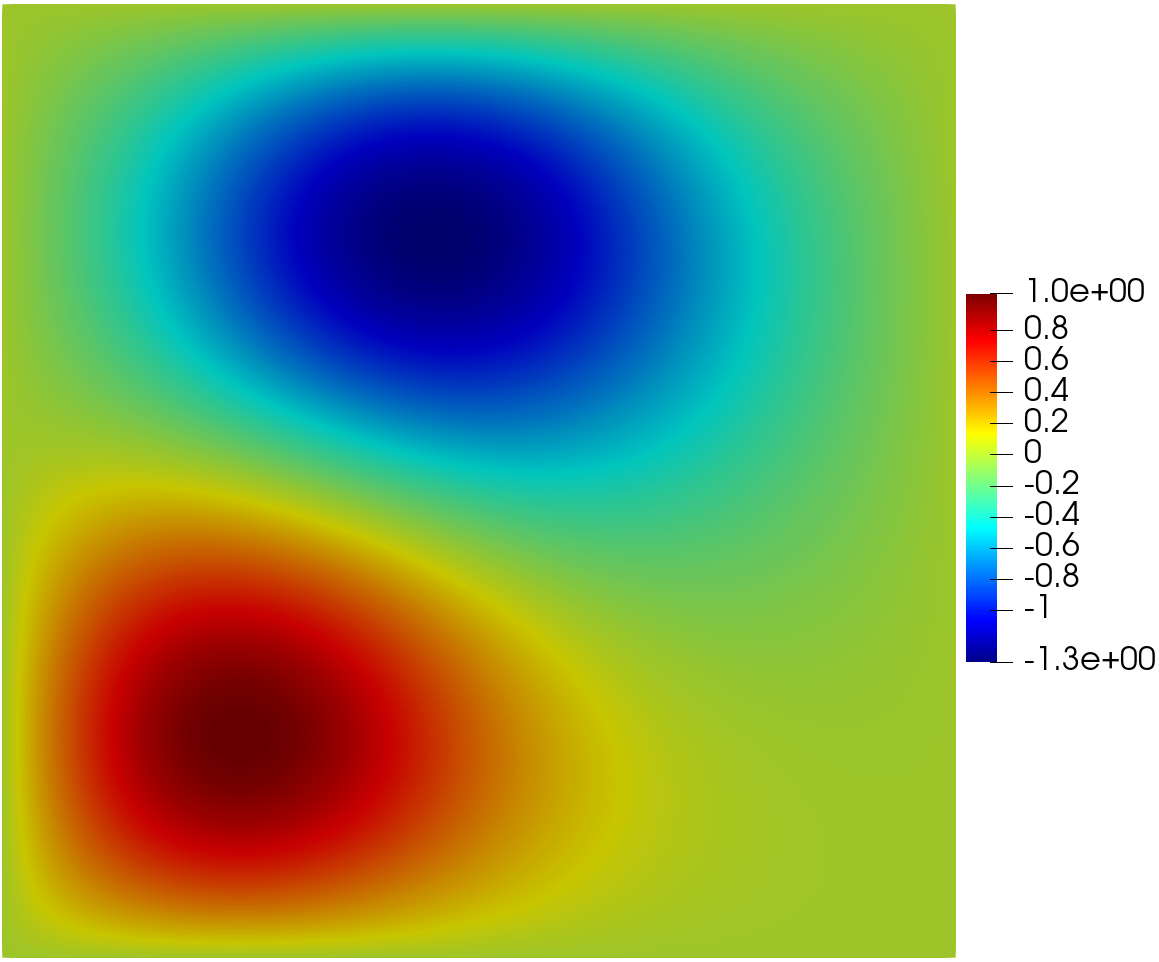}\\
(a) & (b)\\
\includegraphics[width=.35\textwidth]{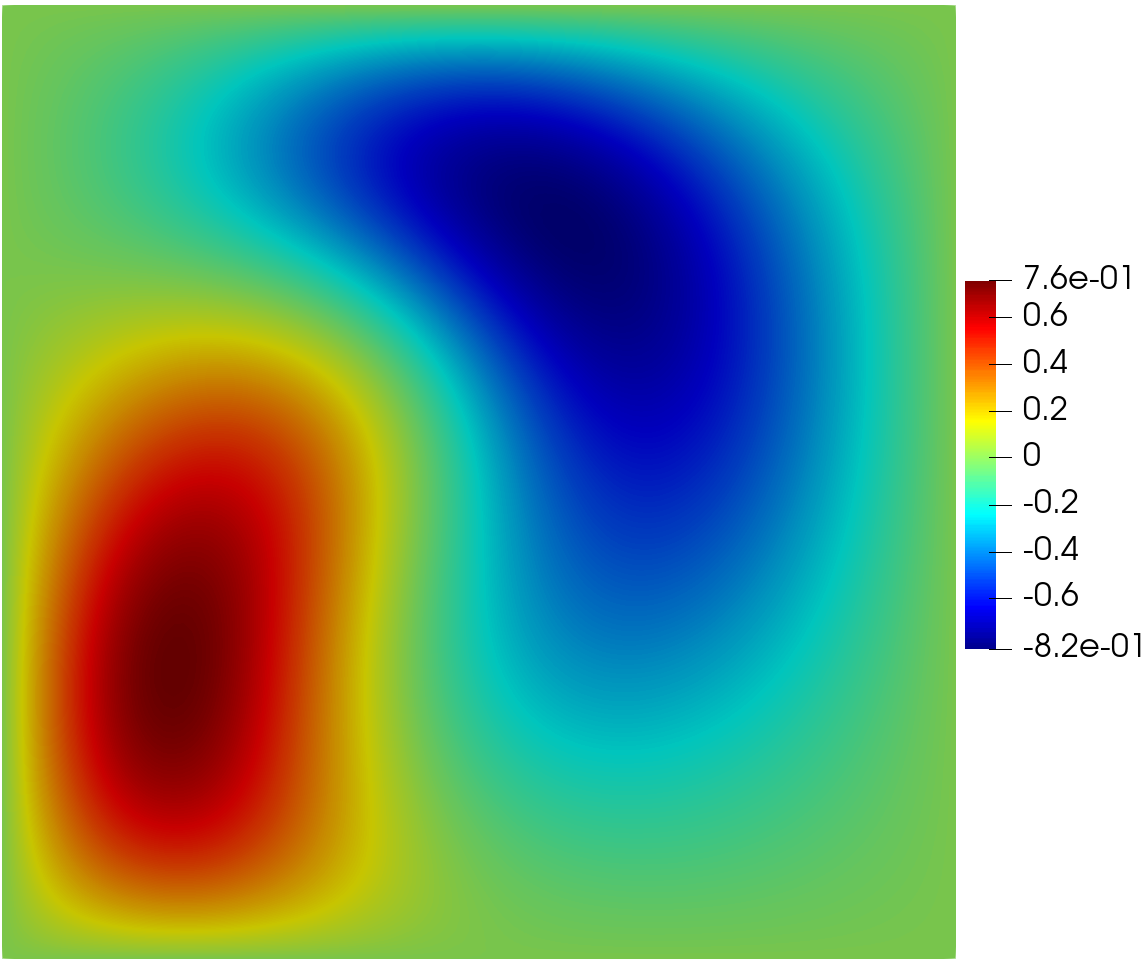}
&\includegraphics[width=.35\textwidth]{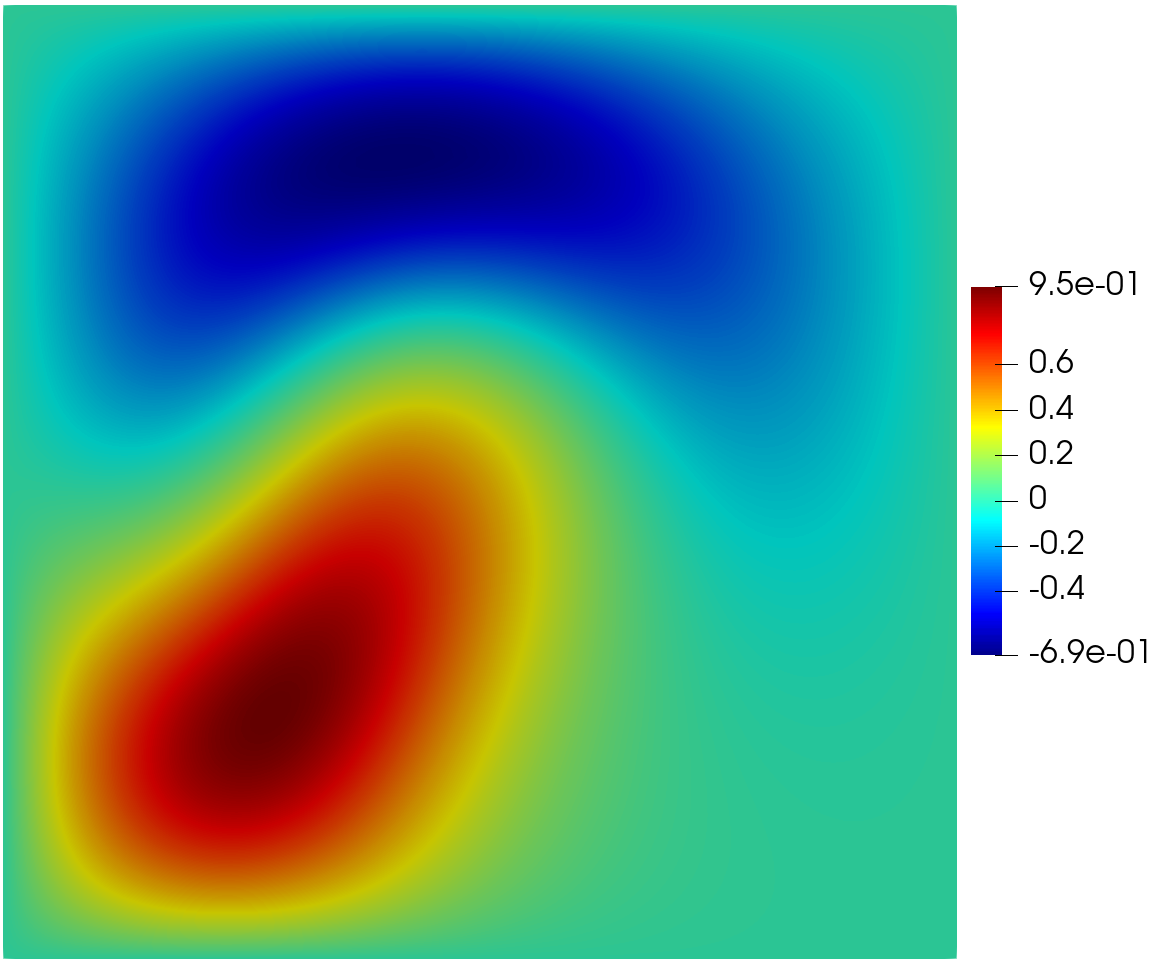}
\\
(c) & (d)\\
\end{tabular}
\caption{Example~\ref{DivFree-Test4}: Plots of optimal $T_h$  for $\kappa = 1.0$ of (a). Initial heat distribution $T_h^0$; and (b). $\gamma = 5$E-5; (c). $\gamma = 1$E-5; (d). $\gamma = 6.9$E-6 .}\label{fig:divfree-4-T}
\end{figure}

{The initial heat distribution corresponding to $\gamma = 1.0$ and $\bv=0$ is plotted in Fig.~\ref{fig:divfree-4-T}a.
As shown in this figure, the maximum and minimum values of of $T_h^0$ are $1.0$ and $-1.4$, respectively.
 The numerical optimal solutions for heat distribution $T_h$ are plotted in Fig.~\ref{fig:divfree-4-T}b-d for $\gamma =$5E-5, 1E-5, and 6.9E-6. 
We  observe that the upper and lower bounds of the initial temperate are reduced from   $T_{\min}=-1.4$  and $T_\text{max} = 1$  (shown in Fig.~\ref{fig:divfree-4-Conv}a) to $(\min T_h = -1.3,\max T_h = 1.0)$, $(\min T_h = -0.82,\max T_h = 0.76)$, and $(\min T_h = -0.69,\max T_h=0.95)$ with respective to $\gamma$ =5E-5, 1E-5, and 6.9E-6. 
Different to former examples, it is shown in Fig.~\ref{fig:divfree-4-V}-\ref{fig:divfree-4-V-Stream} that the velocity profiles differ significantly for these three values of $\gamma$. When $\gamma = $5E-5, as we can see in  Figs.~\ref{fig:divfree-4-V-Stream}-\ref{fig:divfree-4-V}a, the velocity field seems to steer the cold region toward the hot region and thus the minimum value is increased from $-1.4$ to $-1.3$,  however the maximum value remains at $1$. When $\gamma$ = 1E-5, as shown in Figs.~\ref{fig:divfree-4-V-Stream}-\ref{fig:divfree-4-V}a,  it seems that   the cold and the hot  regions are advected simultaneously,  and hence both the maximum and minimum values are tuned. 
However,  as one further reduces the value in $\gamma$ from 5E-5 to 6.9E-6,  the circulation  between the cold and hot regions  becomes disproportional, which results in a smaller  minimum value of the temperature but a higher maximum compared to the case with $\gamma$=5E-5. This may be due to the disproportional steering effect of the velocity field shown in  Figs.~\ref{fig:divfree-4-V}-\ref{fig:divfree-4-V-Stream}.}

\begin{figure}[H]
\centering
\begin{tabular}{ccc}
\includegraphics[width=.3\textwidth]{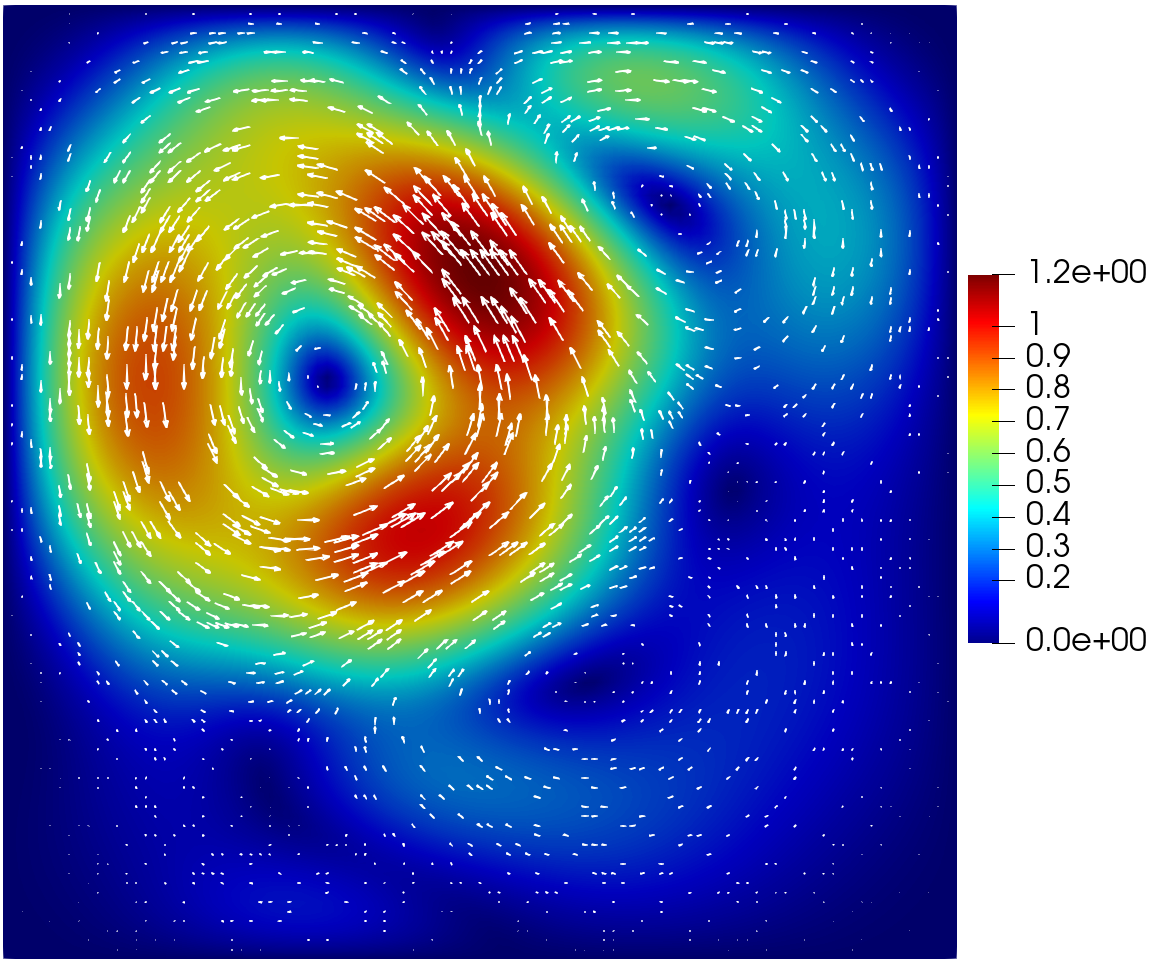}
&\includegraphics[width=.3\textwidth]{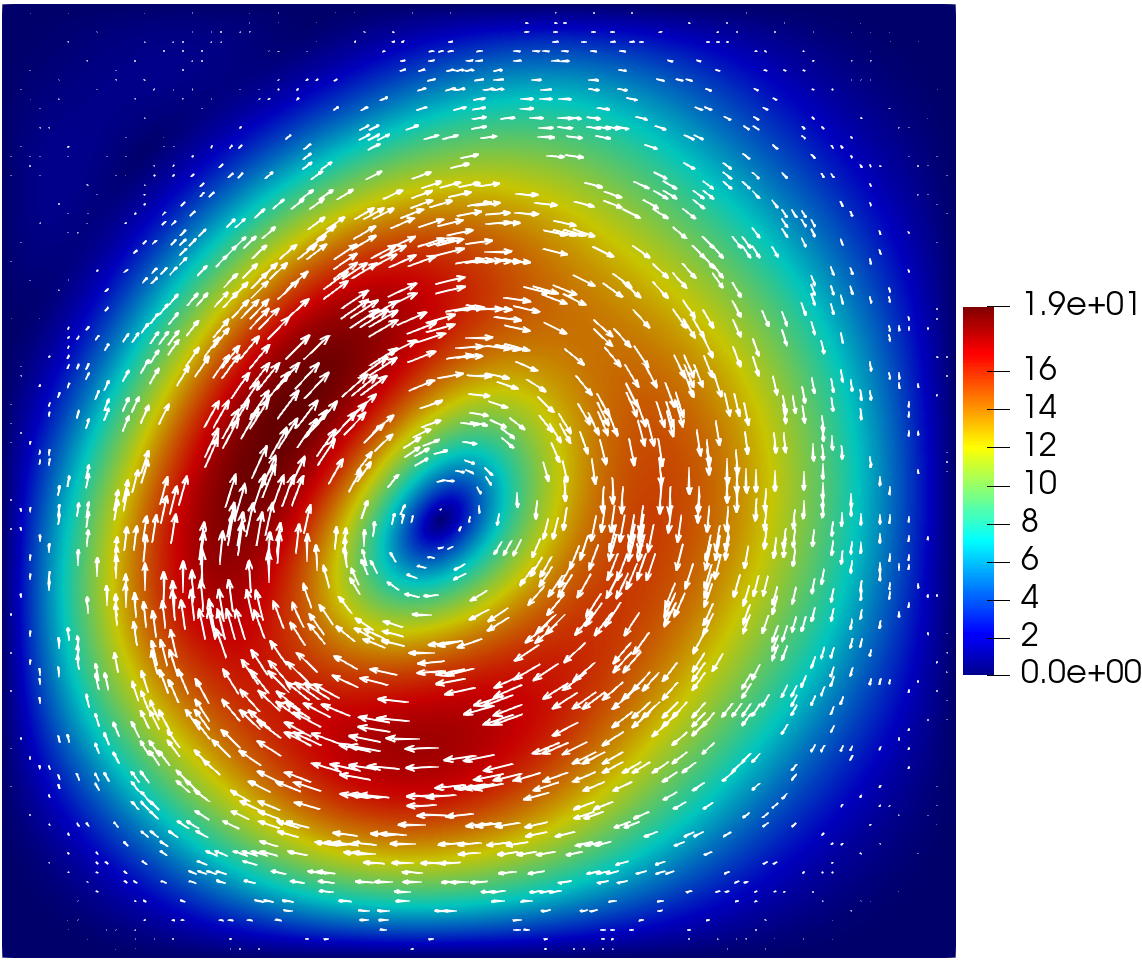}
&\includegraphics[width=.3\textwidth]{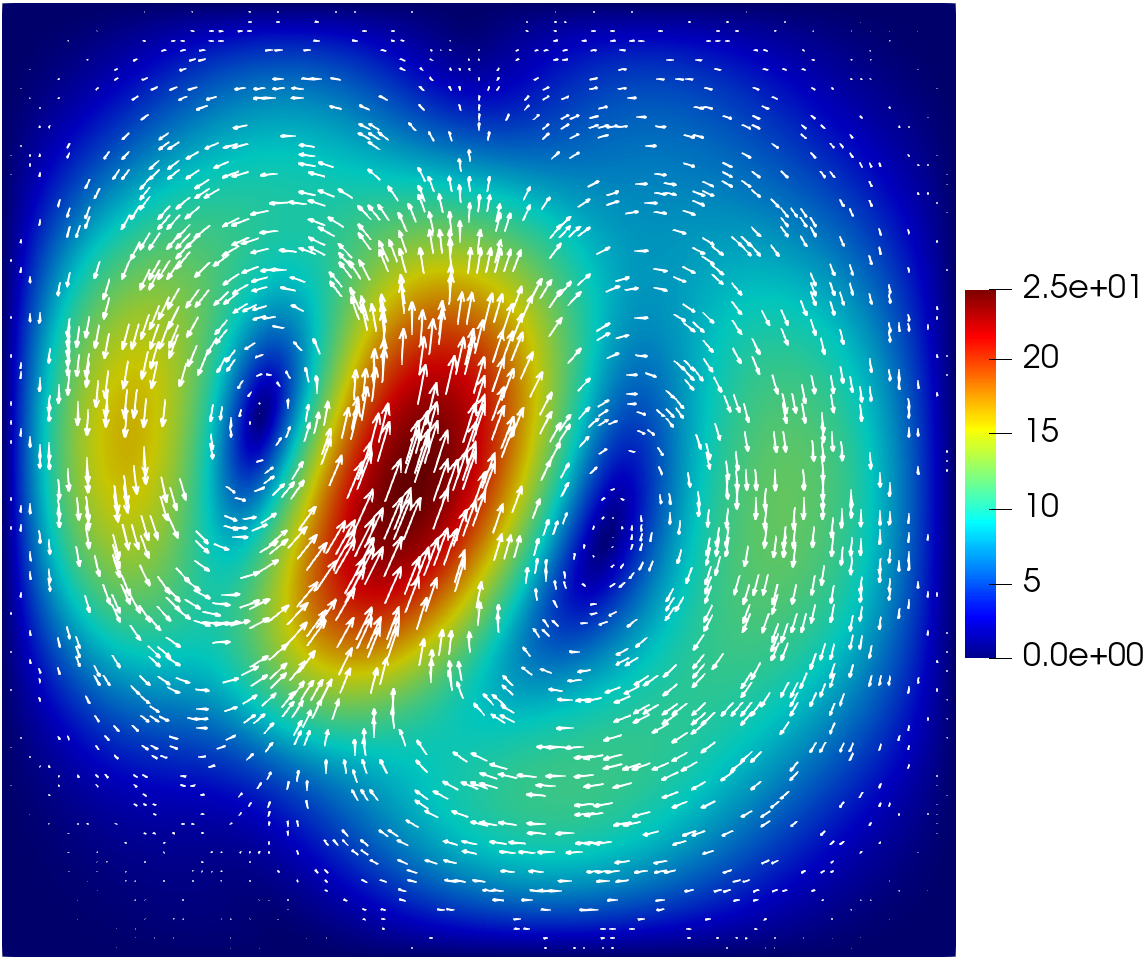}
\\
(a) & (b) &(c)
\end{tabular}
\caption{Example~\ref{DivFree-Test4}: Plots of optimal $\bv_h$  for $\kappa = 1.0$ and (a). $\gamma = 5$E-5; (b). $\gamma = 1$E-5; (c). $\gamma = 6.9$E-6. Here, the color illustrates the magnitude of velocity $\bv_h$ and the vector plots the field of $\bv_h$.}\label{fig:divfree-4-V}
\end{figure}

\begin{figure}[H]
\centering
\begin{tabular}{ccc}
\includegraphics[width=.3\textwidth]{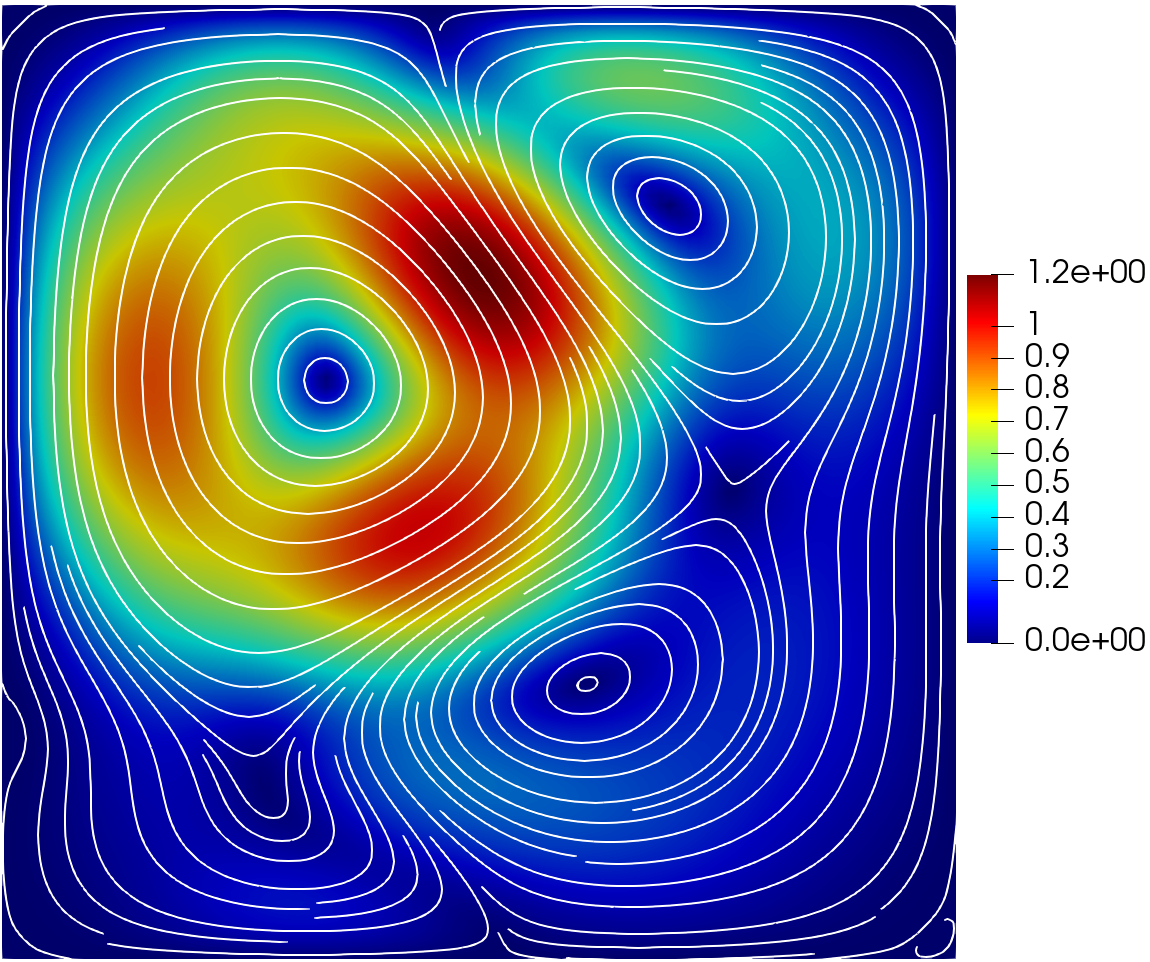}
&\includegraphics[width=.3\textwidth]{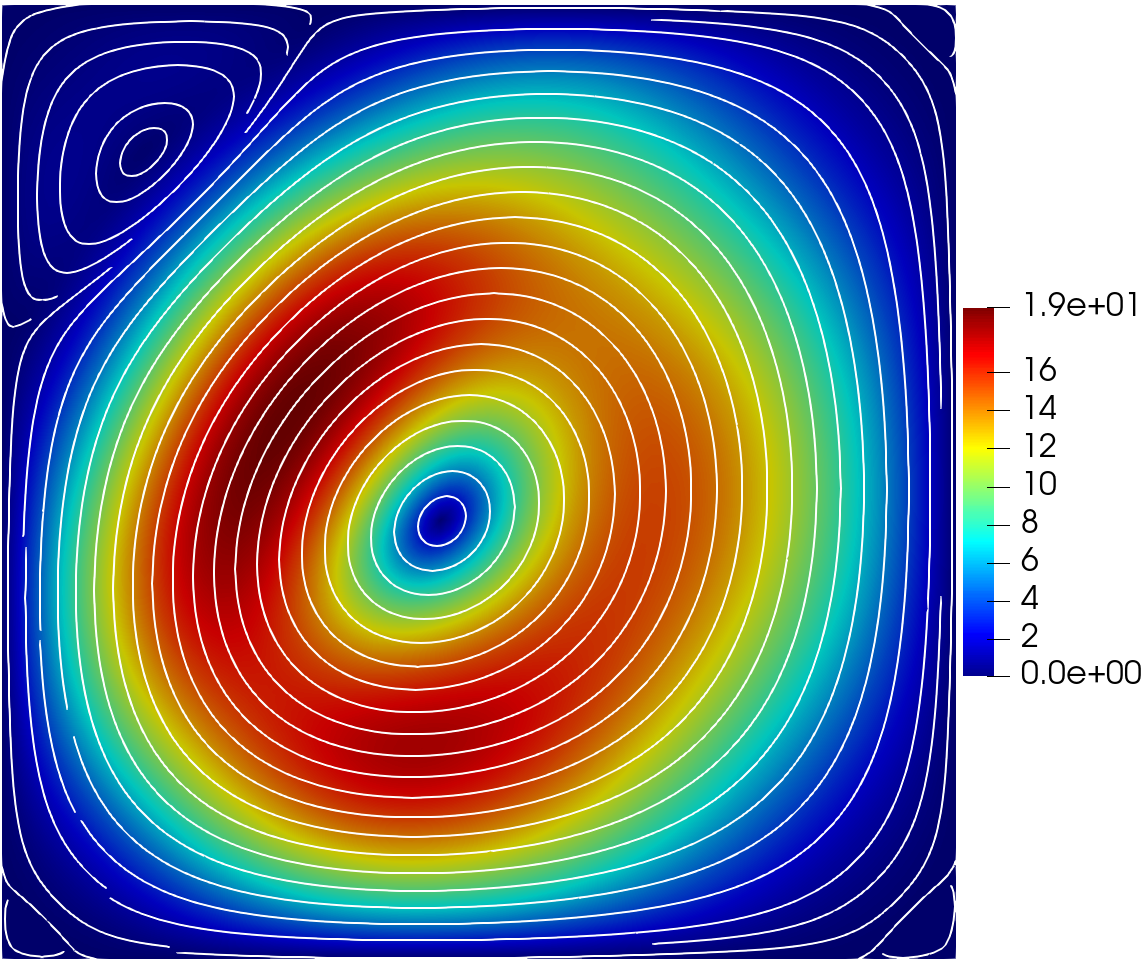}
&\includegraphics[width=.3\textwidth]{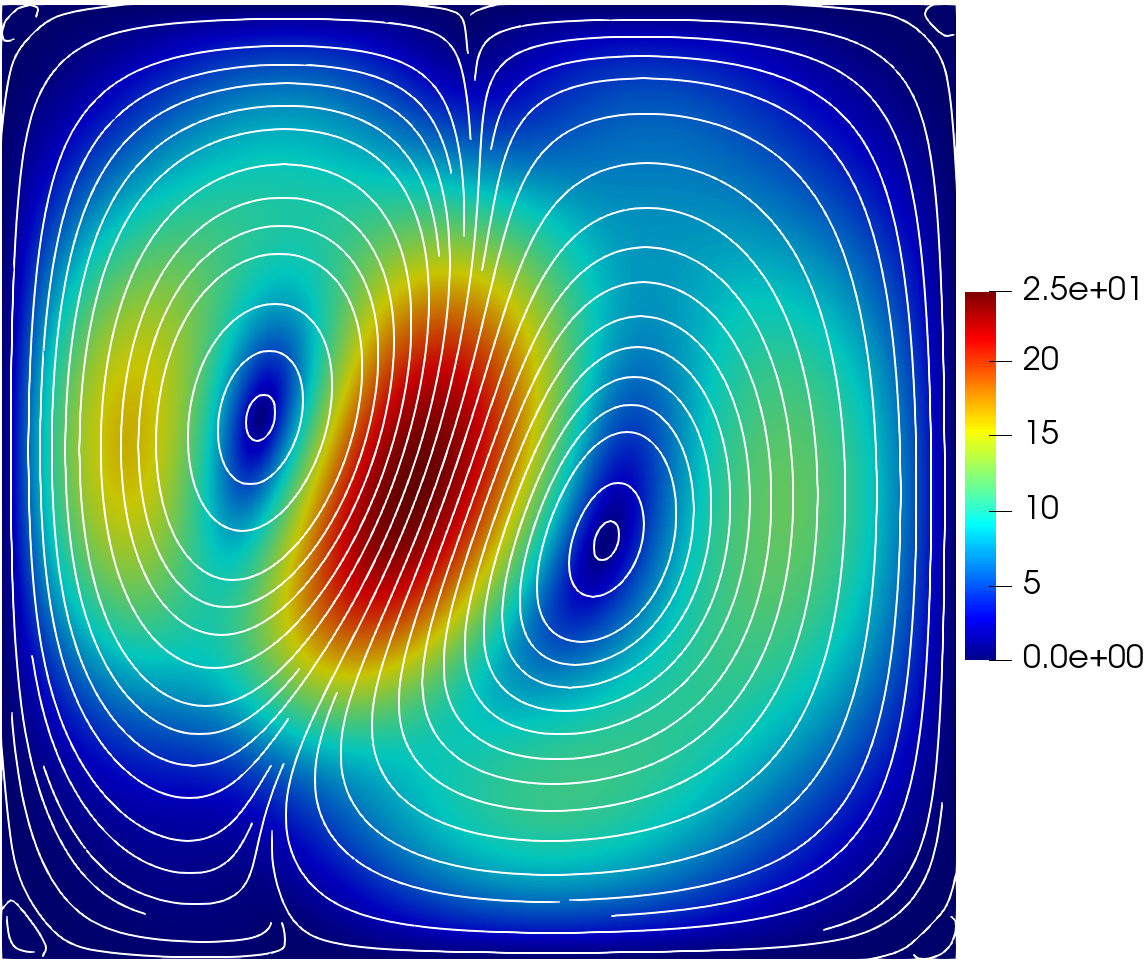}
\\
(a) & (b) &(c)
\end{tabular}
\caption{Example~\ref{DivFree-Test4}: Plots of optimal $\bv_h$  for $\kappa = 1.0$ and (a). $\gamma = 5$E-5; (b). $\gamma = 1$E-5; (c). $\gamma = 6.9$E-6. Here, the color illustrates the magnitude of velocity $\bv_h$ and the curve plots the streamline of $\bv_h$.}\label{fig:divfree-4-V-Stream}
\end{figure}

Lastly, the convergence results are plotted in Fig.~\ref{fig:divfree-4-Conv}. Similar results as in the previous tests can be observed from these two figures. For $\gamma=$6.9E-6,  the cost function $J_{\min} = $9.17E-2, which is $29\%$ smaller than the initial value (1.29E-1). 
In this case, we observe that the convergence rate $r_J$ gradually decreases from $0.31$ to almost $0$.

In summary, we have   conducted a wide range of tests  with  differential values of $\gamma$ for  different heat source distributions in this section.  The numerical results  demonstrate that  using   the optimal convection strategy,  the cost functional value can be reduced by 25\%-40\% depending upon the source terms,   when $\gamma\in$ [E-5, E-7]. 

\begin{figure}[H]
\centering
\begin{tabular}{cc}
\includegraphics[width=.45\textwidth]{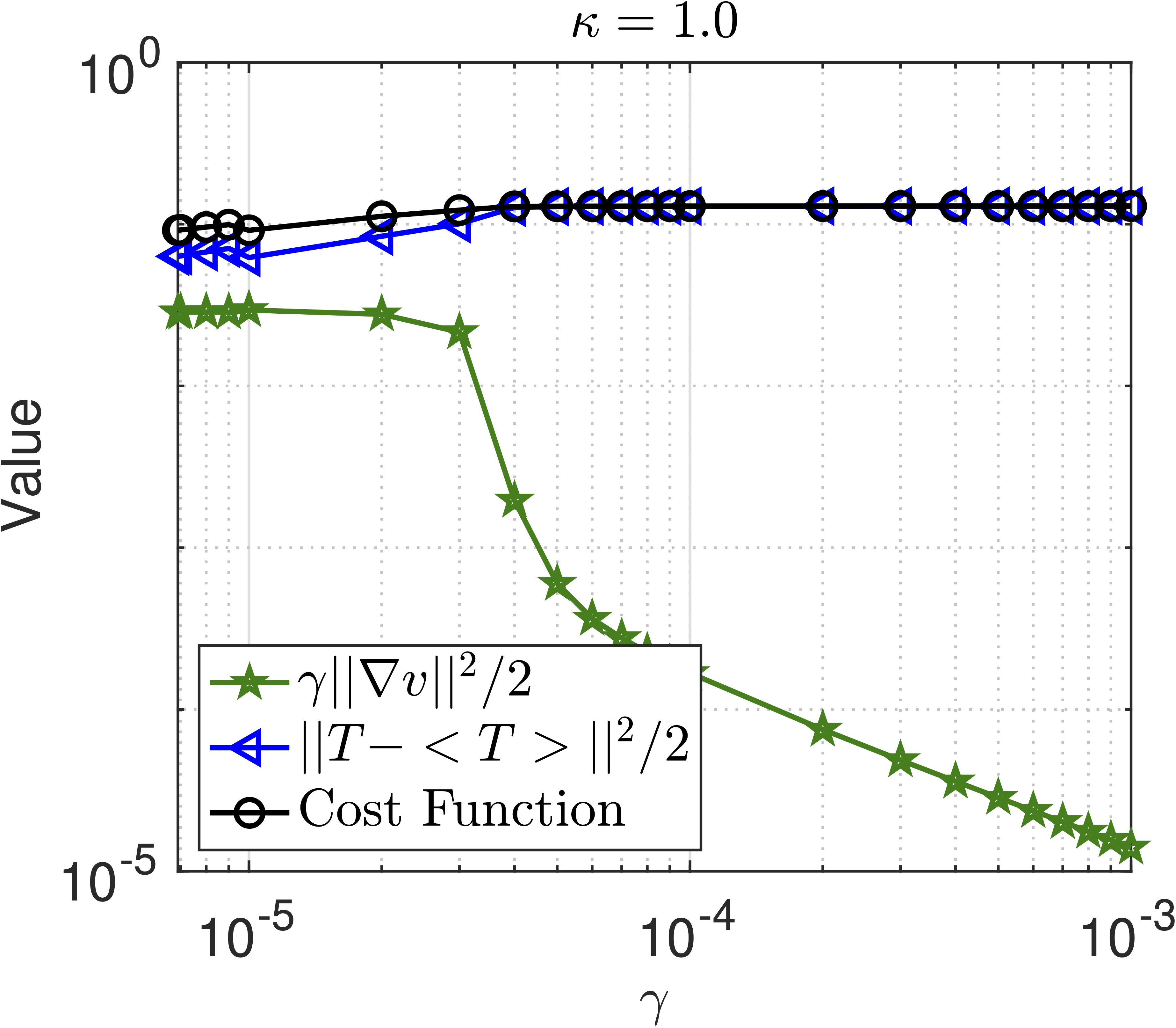}
&\includegraphics[width=.45\textwidth,height = .4\textwidth]{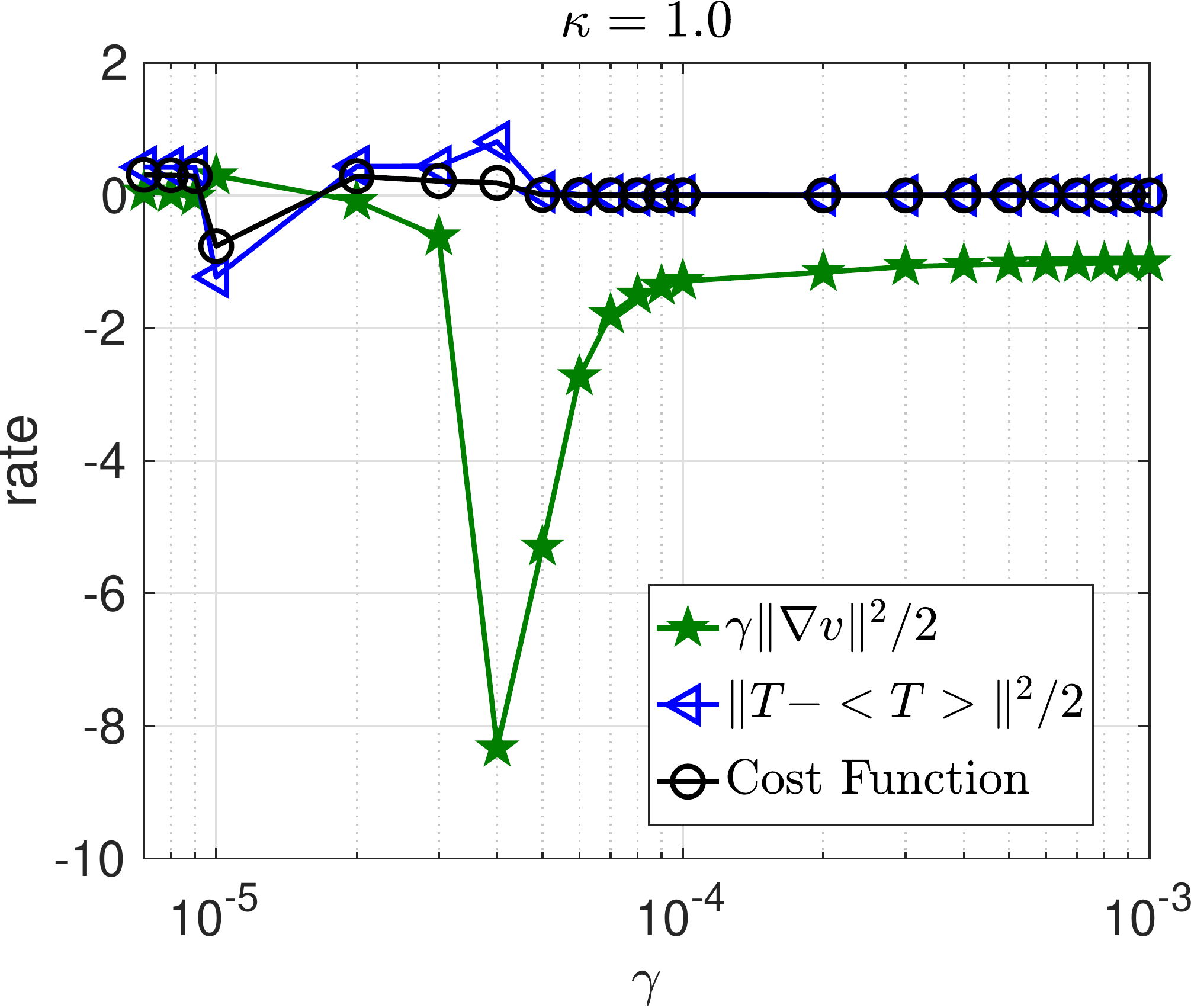}
\\
(a) & (b) 
\end{tabular}
\caption{Example~\ref{DivFree-Test4}: Illustration of results for $\kappa= 1.0$ 
(a). Plot of profiles in the cost functional with respect to $\gamma$ (here $\|T_h^0-\langle T_h^0\rangle\|^2/2 = 1.29$E-1); 
(b). Convergence  rates $r_J, r_T$ and $r_{\bv}$ computed by \eqref{r_J}--\eqref{r_v}}.\label{fig:divfree-4-Conv}
\end{figure}

\section{Conclusion}\label{Sect:Con}
{In this paper, we discussed   the optimal control design for convection-cooling via an incompressible velocity field. We presented  rigorous theoretical    analysis and  conditions  for solving and characterizing   the optimal controller.  Our numerical experiments  demonstrate  the effectiveness of  the cooling process through flow advection. Moreover, we observed that  to enhance heat transfer, small  values in $\gamma$ may be employed in  the convection-cooling design.  
 We shall continue to address the  convergence  issues of our current numerical schemes applied to such nonlinear optimality systems.  We shall also   extend our results to  study the non-stationary convection-cooling problems  for  more physical systems. Specifically, we shall consider to incorporate   the  flow dynamics into the velocity field, which will be controlled in real-time.  How to construct effective numerical schemes to tackle  such problems will be further investigated in our future work.  }

\section{Acknowledgments}
The authors sincerely thank the  anonymous referees for their valuable comments
and constructive suggestions. W. Hu was partially supported by the  NSF grant DMS-1813570.


\end{document}